\documentclass{lms} 
\usepackage[shortlabels]{enumitem}
\usepackage{pb-diagram}
\usepackage{amsmath, amsfonts}
\usepackage{amssymb}
\usepackage{amscd}
\usepackage{stackrel,pifont}
\usepackage[all]{xy}

\newnumbered{Df}{Definition}[section]
\newtheorem{Te}[Df]{Theorem}
\newtheorem{Po}[Df]{Proposition}
\newtheorem{Cr}[Df]{Corollary}
\newtheorem{Lm}[Df]{Lemma}
\newnumbered{Cn}[Df]{Conjecture}
\newnumbered{Ex}[Df]{Example}
\newnumbered{Rm}[Df]{Remark}

\newcommand{\set}[1]{\left\{ #1 \right\}}

\newcommand{\ov}[1]{\mkern 2mu\overline{\mkern-2mu#1\mkern1mu}\mkern -1mu}
\newcommand{\np}{{\ov\nu}}
\newcommand{\nk}{\nu}
\newcommand{\ff}{\ensuremath{\mathbb{F}}}
\newcommand{\cc}{\ensuremath{\mathbb{C}}}
\newcommand{\nn}{\ensuremath{\mathbb{N}}}

\newcommand{\ppq}{\leqslant}
\newcommand{\pgq}{\geqslant}

\newcommand{\mo}{\ensuremath{\mathfrak{s}}}
\newcommand{\mt}{\ensuremath{\mathfrak{t}}}
\newcommand{\me}{\ensuremath{\mathfrak{e}}}
\newcommand{\B}{\ensuremath{\mathcal{B}}}
\newcommand{\tp}{\mathop{\rm top}\nolimits}
\newcommand{\soc}{\mathop{\rm soc}\nolimits}
\newcommand{\id}{\mathop{\rm id}\nolimits}
\newcommand{\im}{\mathop{\rm Im}\nolimits}
\renewcommand{\ker}{\mathop{\rm Ker}\nolimits}
\newcommand{\car}{\mathop{\rm char}\nolimits}
\newcommand{\ot}{\otimes}
\newcommand{\otk}{\otimes_k}

\newcommand{\modb}{\text{\textsf{\upshape Mod}-}\ensuremath{B}}
\newcommand{\abimod}{\ensuremath{A}\text{-\textsf{\upshape Bimod}}}
\newcommand{\catvs}{\text{\textsf{\upshape Vect}\ensuremath{_\ff}}}

\newcommand{\Ext}{\mathop{\rm Ext}\nolimits}
\newcommand{\Tor}{\mathop{\rm Tor}\nolimits}
\newcommand{\Hom}{\mathop{\rm Hom}\nolimits}

\newcommand{\RHom}{\mathop{\rm RHom}\nolimits}
\newcommand{\HH}{\mathop{\rm HH}\nolimits}
\newcommand{\HK}{\mathop{\rm HK}\nolimits}
\newcommand{\cupk}{\stackrel[K]{}{\smile}}
\newcommand{\capk}{\stackrel[K]{}{\frown}}

\newcommand{\ma}{\ensuremath{{m_A}}}
\newcommand{\md}{\ensuremath{{m_D}}}
\newcommand{\pe}[1]{\left\lfloor {#1} \right\rfloor}
\newcommand{\ch}{\ensuremath{\check{h}}}
\newcommand{\ct}{\ensuremath{\check{\zeta}}}

\newcommand{\cp}{\ensuremath{\check{\pi}}}
\newcommand{\cz}{\ensuremath{\check{z}}}
\newcommand{\spn}[1]{\mathop{\rm span}\set{#1}}
\newcommand{\tr}{\mathop{\rm tr}\nolimits}
\renewcommand{\aa}{\kappa}

\usepackage[table]{xcolor}
\usepackage[colorlinks=true,linkcolor=blue,citecolor=olive]{hyperref}

\allowdisplaybreaks

\title{Koszul calculus of preprojective algebras}%
\author{Roland Berger and Rachel Taillefer}

\classno{16S37, 16E40, 16E35}
\extraline{The second author thanks the project CAP2025 of the Universit\'e Clermont Auvergne for its support.}

\begin{document}

\maketitle

\begin{abstract}
We show that the Koszul calculus of a preprojective algebra, whose graph is distinct from A$_1$ and
A$_2$, vanishes in any (co)homological degree $p>2$. Moreover, its (higher) cohomological calculus
is isomorphic as a bimodule to its (higher) homological calculus, by exchanging degrees $p$ and
$2-p$, and we prove a generalised version of the 2-Calabi-Yau property. For the ADE Dynkin graphs,
the preprojective algebras are not Koszul and they are not Calabi-Yau in the sense of Ginzburg's
definition, but they satisfy our generalised Calabi-Yau property and we say that they are Koszul
complex Calabi-Yau (Kc-Calabi-Yau) of dimension $2$. For Kc-Calabi-Yau (quadratic) algebras of any dimension, defined in terms of derived categories, we prove a Poincar\'e Van den Bergh duality theorem. We compute explicitly the Koszul calculus of preprojective algebras for the ADE Dynkin graphs.
\end{abstract} 

\section{Introduction} \label{intro}  

Preprojective algebras are quiver algebras with quadratic relations, that play an important role in the
representation theory of quiver algebras~\cite{gp:preproj,dr:preproj,bgl:preproj,bbk}, with various applications~\cite{cbh:klein,cbvdb:kac} and many developments~\cite{cbeg:quiver,gls:ppcluster,bes:typeL,griy:hpa}. In~\cite{griy:hpa}, the reader will find an introduction to the various aspects of the preprojective algebras in representation theory, with an extended bibliography. In our paper, we are interested in some homological properties linked to Hochschild cohomology. 

In the last two decades, the Hochschild cohomology of preprojective algebras, as well as some extra algebraic structures, have been computed in several steps, as follows.

1) Erdmann and Snashall~\cite{es:first,es:second} determined the Hochschild cohomology and its
cup-product in type A.

2) Crawley-Boevey, Etingof and Ginzburg~\cite{cbeg:quiver} determined the Hochschild cohomology for
all preprojective algebras of non-Dynkin type (which are Koszul in this case~\cite{mv:kprepro,green:intro,bbk}).

3) In type DE and characteristic zero, Etingof and Eu~\cite{eteu:ade} determined the Hochschild
cohomology and  Eu~\cite{eu:product} the cup product. The cyclic homology was computed in type ADE in~\cite{eteu:ade}.

4) Assembling and completing the previous results in characteristic zero, Eu gave an explicit
description of the Tamarkin-Tsygan calculus~\cite{tt:calculus} of the preprojective algebras in type
ADE, that is, the homology and the cohomology, the cup product, the contraction map and the Lie derivative, the Connes differential and the Gerstenhaber bracket~\cite{eu:calculus}.

5) Eu and Schedler extended the ADE results to the case where the base ring is $\mathbb{Z}$, and obtained the corresponding ADE results in any characteristic~\cite{eusched:cyfrob}.

In~\cite{bls:kocal}, a Koszul calculus was associated with any quadratic algebra over a field, in
order to produce new homological invariants for non-Koszul quadratic algebras. We begin this paper
by extending the Koszul calculus to quadratic quiver algebras. We shall compute the Koszul calculus
of the preprojective algebras whose graphs are Dynkin of type ADE (the preprojective algebras are then finite dimensional). Except for types A$_1$ and A$_2$, these quadratic quiver algebras are not Koszul~\cite{mv:kprepro,green:intro,bbk}, so that the Koszul calculus and the Hochschild calculus provide different information.

Before presenting our computations, we state and develop a Poincar\'e Van den Bergh duality theorem~\cite{vdb:dual} for the Koszul homology/cohomology of \emph{any preprojective algebra} whose graph is different from A$_1$ and A$_2$. This theorem is formulated as follows and constitutes the first main result of the present paper. The duality is precisely part \textit{(ii)} in this theorem.

\begin{Te} \label{funda1}
Let $A$ be the preprojective algebra of a (non-labelled) connected graph $\Delta$ distinct from $\mathrm{A}_1$ and $\mathrm{A}_2$, over a field $\mathbb{F}$. Let $M$ be an $A$-bimodule.
\begin{enumerate}[\itshape(i)]
\item The Koszul bimodule complex $K(A)$ of $A$ has length 2. In particular,
  $\HK^p(A,M)= \HK_p(A,M)= 0$ for any $p>2$.

\item  The $\HK^{\bullet}(A)$-bimodules $\HK^{\bullet}(A,M)$ and $\HK_{2-\bullet}(A,M)$ are
  isomorphic.

\item The $\HK^{\bullet}_{hi}(A)$-bimodules $\HK^{\bullet}_{hi}(A,M)$ and
  $\HK^{hi}_{2-\bullet}(A,M)$ are isomorphic.
\end{enumerate}

\end{Te}

In this statement, following~\cite{bls:kocal}, $\HK_p(A,M)$ and $\HK^p(A,M)$ denote the Koszul homology and cohomology spaces with coefficients in $M$, while $\HK^{hi}_p(A,M)$ and $\HK^p_{hi}(A,M)$ denote the higher Koszul homology and cohomology spaces. When $M=A$, these notations are simplified into $\HK_p(A)$, $\HK^p(A)$, $\HK^{hi}_p(A)$ and $\HK^p_{hi}(A)$.

In the general setting~\cite{bls:kocal}, the Koszul calculus of a quadratic algebra $A$ consists of the graded associative algebra $\HK^{\bullet}(A)$ endowed with the Koszul cup product and, for all $A$-bimodules $M$, of the graded $\HK^{\bullet}(A)$-bimodules $\HK^{\bullet}(A,M)$ and $\HK_{\bullet}(A,M)$, with actions respectively defined by the Koszul cup and cap products. The higher Koszul calculus of $A$ is given by the analogous data, adding the subscript and superscript \emph{hi}. Sometimes (as will be the case with our computations in ADE types), these calculi are \emph{restricted}, meaning that the data is limited to $M=A$, so that the restricted Koszul calculus consists of the graded associative algebra $\HK^{\bullet}(A)$ and of the graded $\HK^{\bullet}(A)$-bimodule $\HK_{\bullet}(A)$ -- similarly for the higher version.

Using Theorem \ref{funda1} for $\Delta$ Dynkin of type ADE, we shall deduce the (higher) homological restricted Koszul calculus from the computation of the (higher) cohomological restricted Koszul calculus.  


 Part \textit{(ii)} in Theorem \ref{funda1} comes from an explicit isomorphism from the complex $C_1$ of Koszul cochains with coefficients in $M$, whose $p$th cohomology is $\HK^p(A,M)$,  to the complex $C_2$ of Koszul chains  with coefficients in $M$, whose $p$th homology is  $\HK_{2-p}(A,M)$, described as follows. 

\begin{Po} \label{funda2}
Let $A$ be the preprojective algebra of a connected graph $\Delta$ distinct from $\mathrm{A}_1$ and $\mathrm{A}_2$. Let $M$ be an $A$-bimodule. The Koszul cup and cap products are denoted by $\underset{K}{\smile}$ and $\underset{K}{\frown}$. Define $\omega_0= \sum_i e_i \otimes \sigma_i$, where the sum runs over the vertices $i$ of $\Delta$ and, for each vertex $i$,  $e_i$ is the idempotent and $\sigma_i$ is the quadratic relation in $A$ associated with $i$.

For each Koszul $p$-cochain $f$ with coefficients in $M$, we define the Koszul $(2-p)$-chain $\theta_M (f)$ with coefficients in $M$ by
$$\theta_M (f) = \omega_0 \underset{K}{\frown} f .$$
Then $\theta_M: C_1 \rightarrow C_2$ is an isomorphism of complexes. Moreover, the equalities
$$\theta_{M\otimes_A N} (f\underset{K}{\smile} g) = \theta_M (f) \underset{K}{\frown} g = f \underset{K}{\frown} \theta_N (g) $$
hold for any Koszul cochains $f$ and $g$ with coefficients in bimodules $M$ and $N$ respectively.
\end{Po}

The proof of Proposition \ref{funda2} relies on some manipulations of the defining formula of $\theta_M$ with fundamental formulas of Koszul calculus~\cite{bls:kocal}, using actions involving $\underset{K}{\smile}$ and $\underset{K}{\frown}$. The fundamental formulas of Koszul calculus express the differential $b_K$ of $C_1$ and the differential $b^K$ of $C_2$ respectively as a cup bracket and a cap bracket, namely
$$b_K=- [\me_A, -]_{\underset{K}{\smile}}, \ \ b^K = -[\me_A, -]_{\underset{K}{\frown}},$$
where $\me_A:V\rightarrow A$ is a fundamental Koszul 1-cocycle defined on the arrow space $V$ by $\me_A(x)=x$ for all $x\in
V$.

In order to extract a generalised version of the 2-Calabi-Yau property from our Poincar\'e Van den
Bergh duality (Theorem \ref{funda1}) for quadratic algebras, we apply this theorem to the left $A^e$-module $M=A^e:=A\otimes A^{op}$ viewed as an $A$-bimodule. We show that the complex of Koszul chains with coefficients in the left $A^e$-module $A^e$ is naturally isomorphic, as a right $A^e$-module, to the Koszul bimodule complex $K(A)$. Using the fact that the homology of $K(A)$ is isomorphic to $A$ in degree 0, and to $0$ in degree 1, we obtain a generalisation of the 2-Calabi-Yau property, formulated as follows. 

\begin{Te} \label{funda4}
Let $A$ be the preprojective algebra of a connected graph $\Delta$ distinct from $\mathrm{A}_1$ and $\mathrm{A}_2$, over a field $\mathbb{F}$. Let us denote by $K(A)$ the Koszul bimodule complex of $A$. Then the $A$-bimodule $\HK^p(A,A^e)$ is isomorphic to the $A$-bimodule $H_{2-p}(K(A))$ for $0\ppq p \ppq 2$. In particular, we have the following.
\begin{enumerate}[\itshape(i)]
\item  The $A$-bimodule $\HK^2(A, A^e)$ is isomorphic to the $A$-bimodule $A$.

\item  $\HK^1(A,A^e)= 0$.

\item  The $A$-bimodule $\HK^0(A,A^e)$ is isomorphic to the $A$-bimodule $H_2(K(A))$, which is
  always non-zero when $\Delta$ is Dynkin of type $\mathrm{ADE}$.
\end{enumerate}

\end{Te}

  We then say that the preprojective algebra $A$ is a \emph{Koszul complex Calabi-Yau algebra of
    dimension 2}. We generalise this definition to any quadratic algebra and any dimension $n$ in Definition \ref{kcy} below, better formulated in terms of derived categories. Since there is an $\mathbb{F}$-linear isomorphism
$$H(\theta_M)= \overline{\omega}_0 \underset{K}{\frown} - : \HK^{\bullet}(A,M) \rightarrow \HK_{2-\bullet}(A,M),$$
we say that the class $\overline{\omega}_0 \in \HK_2(A)$ is the \emph{fundamental class} of the Koszul complex Calabi-Yau algebra $A$, by analogy with Poincar\'e's duality in singular homology/cohomology~\cite{hat:algtop}. In Definition \ref{strongcy}, we give a stronger version of Definition \ref{kcy} in order to obtain a Poincar\'e-like duality, that is, a duality isomorphism expressed as a cap action by a suitably defined fundamental class. 

Let us remark that the $A$-bimodule structures in Theorem \ref{funda4} are compatible with the Koszul cup and cap actions of $\HK^{\bullet}(A)$ on $\HK^{\bullet}(A, A^e)$ and $H(K(A))$. These actions can be viewed as graded actions of left $\HK^{\bullet}(A)^e$-modules, while the $A$-bimodules can be viewed as compatible right $A^e$-modules. So the isomorphism $\HK^{\bullet}(A,A^e)\cong H_{2-\bullet}(K(A))$ in Theorem \ref{funda4} is an isomorphism of graded $\HK^{\bullet}(A)^e$-$A^e$-bimodules. This enriched isomorphism is the expression of the stronger version of the Koszul complex Calabi-Yau property, as we shall see in Definition \ref{strongcy}.

Note that if $\Delta$ is not Dynkin ADE, then $A$ is Koszul, so Theorem \ref{funda4} enables us to recover the well-known result that $A$ is 2-Calabi-Yau in the sense of Ginzburg~\cite{cbeg:quiver,boc:3cy}.  However, 
if $\Delta$ is Dynkin ADE, then $A$ is not homologically smooth since its minimal projective
resolution has  infinite length, so that Ginzburg's definition of Calabi-Yau algebras cannot be
applied in this case~\cite{vg:cy}. Moreover, the restricted Hochschild calculus is drastically
different from the restricted Koszul calculus, because by~\cite{es:third} there is a cohomological Hochschild periodicity
$$\HH ^{p+6}(A) \cong \HH ^p(A), \ p>0$$ 
and, consequently, there are non-zero spaces $\HH ^p(A)$ for infinitely many values of $p$. Even taking into account this 6-periodicity, the list of cohomological Koszul invariants consists only of  $\HK^0(A)$, $\HK^1(A)$ and $\HK^2(A)$ and is therefore shorter than the list of Hochschild invariants.

In~\cite{eusched:cyfrob}, Eu and Schedler define periodic Calabi-Yau Frobenius algebras, for finite dimensional algebras only. Their main example is given by the preprojective algebras of Dynkin ADE graphs~\cite[Example 2.3.10]{eusched:cyfrob}. Then the above cohomological Hochschild periodicity is a part of remarkable isomorphims in Hochschild calculus for any periodic Calabi-Yau Frobenius algebra~\cite[Theorem 2.3.27 and Theorem 2.3.47]{eusched:cyfrob}.

From Theorem \ref{funda4}, we are led to introduce a general definition.

\begin{Df} \label{kcy}
Let $\mathcal{Q}=(\mathcal{Q}_0, \mathcal{Q}_1)$ be a finite quiver, and let $\mathbb{F}$ be a field. Let $A$ be an $\mathbb{F}$-algebra defined on the path algebra $\mathbb{F}\mathcal{Q}$ of $\mathcal{Q}$ by homogeneous quadratic relations. Define the ring $k=\mathbb{F} \mathcal{Q}_0$, 
so that $A$ is regarded as a quadratic $k$-algebra. We say that $A$ is Koszul complex Calabi-Yau (Kc-Calabi-Yau) of dimension $n$, for an integer $n\pgq 0$, if
\begin{enumerate}[\itshape(i)]
\item  the bimodule Koszul complex $K(A)$ of $A$ has length $n$, and

\item  $\RHom _{A^e}(K(A), A^e) \cong K(A)[-n]$ in the bounded derived category of $A$-bimodules.
\end{enumerate}

\end{Df}

In our context (that of quadratic algebras), Definition \ref{kcy} is a definition of a new Calabi-Yau property, valid whether $A$ is finite dimensional or not. In this definition, we do not impose that $K(A)$ be a resolution of $A$, that is, $A$ is not necessarily Koszul, meaning that the bimodules $\HK^p(A, A^e)$ for $0 \ppq p \ppq n-2$ may be non-zero. Under the assumptions of Definition \ref{kcy}, we verify that, if $A$ is Koszul, Definition \ref{kcy} is equivalent to Ginzburg's definition of $n$-Calabi-Yau algebras~\cite{vg:cy,vdb:cypot}. We then prove a new Poincar\'e Van den Bergh duality for Kc-Calabi-Yau algebras, adapted to Koszul (co)homologies.

\begin{Te} \label{duality}
Let $A$ be a Kc-Calabi-Yau algebra of dimension $n$. Then for any $A$-bimodule $M$, the $\mathbb{F}$-vector spaces $\HK^p(A,M)$ and $\HK_{n-p}(A,M)$ are isomorphic.
\end{Te} 

\begin{Df} \label{fundclass}
Let $A$ be a Kc-Calabi-Yau algebra of dimension $n$. The image $c \in \HK_n(A)$ of the unit $1$ of the algebra $A$ under the isomorphism $\HK^0(A) \cong \HK_n(A)$ in Theorem \ref{duality} is called the fundamental class of the Kc-Calabi-Yau algebra $A$.
\end{Df}

In order to describe the duality isomorphism of Theorem \ref{duality} explicitly as a cap-product by the fundamental class for strong Kc-Calabi-Yau algebras, we shall use derived categories in the general context of DG algebras, as presented and detailed in the preprint book by Yekutieli~\cite{yeku:dercat}. Let us present briefly what we need in this general context.

We introduce the DG algebra $\tilde{A}=\Hom _{A^e}(K(A), A)$. The complexes $K(A)$ and $\Hom _{A^e}(K(A), A^e)$ of $A$-bimodules have an enriched structure since they can be viewed as DG $\tilde{A}$-bimodules in the abelian category $\abimod$ of $A$-bimodules, in the sense of~\cite{yeku:dercat}.

Denote by $\mathcal{C}(\tilde{A},\abimod)$ the category of DG $\tilde{A}$-bimodules in $\abimod$~\cite{yeku:dercat}. Let $M$ be an $A$-bimodule. For any chain DG $\tilde{A}$-bimodule $C$ in $\abimod$, $\Hom _{A^e} (C, M)$ is a cochain DG $\tilde{A}$-bimodule in the abelian category $\catvs$ of $\mathbb{F}$-vector spaces (in $\abimod$ when $M=A^e$). For any cochain DG $\tilde{A}$-bimodule $C'$ in $\abimod$, $M\otimes_{A^e} C'$ is a cochain DG $\tilde{A}$-bimodule in $\catvs$. The bounded derived categories $\mathcal{D}^b(\tilde{A},\abimod)$ and $\mathcal{D}^b(\tilde{A},\catvs)$ are defined in~\cite{yeku:dercat}. However we do not know if the functors $\Hom _{A^e} (-, M)$ and $M\otimes_{A^e} - $ from $\mathcal{C}^b(\tilde{A},\abimod)$ to $\mathcal{C}^b(\tilde{A},\catvs)$ are derivable.

\begin{Df} \label{strongcy}
\sloppy Let $A$ be a Kc-Calabi-Yau algebra of dimension $n$. Then $A$ is said to be strong Kc-Calabi-Yau if the derived functor of the endofunctor $\Hom _{A^e}(-, A^e)$ of  $\mathcal{C}^b(\tilde{A},\abimod)$ exists and if $\RHom  _{A^e}(K(A), A^e) \cong K(A)[-n]$ in the bounded derived category $\mathcal{D}^b(\tilde{A},\abimod)$.
\end{Df}

\begin{Te} \label{strongduality}
Let $A$ be a Kc-Calabi-Yau algebra of dimension $n$ and let $c$ be its fundamental class. We assume that $A$ is strong Kc-Calabi-Yau and that the derived functors of the functors $\Hom _{A^e}(-,A)$ and $A\otimes_{A^e} -$ from $\mathcal{C}^b(\tilde{A},\abimod)$ to $\mathcal{C}^b(\tilde{A},\catvs)$ exist. Then
$$c \underset{K}{\frown} - : \HK^{\bullet}(A) \rightarrow \HK_{n-\bullet}(A)$$
is an isomorphism of $\HK^{\bullet}(A)$-bimodules, inducing an isomorphism of $\HK^{\bullet}_{hi}(A)$-bimodules from $\HK^{\bullet}_{hi}(A)$ to $\HK^{hi}_{n-\bullet}(A)$.  For  all $\alpha \in \HK^p(A)$, we have $c \underset{K}{\frown} \alpha = (-1)^{np}\alpha \underset{K}{\frown} c$.
\end{Te}

Let us describe the contents of the paper. In Section 2, we extend the general formalism  -- including some results -- of Koszul calculus~\cite{bls:kocal} to quadratic quiver algebras. In Section 3, we introduce a right action which is an important tool in order to adapt the definition of Calabi-Yau algebras to quadratic quiver algebras endowed with the Koszul calculus instead of the Hochschild calculus. The Poincar\'e Van den Bergh duality for preprojective algebras is presented in Section 4, where Theorem \ref{funda1}, Proposition \ref{funda2} and Theorem \ref{funda4} of our introduction are proved. In Section 5, we define our generalisations of Calabi-Yau algebras and we thoroughly explain the new objects and remaining results outlined in the introduction. Section 6 is devoted to the computations of the Koszul calculus in ADE Dynkin types. As an application of the computations, we prove that the spaces $\HK^0_{hi}(A)$, $\HK^1_{hi}(A)$ and $\HK^2_{hi}(A)$ form a minimal complete list of cohomological invariants for the ADE preprojective algebras.

\paragraph{Acknowledgement.} The authors are grateful to the anonymous referee, whose useful and detailed comments helped to improve this manuscript. 

\setcounter{equation}{0}

\section{Koszul calculus for quiver algebras with quadratic relations} \label{kcquivers}

\subsection{Setup} \label{not}

Let $\mathcal{Q}$ be a \emph{finite} quiver, meaning that the vertex set $\mathcal{Q}_0$ and the arrow set $\mathcal{Q}_1$ are finite. Let $\mathbb{F}$ be a field. The vertex space $k=\mathbb{F} \mathcal{Q}_0$ becomes a commutative ring by associating with ${\mathcal Q}_0$ a complete set of orthogonal idempotents $\set{e_i\,;\,i\in{\mathcal Q}_0}$. The ring $k$ is isomorphic to $\mathbb{F}^{|\mathcal{Q}_0|}$, where $|\mathcal{Q}_0|$ is the cardinal of $\mathcal{Q}_0$. Throughout the paper, the case $|\mathcal{Q}_0|=1$ will be called \emph{the one vertex case}, which is equivalent to saying that $k$ is a field. Koszul calculus over a field $k$ is treated in~\cite{bls:kocal}. 

For each arrow $\alpha \in \mathcal{Q}_1$, denote its source vertex by $\mo(\alpha)$ and its target vertex by $\frak{t}(\alpha)$. The arrow space $V=\mathbb{F} \mathcal{Q}_1$ is a $k$-bimodule for the following actions: $e_j\alpha e_i$ is equal to zero if $i\neq \frak{s}(\alpha)$ or $j\neq \frak{t}(\alpha)$, and is equal to $\alpha$ if $i= \frak{s}(\alpha)$ and $j= \frak{t}(\alpha)$. 

Via the ring morphism $\mathbb{F} \rightarrow k$ that maps $1$ to $\sum_{i\in{\mathcal Q}_0}e_i$, the tensor $k$-algebra $T_k(V)$ of the $k$-bimodule $V$ is an $\mathbb{F}$-algebra isomorphic to the path algebra $\mathbb{F}\mathcal{Q}$, so that $V^{\otimes_k m}$ is identified with $\mathbb{F}\mathcal{Q}_m$, where $\mathcal{Q}_m$ is the set of paths of length $m$. For two arrows $\alpha$ and $\beta$, note that 
$\alpha \otimes_k \beta$ is zero if $\frak{t}(\beta)\neq \frak{s}(\alpha)$, and otherwise $\alpha \otimes_k \beta$ is identified with the composition $\alpha \beta$ of paths (where paths are written from right to left, as in~\cite{benson:repcoh}).

Let $R$ be a sub-$k$-bimodule of $V\otimes_kV\cong \mathbb{F}\mathcal{Q}_2$. The unital associative
$k$-algebra $A=T_k(V)/(R)$, where $(R)$ denotes the two-sided ideal of $T_k(V)$ generated by $R$, is
called a \emph{quadratic $k$-algebra over the finite quiver $\mathcal{Q}$}\label{quadratic over Q}. The degree induced on $A$ by the path length is called the \emph{weight}, so that $A$ is a graded algebra for the weight grading. The component of weight $m$ of $A$ is denoted by $A_m$. Clearly, $A_0\cong k$ and $A_1\cong V$. The algebra $A$ is $\mathbb{F}$-central, meaning that the left action of $\lambda \in \ff$ on $A$ is the same as its right action. However if there is an arrow $\alpha$ joining two distinct vertices $i$ and $j$, $A$ is not $k$-central since the left and right actions of $e_i$ on $\alpha$ are different.  The $A$-bimodules considered in this paper are not necessarily $k$-symmetric, meaning that the left and right actions of an element of $k$ are not necessarily equal, but they are always assumed to be $\mathbb{F}$-symmetric. Setting $A^e=A\otimes_{\mathbb{F}} A^{op}$, any $A$-bimodule can be viewed as a left (or right) $A^e$-module, as usual. 

For brevity, the notation $\otimes_{\mathbb{F}}$ is replaced by the unadorned tensor product $\otimes$. Similarly for the notations $\Hom _{\mathbb{F}}$ and $\dim_{\mathbb{F}}$ abbreviated to $\Hom$  and $\dim$. If unspecified, a vector space is an $\mathbb{F}$-vector space and a linear map is $\mathbb{F}$-linear.

The tensor product $\otimes_k$ is different from the unadorned tensor product $\otimes$. However, if
$M$ is a right $A$-module and $N$ is a left $A$-module, 
then the natural linear map
$Me_i\otimes e_iN \rightarrow Me_i \otimes_k e_iN$
is an isomorphism, so that for $a\in Me_i$ and $b\in e_iN$, we can identify $a\otimes_k b = a\otimes b$. Similarly, if $M$ and $N$ are $A$-bimodules, $e_jMe_i$ and $e_iNe_j$ are $k$-bimodules, that may be viewed as left and right $k^e$-modules, where $k^e=k\otimes k$. The natural linear map
$e_jMe_i\otimes e_iNe_j \rightarrow e_jMe_i \otimes_{k^e} e_iNe_j$ 
is an isomorphism, so for $a\in e_jMe_i$ and $b\in e_iNe_j$, we can identify $a\otimes_{k^e} b = a\otimes b$. We shall freely use these identifications, without explicitly mentioning them.

Although the algebra $A$ is not $k$-central, we define its bar resolution $B(A)$ following the standard text~\cite{weib:homo} by
$ (A\otimes_k A^{\otimes_k \bullet} \otimes_k A, d)$ with
\begin{align*}
d(a\otimes_k a_1\ldots a_p \otimes_k a')&= aa_1\otimes_k a_2\ldots a_p \otimes_k a'+ \sum_{1\ppq i \ppq p-1} (-1)^i a\otimes_k a_1\ldots (a_i a_{i+1}) \ldots a_p \otimes_k a' \\&\quad+ (-1)^p a\otimes_k a_1\ldots a_{p-1} \otimes_k a_p a',
  \end{align*}
for $a$, $a'$ and $a_1, \ldots ,a_{p}$ in $A$. When $A$ is not $k$-central, the extra degeneracy is defined and is still a contracting homotopy, hence $B(A)$ is a resolution of $A$ by projective $A$-bimodules. See Lemma \ref{proj} below for the fact that the $A$-bimodules $A\otimes_k A^{\otimes_k p} \otimes_k A$ are projective.

For any $A$-bimodule $M$, Hochschild homology and cohomology are defined by
$$\HH _{\bullet}(A,M)=\Tor ^{A^e}_{\bullet}(M,A)=H_{\bullet}(M\otimes_{A^e} B(A),M\otimes_{A^e}d),$$
$$\HH ^{\bullet}(A,M)=\Ext _{A^e}^{\bullet}(A,M)=H^{\bullet}(\Hom _{A^e}(B(A),M),\Hom _{A^e}(d,M)).$$
Given any $k$-bimodule $E$, there are well known vector space isomorphisms 
\begin{equation}\label{isos}
\begin{array}{crcl}
&M\otimes_{A^e}(A\otk E\otk A)&\rightarrow &M\otimes_{k^e}E\\
&m\otimes_{A^e}(a\otk x\otk a')&\mapsto&a'ma\otimes_{k^e}x\\\\
\text{ and }&\Hom_{k^e}(E,M)&\rightarrow
  &\Hom_{A^e}(A\otk E\otk A,M)\\
&f&\mapsto&(a\otk x\otk a'\mapsto af(x)a')
\end{array}
\end{equation}

Taking $E=A^{\otk p}$ and transporting $M\otimes_{A^e}d$ and $\Hom _{A^e}(d,M)$ via these isomorphisms, we obtain the Hochschild differentials $b^H$ and $b_H$, so that
$$\HH _{\bullet}(A,M)\cong H_{\bullet}(M\otimes_{k^e} A^{\otimes_k \bullet},b^H),$$
$$\HH ^{\bullet}(A,M)\cong H^{\bullet}(\Hom _{k^e}(A^{\otimes_k \bullet},M),b_H).$$
The Hochschild homology differential is then defined, for $m\in M$ and $a_1, \ldots ,a_{p}$ in $A$, by
\begin{align*}
b^H_p(m\otimes_{k^e} (a_1\ldots a_p))&= ma_1\otimes_{k^e} (a_2\ldots a_p) + \sum_{1\ppq i \ppq p-1} (-1)^i m\otimes_{k^e} (a_1\ldots (a_i a_{i+1}) \ldots a_p) \\&\quad+ (-1)^p a_pm \otimes_{k^e} (a_1\ldots a_{p-1}).
  \end{align*}
  The Hochschild cohomology differential  (including a Koszul sign in $\Hom _{A^e}(d,M)$) is  defined, for $f\in \Hom _{k^e}(A^{\otimes_k p},M)$ and $a_1,  \ldots ,a_{p+1}$ in $A$, by
  \begin{align*}
b_H^{p+1}(f)(a_1\ldots a_{p+1})&= f(a_1\ldots a_p)a_{p+1} - \sum_{1\ppq i \ppq p} (-1)^{p+i} f(a_1\ldots (a_i a_{i+1}) \ldots a_{p+1}) \\&\quad- (-1)^p a_1f(a_2 \ldots a_{p+1}).
    \end{align*}

\subsection{Koszul homology and cohomology } \label{khom}

Let $A=T_k(V)/(R)$ be a quadratic $k$-algebra over $\mathcal{Q}$. Following~\cite{priddy:kreso,vdb:nch,bgs,bls:kocal}, the Koszul complex $K(A)$ is the subcomplex of the bar resolution $B(A)$ defined by the sub-$A$-bimodules $A\otimes_k W_p \otimes_k A$ of $A\otimes_k A^{\otimes_k p} \otimes_k A$, where $W_0=k$, $W_1=V$ and, for $p\pgq 2$,
\begin{equation} \label{defw}
  W_{p}=\bigcap_{i+2+j=p}V^{\otimes_k i}\otimes_k R\otimes_k V^{\otimes_k j}.
\end{equation}
Here $W_p$ is considered as a sub-$k$-bimodule of $V^{\otimes_k p} \subseteq A^{\otimes_k p}$. It is immediate that the differential $d$ of $K(A)$ is defined on $A\otimes_k W_p \otimes_k A$ by
\begin{equation} \label{defd}
d(a \otimes_k x_1 \ldots x_{p} \otimes_k a') =ax_1\otimes_k x_2 \ldots x_{p}\otimes_k a'+(-1)^p a \otimes_k x_1 \ldots x_{p-1}\otimes_k x_{p} a',
\end{equation}
for $a$, $a'$ in $A$ and $x_1 \ldots x_{p}$ in $W_p$.

In this paper, we systematically follow~\cite{bls:kocal} for the notation of elements of $W_p$. Let us recall this notation. As in (\ref{defd}), an arbitrary element of $W_p$ is denoted by a product $x_1 \ldots x_{p}$ thought of as a sum of such products, where $x_1, \ldots , x_p$ are in $V$. Moreover, regarding $W_p$ as a subspace of $V^{\otimes_k q}\otimes_k W_r \otimes_k V^{\otimes_k s}$ with $q+r+s=p$, the 
element $x_1 \ldots  x_p$ viewed in $V^{\otimes_k q}\otimes_k W_r \otimes_k V^{\otimes _k s}$ will be denoted by the \emph{same} notation, meaning that 
the product $x_{q+1} \ldots x_{q+r}$ represents an element of $W_r$ and the other $x_i$ are  arbitrary in $V$.

We pursue along the same lines as~\cite{bls:kocal}. We present the different objects with their fundamental results more quickly. We keep the same notations as in~\cite{bls:kocal} and we leave the details to the reader when they are the same as in the one vertex case.

The homology of $K(A)$ is equal to $A$ in degree $0$, and to $0$ in degree $1$.
The quadratic algebra $A$ is said to be \emph{Koszul} if the homology of $K(A)$ is $0$ in any degree
$>1$. Denote by $\mu:A\otimes_k  A \rightarrow A$ the multiplication of $A$. Then $A$ is Koszul if
and only if $\mu: K(A)\rightarrow A$ is a resolution of $A$. If $R=0$ and if $R=V\otimes_kV$, then
$A$ is Koszul. Besides these extreme examples, many  Koszul algebras occur in the literature, see
for instance \cite{popo:quad,mv:intro} among many others, and it is well-known that preprojective algebras are Koszul when the graph is not Dynkin of type ADE (see Proposition~\ref{notkoszul} and the references in its proof). 

 The $A$-bimodules $A\otimes_k W_p \otimes_k A$ forming $K(A)$ are projective and finitely
 generated. Indeed,  $W_p$ is a sub-$k$-bimodule of $V^{\otimes_k p}$, so that  this fact is an
 immediate consequence of the following well known lemma (see for instance \cite[Proof of Lemma 2.1]{cibils}). We give an elementary proof here.

\begin{Lm} \label{proj}
Let $E$ be a $k$-bimodule. 
\begin{enumerate}[\itshape(i)]
\item The $A$-bimodule $A\otimes_k E \otimes_k A$ is projective.

\item  If $E$ is finite dimensional, then the $A$-bimodule $A\otimes_k E \otimes_k A$ is finitely
  generated.
\end{enumerate}

\end{Lm}
\begin{proof}
Clearly $E=\bigoplus_{i,j \in \mathcal{Q}_0} e_jEe_i$. From $A=\bigoplus_{i \in \mathcal{Q}_0} Ae_i=\bigoplus_{j \in \mathcal{Q}_0} e_jA$, we deduce that the $A$-bimodule $A\otimes_k E \otimes_k A$ is isomorphic to the $A$-bimodule
$$F_1= \bigoplus_{i,j \in \mathcal{Q}_0} Ae_j \otimes e_jEe_i\otimes e_iA.$$
Considering $F_1$ as a sub-$A$-bimodule of $F=A\otimes E \otimes A$, we see that $F=F_1\oplus F_2$, where
$$F_2=\bigoplus_{i_1, i_2, i_3, i_4 \in \mathcal{Q}_0} Ae_{i_1} \otimes e_{i_2}Ee_{i_3}\otimes e_{i_4}A$$
in which the sum is taken over the set of indices with $i_1\neq i_2$ and $i_3\neq i_4$. As the $A$-bimodule $F$ is free, we conclude that $F_1$ is projective.  Part \textit{(ii)} follows from the fact that $E$ is finite dimensional if and only if all the $ (e_jEe_i)$ are finite dimensional and is left to the reader.
\end{proof}
\begin{Df}
For any $A$-bimodule $M$, Koszul homology and cohomology are defined by
$$\HK_{\bullet}(A,M)=H_{\bullet}(M\otimes_{A^e} K(A))\text{ and }  \HK^{\bullet}(A,M)=H^{\bullet}(\Hom _{A^e}(K(A),M)).$$
We set $\HK_{\bullet}(A)=\HK_{\bullet}(A,A)$ and $\HK^{\bullet}(A)=\HK^{\bullet}(A,A)$.
\end{Df}

Since $K(A)$ is a complex of projective $A$-bimodules, $M\mapsto \HK_{\bullet}(A,M)$ and $M\mapsto
\HK^{\bullet}(A,M)$ define $\delta$-functors from the category of $A$-bimodules to the category of
vector spaces, that is, a short exact sequence of $A$-bimodules naturally gives rise to a long exact sequence
in Koszul homology and in Koszul cohomology~\cite[Chapter 2]{weib:homo}. As in~\cite{bls:kocal}, $\HK_p(A,M)$
(respectively  $\HK^p(A,M)$) is isomorphic to a Hochschild hyperhomology (respectively hypercohomology) space.

The inclusion $\chi : K(A) \rightarrow B(A)$ is a  morphism of complexes that induces the following morphisms of complexes
$$\tilde{\chi}=M\otimes_{A^e}\! \chi: M\otimes_{A^e} K(A) \rightarrow M\otimes_{A^e} B(A),$$
$$\chi^{\ast}=\Hom _{A^e} (\chi, M): \Hom _{A^e}(B(A),M) \rightarrow \Hom _{A^e}(K(A),M).$$
The linear maps $H(\tilde{\chi}): \HK_p(A,M) \rightarrow \HH _p(A,M)$ and $H(\chi^{\ast}): \HH ^p(A,M) \rightarrow \HK^p(A,M)$ are always isomorphisms for $p=0$ and $p=1$, and if $A$ is Koszul they are isomorphisms for any $p$.

Taking $E=W_p$ in the isomorphisms  \eqref{isos}, we get isomorphisms 
\begin{align*}
  \HK_{\bullet}(A,M)&\cong H_{\bullet}(M\otimes_{k^e} W_{\bullet},b^K),\\
  \HK^{\bullet}(A,M)&\cong H^{\bullet}(\Hom _{k^e}(W_{\bullet},M),b_K)
  \end{align*}
with differentials $b^K: M\otimes_{k^e} W_{p} \rightarrow M\otimes_{k^e} W_{p-1}$ and $b_K: \Hom _{k^e}(W_{p},M)\rightarrow \Hom _{k^e}(W_{p+1},M)$  given by
\begin{align} \label{defb}
b^K_p(m \otimes_{k^e} x_1 \ldots x_{p}) &=m.x_1\otimes_{k^e} x_2 \ldots x_{p} +(-1)^p x_p .m \otimes_{k^e} x_1 \ldots x_{p-1},
\\ \label{defcob}
b_K^{p+1}(f)( x_1 \ldots x_{p+1}) &=f(x_1\ldots x_{p}).x_{p+1} -(-1)^p x_1 .f(x_2 \ldots x_{p+1}),
\end{align}
where $m \in M$ and $x_1 \ldots x_p \in W_p$, respectively $f\in \Hom _{k^e}(W_p,M)$ and $x_1 \ldots x_{p+1} \in W_{p+1}$. 

Note that the $k$-algebra $A$ is \emph{augmented} by the natural projection $\epsilon_A : A \rightarrow A_0 \cong k$. Let us examine now the particular case $M=k$, where $k$ is the $A$-bimodule defined by $\epsilon_A$. The action on $k$ of an element of $A_p$ with $p>0$ is zero, so that the Koszul differentials vanish when $M=k$. Consequently, we have the linear isomorphisms
$$\HK_p(A,k) \cong k\otimes_{k^e} W_p \cong \bigoplus_{i\in \mathcal{Q}_0} e_i W_p e_i,$$
$$\HK^p(A,k) \cong \Hom _{k^e} (W_p,k) \cong \bigoplus_{i\in \mathcal{Q}_0} \Hom (e_i W_p e_i, \mathbb{F}).$$
In particular $\HK^p(A,k)\cong \Hom  (\HK_p(A,k), \mathbb{F})$, generalising~\cite[Proposition 2.8]{bls:kocal}.

\subsection{Koszul cup and cap products} \label{prod}

Let $A=T_k(V)/(R)$ be a quadratic $k$-algebra over $\mathcal{Q}$.  As in~\cite{bls:kocal}, the usual cup and cap products $\smile$ and $\frown$ in Hochschild cohomology and homology provide, by restriction from $B(A)$ to $K(A)$, the Koszul cup and cap products $\underset{K}{\smile}$ and $\underset{K}{\frown}$ in Koszul cohomology and homology. Let us give these products, expressed on Koszul cochains and chains. Let $P$, $Q$ and $M$ be $A$-bimodules. For $f\in \Hom _{k^e}(W_p,P)$, $g\in \Hom _{k^e}(W_q,Q)$ and $z=m \otimes_{k^e} x_1 \ldots x_q \in M\otimes_{k^e} W_q$, we define $f\underset{K}{\smile}g \in \Hom _{k^e}(W_{p+q},P\otimes_A Q)$, $f\underset{K}{\frown}z \in (P\otimes_A M)\otimes_{k^e} W_{q-p}$ and $z\underset{K}{\frown}f \in (M\otimes_A P)\otimes_{k^e} W_{q-p}$ by
\begin{align}
\label{kocup}
 (f\underset{K}{\smile} g&) (x_1 \ldots x_{p+q}) = (-1)^{pq} f(x_1 \ldots x_p)\otimes_A \, g(x_{p+1} \ldots  x_{p+q}),\\
  \label{kolcap}
   f\underset{K}{\frown} z &= (-1)^{(q-p)p} (f(x_{q-p+1} \ldots x_q)\otimes_A m) \otimes_{k^e} x_1 \ldots  x_{q-p},\\
  \label{korcap}
  z\underset{K}{\frown} f &= (-1)^{pq} (m\otimes_A f(x_1 \ldots x_p)) \otimes_{k^e} x_{p+1} \ldots  x_{q}.
  \end{align}

For any Koszul cochains $f$, $g$ $h$ and any Koszul chain $z$, we have the {\em associativity relations}
\begin{eqnarray*}\label{associativityrelations}
  (f\underset{K}{\smile}g)\underset{K}{\smile} h & = & f\underset{K}{\smile} (g\underset{K}{\smile} h), \\
  f\underset{K}{\frown} (g\underset{K}{\frown} z) & = & (f\underset{K}{\smile}g)\underset{K}{\frown} z,\\
  (z\underset{K}{\frown} g) \underset{K}{\frown} f & = & z \underset{K}{\frown} (g\underset{K}{\smile} f),\\
  f\underset{K}{\frown} (z\underset{K}{\frown} g) & = & (f\underset{K}{\frown}z)\underset{K}{\frown} g,
\end{eqnarray*}
inducing the same relations on Koszul classes. 

As in~\cite[Subsection 3.1]{bls:kocal}, for any Koszul cochains $f\in\Hom_{k^e}(W_p,P)$ and
$g\in\Hom_{k^e}(W_q,Q)$, we have the identity
\begin{equation}
b_K(f\cupk g)=b_k(f)\cupk g+(-1)^p f\cupk b_K(g),\label{eq:dga}
\end{equation}
so that  $(\Hom _{k^e}(W_{\bullet},A), b_K, \underset{K}{\smile})$ is a  DG algebra. This DG algebra
will play an essential role in Section 3 and will be denoted by $\tilde{A}$. Note that
$H(\tilde{A})=\HK^{\bullet}(A)$ is a graded algebra for $\underset{K}{\smile}$. Moreover, formula
\eqref{eq:dga} and analogous formulas for $b^K(f\capk z)$ and $b^K(z\capk f)$ show that for any $A$-bimodule $M$, $\Hom _{k^e}(W_{\bullet},M)$ and  $M\otimes_{k^e} W_{\bullet}$ are DG bimodules over $\tilde{A}$ for the actions of $\underset{K}{\smile}$ and $\underset{K}{\frown}$ respectively, so that $\HK^{\bullet}(A,M)$ and $\HK_{\bullet}(A,M)$ are graded $\HK^{\bullet}(A)$-bimodules.

\begin{Df} \label{variouskc}
  Let $A=T_k(V)/(R)$ be a quadratic $k$-algebra over a finite quiver $\mathcal{Q}$.
  \begin{enumerate}[\itshape(i)]
  \item 

 The general Koszul calculus of $A$ is the datum of all the spaces $\HK^{\bullet}(A,P)$ and $\HK_{\bullet}(A,M)$ endowed with $\underset{K}{\smile}$ and $\underset{K}{\frown}$, when the $A$-bimodules $P$ and $M$ vary.

\item  The Koszul calculus of $A$ consists of the graded associative algebra $\HK^{\bullet}(A)$ and of
all the graded $\HK^{\bullet}(A)$-bimodules $\HK^{\bullet}(A,M)$ and $\HK_{\bullet}(A,M)$, when the
$A$ bimodule $M$ vary.

\item  The restricted Koszul calculus of $A$ consists of the graded associative algebra
$\HK^{\bullet}(A)$ and of the graded $\HK^{\bullet}(A)$-bimodule $\HK_{\bullet}(A)$.

\item  The scalar Koszul calculus of $A$ consists of the graded associative algebra
$\HK^{\bullet}(A,k)$ and of the graded $\HK^{\bullet}(A,k)$-bimodule $\HK_{\bullet}(A,k)$.
\end{enumerate}

  \end{Df}

Since $\HK^0(A)=Z(A)$ is the centre of the algebra $A$, the spaces $\HK^p(A,M)$ and $\HK_p(A,M)$ are
symmetric $Z(A)$-bimodules (left and right actions coincide). However $\HK^p(A,M)$ and $\HK_p(A,M)$
are not $k$-bimodules in general. Indeed, $\HK^0(A)=Z(A)$ itself is not a $k$-bimodule whenever there is an
arrow joining two different vertices $i$ and $j$, since in this case $e_i1=e_i$ is not in $Z(A)$.

\begin{Ex} \label{point}
If $\mathcal{Q}_1=\emptyset$, then $V=0$ and $A$ is reduced to $k$. The (general) Koszul calculus of $k$ coincides with the (tensor) category of $k$-bimodules. 
\end{Ex}

\begin{Ex} \label{example-A3}
  In order to illustrate the notation and the forthcoming results in the paper, we present the case of the  preprojective algebra of type $\mathrm{A}_3$. This algebra is not Koszul and its Koszul calculus differs from its Hochschild calculus (see Subsection \ref{subsec:comparison hochschild}). This is a special case of the more general examples detailed in Section \ref{sec:explicit}.

Let $A$ be the preprojective algebra of type  $\mathrm{A}_3$ over $\ff=\cc$, that is, the $\cc$-algebra defined by the quiver  \[\xymatrix@C=10pt{\ov Q&&
    0\ar@/^/[rr]^{a_0}&&1\ar@/^/[rr]^{a_1}\ar@/^/[ll]^{a_0^*}&&2\ar@/^/[ll]^{a_1^*}}\] subject to the relations
\begin{align*}
\sigma _0&=-a_0^*a_0,& \sigma _1&=a_0a_0^*-a_1^*a_1, &\sigma _2=a_1a_1^*.
\end{align*} The algebra $A$ has dimension $10$ and a basis of $A$ over $\cc$ is given by the elements $e_i$ for $0\ppq i\ppq 2$, $a_i$ and $a_i^*$ for $0\ppq i\ppq 1$, $a_1^*a_1$, $a_1a_0$ and $a_0^*a_1^*$.
We then have $W_0=\cc \ov{Q}_0=\cc\langle e_1,e_1,e_2 \rangle=k$, $W_1=\cc \ov{Q}_1=\cc\langle a_1,a_1,a_0^*,a_1^*\rangle=V$ and $W_2=\cc\langle{\sigma _0,\sigma _1,\sigma _2}\rangle=R$.

Now consider $W_3=(V\otk R)\cap (R\otk V)$, viewed inside $\cc\ov{Q}_3$. An element $u$ in $W_3$ can therefore be written as a path in $\cc\ov{Q}_3$ in two ways:
\[ u=\sum_{i=0}^1(\lambda _ia_i\sigma _i+\lambda _i^* a_i^*\sigma _{i+1})=\sum_{i=0}^1(\mu _i\sigma _{i+1}a_i+\mu _i^*\sigma _ia_i^* )\] with $\lambda _i$, $\lambda _i^*$, $\mu _i$, $\mu _i^*$ in $\cc$. Then, in $\cc\ov{Q}_3$, we have
\begin{align*}
(-\lambda _0-\mu _0)a_0a_0^*a_0&+(-\lambda _1-\mu _1)a_1a_1^*a_1+(\lambda _0^*+\mu _0^*)a_0^*a_0a_0^*+(\lambda _1^*+\mu _1^*)a_1^*a_1a_1^*\\&\quad+\lambda _1a_1a_0a_0^*-\lambda _0^*a_0^*a_1^*a_1+\mu _0a_1^*a_1a_0-\mu _1^*a_0a_0^*a_1^*=0
\end{align*} so that all the coefficients $\lambda _i$, $\lambda _i^*$, $\mu _i$ and $\mu _i^*$ must be zero, hence $u=0$ and $W_3=0$.

Since $W_p=(W_{p-1}\otk V)\cap(V\otk W_{p-1}) $ for all $p\geqslant 3$, it follows that $W_p=0$ for all $p\pgq 3$. This is true for any preprojective algebra of type
$\Delta $ with $\Delta $ different from $\mathrm{A}_1$ and $\mathrm{A}_2$, see Theorem~\ref{w3iszero}.

 The Koszul complex $K(A)$ is therefore
\[ \cdots\rightarrow  0\longrightarrow A\otk R\otk A\xrightarrow {d_2}A\otk V\otk A\xrightarrow {d_1}A\otk A\rightarrow 0 \] with
\begin{align*}
d_1(a\otk x\otk a')&=ax\otk a'-a\otk xa'\\d_2(a\otk \sum_{i=1}^n\lambda _ix_iy_i\otk a')&=\sum_{i=1}^n\lambda _i\left(ax_i\otk y_i\otk a'_a\otk x_i\otk y_ia'\right).
\end{align*}

Applying $\Hom_{A^e}(-,A)$ and using the natural isomorphism $\Hom_{A^e}(A\otk E\otk A,A)\cong
\Hom_{k^e}(E,A)$ for any $k$-bimodule $E$  \eqref{isos}, we get the complex
\[ 0\rightarrow \Hom_{k^e}(k,A)\xrightarrow {b_K^1}\Hom_{k^e}(V,A)\xrightarrow {b_K^2}\Hom_{k^e}(R,A)\rightarrow 0\rightarrow \cdots \] Before we describe the maps, let us note that $\Hom_{k^e}(k,A)\cong \bigoplus_{i=0}^2e_iAe_i$ has basis $\set{e_0,e_1,e_2,a_1^*a_1}$, that a general element $f$ in $\Hom_{k^e}(V,A)$ is defined by $f(a_i)=\lambda _ia_i$ and $f(a_i^*)=\lambda _i^*a_i^*$ for $i=0,1$ with $\lambda _i,$ $\lambda _i^*$ in $\cc$, and that a general element $g$ in $\Hom_{k^e}(R,A)$ is defined by $g(\sigma _i)=\alpha _i e_i$ for $i=0,2$ and $g(\sigma _2)=\alpha _2e_1+\beta  a_1^*a_1$ for some scalars $\alpha _i$ and $\beta $. Then
\begin{align*}
&  b_K^1(\sum_{i=0}^2u_ie_i+va_1^*a_1) \text{ is defined by } a_i\mapsto (u_{i+1}-u_i)a_i \text{ and } a_i^* \mapsto(u_i-u_{i+1})a_i^*\text{ for }i=0,1\\
&  b_K^2(f)\text{ is defined by }\sigma _0\mapsto 0,\ \sigma _1\mapsto (\lambda _0+\lambda _0^*-\lambda _1-\lambda _1^*)a_1^*a_1 \text{ and }\sigma _2\mapsto 0.
\end{align*}

It is then easy to see that $\HK^0(A)=\cc\langle z_0=1,z_1=a_1^*a_1\rangle$, that $\HK^1(A)=\cc\langle \ov\zeta _0\rangle$ with $\zeta _0\in \Hom_{k^e}(V,A)$ defined by $\zeta _0(a_i)=a_i$ and $\zeta _0(a_i^*)=0$, and that $\HK^2(A)=\cc\langle \ov h_0,\ov h_1,\ov h_2\rangle$ with $h_i\in\Hom_{k^e}(R,A)$ defined by $h_i(\sigma _j)=\delta _{ij}e_i$.

Moreover, the fundamental $1$-cocycle $\me_A$ is equal to $2\zeta _0+b_K^1(2e_0+e_1)$.

The Koszul cup products can easily be found using the formula \eqref{kocup}. It follows that $\cupk$ is graded commutative, that  $1\cupk x=x$ for any $x\in \HK^\bullet (A)$ and that all other cup products are $0$ in $\HK^\bullet (A)$. For instance, $z_1\cupk h_1$ is the coboundary $b_K^2(f)$ where $f$ sends $a_0$ to $a_0$ and all other arrows to $0$.
In particular, $\me_A\cupk \me_A=0$.

In order to determine the Koszul homology of $A$, 
we could also apply the functor $A\ot_{A^e}-$ to the complex $K(A)$ and compute the homology of the
complex obtained. However, we can also use our duality result, Theorem \ref{complexduality}. Set $\omega _0=\sum_{i=0}^2e_i\ot \sigma _i\in A\ot_{k^e}R$. There is an isomorphism $\theta _A:\HK^\bullet(A)\rightarrow \HK_{2-\bullet}(A)$ given by $\ov f\mapsto\ov\omega _0\capk \ov f$. Explicitly in our example,
\begin{enumerate}[\ ]
\item $\theta _A(z_0)=\ov\omega _0$ and $\theta _A(z_1)=\ov{z_1\ot \sigma _1}$ form a basis of $\HK_2(A)$;
\item $\theta _A(\ov{\zeta }_0)=\ov{a_0\ot a_0^*+a_1\ot a_1^*}$ forms a basis of $\HK_1(A)$;
\item $\theta _A(\ov h_i)=\ov{e_i\ot e_i}$ for $0\ppq i\ppq 2$ form a basis of $\HK_0(A)$.
  \end{enumerate} The Koszul cap products can also be obtained using duality and they all vanish except the cap products $z_0\capk x=x=x\capk z_0$ for all $x\in \HK_\bullet(A)$.

\end{Ex}

\subsection{Fundamental formulas of Koszul calculus} \label{fundaformulas}

Let $A=T_k(V)/(R)$ be a quadratic $k$-algebra over $\mathcal{Q}$. We continue to follow the one vertex case~\cite{bls:kocal}. First, we define the Koszul cup and cap brackets. Let $P$, $Q$ and $M$ be $A$-bimodules, and take $f\in \Hom _{k^e}(W_p,P)$, $g\in \Hom _{k^e}(W_q,Q)$, $z \in M\otimes_{k^e} W_q$. When $P$ or $Q$ is equal to $A$, we set
\begin{equation} \label{kocupbra}
[f, g]_{\underset{K}{\smile}} =f\underset{K}{\smile} g - (-1)^{pq} g\underset{K}{\smile} f.
\end{equation}
When $P$ or $M$ is equal to $A$, we set
\begin{equation} \label{kocapbra}
[f, z]_{\underset{K}{\frown}} =f\underset{K}{\frown} z - (-1)^{pq} z\underset{K}{\frown} f.
\end{equation}
These brackets induce brackets on the Koszul classes.

The Koszul 1-cocycles $f:V \rightarrow M$ are called Koszul derivations with coefficients in $M$. Such an $f$ extends to a unique derivation from the $k$-algebra $A$ to the $A$-bimodule $M$, realising an isomorphism from the space of Koszul derivations with coefficients in $M$ to the space of derivations from $A$ to $M$. In particular, the Koszul 1-cocycle from $V$ to $A$ coinciding with the identity map on $V$, is sent to the Euler derivation $D_A$ of the graded algebra $A$. This Koszul 1-cocycle is denoted by $\me_A$ and is called the \emph{fundamental 1-cocycle}. Its Koszul class is denoted by $\overline{\me}_A$ and is called the \emph{fundamental $1$-class}. In the one vertex case, $\me_A$ is not a coboundary if $V\neq 0$~\cite{bls:kocal}, but this property does not hold in general.
\begin{Lm} \label{eAcob}
Let $A=T_k(V)/(R)$ be a quadratic $k$-algebra over $\mathcal{Q}$ with $\mathcal{Q}_1 \neq \emptyset$. If the underlying graph of $\mathcal{Q}$ is simple, that is, it contains neither loops nor multiple edges, then $\me_A$ is a coboundary.
\end{Lm}
\begin{proof}
The 1-cocycle $\me_A$ is a coboundary if and only if there exists a $k^e$-linear map $c: k
\rightarrow k$ such that $\me_A=b_K(c)$. Such a map is of the form $c(e_i)=\lambda_ie_i$ with
$\lambda_i \in \mathbb{F}$, for all $i\in \mathcal{Q}_0$. Then $\me_A=b_K(c)$ if and only if
$\lambda_{\mt(\alpha)}-\lambda_{\mo(\alpha)} =1$ for any $\alpha \in
\mathcal{Q}_1$. The assumption on the graph means that $\mathcal{Q}$ has no loop and that given two
distinct vertices, there is at most one arrow joining them. Then we can choose $\lambda_{\mt(\alpha)}=1$ and $\lambda_{\mo(\alpha)}=0$.
\end{proof}

This proof shows that if the quiver $\mathcal{Q}$ has a loop, $\me_A$ is not a coboundary. The same conclusion holds if $\car \ff \neq 2$ and $\mathcal{Q}$ contains an oriented $2$-cycle.

The following propositions are proved as~\cite[Theorem 3.7, Theorem 4.4, Corollary 3.10, Corollary 4.6]{bls:kocal} of the one vertex case. Formulas (\ref{fundacoho}) and (\ref{fundaho}) are the fundamental formulas of Koszul calculus.

\begin{Po} \label{fundamental}
Let $A=T_k(V)/(R)$ be a quadratic $k$-algebra over $\mathcal{Q}$. For any Koszul cochain $f$ and any Koszul chain $z$ with coefficients in an $A$-bimodule $M$, we have
\begin{align}
\label{fundacoho}
 b_K(f)&=- [\me_A, f]_{\underset{K}{\smile}},\\
 \label{fundaho}
 b^K(z)&=-[\me_A, z]_{\underset{K}{\frown}}. 
  \end{align}
\end{Po}
\begin{Po} \label{kocommutation}
Let $A=T_k(V)/(R)$ be a quadratic $k$-algebra over $\mathcal{Q}$ and let $M$ be an $A$-bimodule. For
any $\alpha \in \HK^p(A,M)$ with $p=0$ or $p=1$, $\beta \in \HK^q(A)$ and $\gamma \in \HK_q(A)$, we
have the identities 
\begin{align} \label{kocupbrazero}
[\alpha, \beta]_{\underset{K}{\smile}} &=0,\\
\label{kocapbrazero}
[\alpha, \gamma]_{\underset{K}{\frown}} &=0.
\end{align} Identity \eqref{kocapbrazero} also holds if $p=q\not\in\set{0,1}$.
\end{Po}

\subsection{Higher Koszul calculus} \label{higherkc}

 Higher Koszul homology is the homology of the Koszul homology, and similarly for cohomology.
 Precisely, let $A=T_k(V)/(R)$ be a quadratic $k$-algebra over $\mathcal{Q}$. Formula (\ref{kocup})
 shows that the map $\me_A\underset{K}{\smile}\me_A: W_2=R \rightarrow A$ is zero. Therefore, $\me_A\underset{K}{\smile} -$ is a cochain differential on $\Hom _{k^e}(W_{\bullet},M)$, and $\overline{\me}_A\underset{K}{\smile} -$ is a cochain differential on $\HK^{\bullet}(A,M)$. Similarly, $\me_A\underset{K}{\frown} -$ is a chain differential 
on $M \otimes_{k^e} W_{\bullet}$, and $\overline{\me}_A\underset{K}{\frown} -$ is a chain differential on $\HK_{\bullet}(A,M)$. For a $p$-cocycle $f:W_p \rightarrow M$ and $x_1 \ldots x_{p+1}$ in $W_{p+1}$, we have
$$(\me_A\underset{K}{\smile} f) (x_1 \ldots x_{p+1}) = f(x_1 \ldots x_p). x_{p+1}.$$
For a $p$-cycle $z=m\otimes_{k^e} x_1\ldots x_p$ in $M \otimes_{k^e} W_p$, we have 
\[\me_A\underset{K}{\frown} z = mx_1 \otimes_{k^e} x_2 \ldots  x_p.\]
\begin{Df} \label{hiko}
Let $A=T_k(V)/(R)$ be a quadratic $k$-algebra over a finite quiver $\mathcal{Q}$ and let $M$ be an $A$-bimodule. The differentials $\overline{\me}_A\underset{K}{\smile} -$ and $\overline{\me}_A\underset{K}{\frown} -$ are denoted by $\partial_{\smile}$ and $\partial_{\frown}$. 
The homologies of the complexes $(\HK^{\bullet}(A,M),\partial_{\smile})$  and $(\HK_{\bullet}(A,M),\partial_{\frown})$ are called the higher Koszul cohomology and homology of $A$ with coefficients in $M$ and are denoted by $\HK_{hi}^{\bullet}(A,M)$ and $\HK^{hi}_{\bullet}(A,M)$. We set $\HK_{hi}^{\bullet}(A)=\HK_{hi}^{\bullet}(A,A)$ and $\HK^{hi}_{\bullet}(A)=\HK^{hi}_{\bullet}(A,A)$.
\end{Df}

The higher classes of Koszul classes will be denoted between square brackets. For example, the unit $1$ of $A$ is still the unit of $\HK^{\bullet}(A)$, and $\partial_{\smile}(1)=\overline{\me}_A$ implies that $[\overline{\me}_A]=0$. If $\overline{\me}_A \neq 0$, the unit of $\HK^{\bullet}(A)$ does not survive in higher Koszul cohomology. 

As in the one vertex case, the actions of the Koszul cup and cap products of $\HK^{\bullet}(A)$ on $\HK^{\bullet}(A,M)$ and $\HK_{\bullet}(A,M)$ induce actions on higher cohomology and homology. Thus $\HK_{hi}^{\bullet}(A)$ is a graded algebra, and $\HK_{hi}^{\bullet}(A,M)$, $\HK^{hi}_{\bullet}(A,M)$ are graded $\HK_{hi}^{\bullet}(A)$-bimodules, constituting \emph{the higher Koszul calculus of $A$}. If $\overline{\me}_A = 0$, the higher Koszul calculus coincides with the Koszul calculus. It is the case when $A=k$ as in Example \ref{point}.  

For $M=k$, $\me_A\underset{K}{\smile} -$ and $\me_A\underset{K}{\frown} -$ vanish, so that the higher scalar Koszul calculus coincides with the scalar Koszul calculus. Proposition 3.12 in~\cite{bls:kocal} generalises immediately as follows.
\begin{Po}  \label{zerohikoco}
Let $A=T_k(V)/(R)$ be a quadratic $k$-algebra over $\mathcal{Q}$ and let $M$ be an $A$-bimodule. Then $\HK_{hi}^0(A,M)$ is the space of  elements $u$ in $Z(M)$ such that there exists $v\in M$ satisfying 
$u.a =v.a - a.v$ for any $a$ in $\mathcal{Q}_1$. 
\end{Po}

\subsection{Grading the restricted Koszul calculus by the weight} \label{weight}

A Koszul $p$-cochain $f:W_p \rightarrow A_m$ is said to be homogeneous of weight $m$. Since $\mathcal{Q}_1$ is finite, the spaces $W_p$ are finite dimensional, thus the space of Koszul cochains $\Hom _{k^e}(W_{\bullet},A)$ is 
$\mathbb{N}\times \mathbb{N}$-graded by the \emph{biweight} $(p,m)$, where $p$ is called the \emph{homological weight} and $m$ is called the \emph{coefficient weight}. If $f:W_p \rightarrow A_m$ 
and $g:W_q \rightarrow A_n$ are homogeneous of biweights $(p,m)$ and $(q,n)$ respectively, then $f\underset{K}{\smile} g : W_{p+q} \rightarrow A_{m+n}$ is homogeneous of biweight $(p+q,m+n)$. Moreover $b_K$ is homogeneous of biweight $(1,1)$ and the algebra $\HK^{\bullet}(A)$ 
is $\mathbb{N}\times \mathbb{N}$-graded by the biweight. The homogeneous component of biweight $(p,m)$ of $\HK^{\bullet}(A)$ is denoted by $\HK^p(A)_m$. Since
$$\partial_{\smile}: \HK^p(A)_m \rightarrow \HK^{p+1}(A)_{m+1},$$
the algebra $\HK_{hi}^{\bullet}(A)$ is $\mathbb{N}\times \mathbb{N}$-graded by the biweight, and its $(p,m)$-component is denoted by $\HK_{hi}^p(A)_m$. From Proposition \ref{zerohikoco}, we deduce the following.
\begin{Po}  \label{zerohikocoA}
Let $A=T_k(V)/(R)$ be a quadratic $k$-algebra over $\mathcal{Q}$. Assume that $A$ is finite dimensional. Let $\max$ be the highest $m$ such that $A_m\neq 0$. Then $\HK_{hi}^0(A)_{\max}$ is isomorphic to the space spanned by the cycles of $\mathcal{Q}$ of length $\max$. 
\end{Po}

Similarly, a Koszul $q$-chain $z$ in $A_n \otimes_{k^e} W_q$ is said to be homogeneous of weight $n$. The space of Koszul 
chains $A \otimes_{k^e} W_{\bullet}$ is $\mathbb{N}\times \mathbb{N}$-graded by the \emph{biweight} $(q,n)$, where $q$ is called the \emph{homological weight} and $n$ is called the 
\emph{coefficient weight}. Moreover $b^K$ is homogeneous of biweight $(-1,1)$ and the space $\HK_{\bullet}(A)$ is $\mathbb{N}\times \mathbb{N}$-graded by the biweight. 
The homogeneous component of biweight $(q,n)$ of $\HK_{\bullet}(A)$ is denoted by $\HK_q(A)_n$. Since 
$$\partial_{\frown}: \HK_q(A)_n \rightarrow \HK_{q-1}(A)_{n+1},$$
the space $\HK^{hi}_{\bullet}(A)$ is $\mathbb{N}\times \mathbb{N}$-graded by the biweight, and its $(q,n)$-component is denoted by $\HK^{hi}_q(A)_n$.

If $f:W_p \rightarrow A_m$ and $z \in A_n \otimes_{k^e} W_q$ are homogeneous of biweights $(p,m)$ and $(q,n)$ respectively, 
then $f\underset{K}{\frown} z$ and $z\underset{K}{\frown} f$ are homogeneous of biweight $(q-p, m+n)$ where
\begin{align} \label{kolcapA}
f\underset{K}{\frown} z &= (-1)^{(q-p)p} f(x_{q-p+1} \ldots x_q)a \otimes_{k^e} x_1 \ldots  x_{q-p},
\\ \label{korcapA}
  z\underset{K}{\frown} f &= (-1)^{pq} a\, f(x_1 \ldots x_p) \otimes_{k^e} x_{p+1} \ldots  x_{q},
\end{align}
and $z=a\otimes_{k^e} x_1\ldots x_q$. The $\Hom _{k^e}(W_{\bullet},A)$-bimodule $A \otimes_{k^e} W_{\bullet}$, the $\HK^{\bullet}(A)$-bimodule $\HK_{\bullet}(A)$ and the 
$\HK_{hi}^{\bullet}(A)$-bimodule $\HK^{hi}_{\bullet}(A)$ are thus $\mathbb{N}\times \mathbb{N}$-graded by the biweight. The proof of the following is left to the reader.
\begin{Po}  \label{easyhikohoA}
Let $A=T_k(V)/(R)$ be a quadratic $k$-algebra over $\mathcal{Q}$. We have
$$\HK^{hi}_0(A)_0 \cong \HK_0(A)_0 \cong  k.$$
Moreover $\HK_0(A)_1\cong \HK_1(A)_0$ is isomorphic to the space spanned by the loops of $\mathcal{Q}$, and $\partial_{\frown}: \HK_1(A)_0 \rightarrow \HK_0(A)_1$ identifies with the identity map on this space. 
As a consequence,
$$\HK^{hi}_0(A)_1\cong \HK^{hi}_1(A)_0= 0.$$
\end{Po}

\subsection{Invariance of Koszul calculus} \label{invariance}

In~\cite{berger:Ncal}, the first author proved that the Koszul calculus of an $N$-homogeneous algebra $A$ over a field $k$ only depends on the structure of associative algebra of $A$, independently of any presentation $A=T_k(V)/(R)$ of $A$ as an $N$-homogeneous algebra. This result was based on an isomorphism lemma due to Bell and Zhang~\cite{bz:isolemma}. In the quadratic case $N=2$, we are going to extend this Koszul calculus invariance to any quadratic quiver algebra. For that, we shall use an extension of the isomorphism lemma to quiver algebras with homogeneous relations, due to Gaddis~\cite{gaddis:isolemma}.

Let $\mathcal{Q}$ and $\mathcal{Q'}$ be finite quivers, and $\mathbb{F}$ be a field. We introduce
the commutative rings $k=\mathbb{F} \mathcal{Q}_0$ and $k'=\mathbb{F} \mathcal{Q'}_0$, the
$k$-bimodule $V=\mathbb{F} \mathcal{Q}_1$ and the $k'$-bimodule $V'=\mathbb{F} \mathcal{Q}'_1$. As
explained in Subsection \ref{not}, we make the identifications of graded algebras $T_k(V)\cong
\mathbb{F}\mathcal{Q}$ and $T_{k'}(V')\cong \mathbb{F}\mathcal{Q'}$. We are interested in the graded
$\mathbb{F}$-algebra isomorphisms $u: T_k(V) \rightarrow T_{k'}(V')$ given by a ring isomorphim
$u_0:k \rightarrow k'$ and by a $k$-bimodule isomorphism $u_1: V \rightarrow V'$, where $V'$ is a
$k$-bimodule via $u_0$. By \cite[Lemma 4]{gaddis:isolemma}, this implies that $u_0$ maps $\mathcal{Q}_0$ to $\mathcal{Q'}_0$, and the bijection $\mathcal{Q}_0 \rightarrow \mathcal{Q'}_0$ induced by $u_0$ transforms the adjacency matrix of $\mathcal{Q}$ into the adjacency matrix of $\mathcal{Q'}$.

Let us fix a sub-$k$-bimodule $R$ of $V\otimes_kV$ and a sub-$k'$-bimodule $R'$ of $V'\otimes_{k'}V'$. We define the graded $k$-algebra $A=T_k(V)/(R)$ and the graded $k'$-algebra $A'=T_{k'}(V')/(R')$. Following the terminology of the one vertex case, a graded $\mathbb{F}$-algebra isomorphism $u:A \rightarrow A'$ is called a \emph{Manin isomorphism} if $u$ is defined by a ring isomorphism $u_0:k \rightarrow k'$ (so $V'$ is a $k$-bimodule via $u_0$), and by a $k$-bimodule isomorphism $u_1: V \rightarrow V'$, such that the $k$-bimodule isomorphism $u_1^{\otimes_k 2}: V^{\otimes_k 2} \rightarrow V'^{\otimes_{k'} 2}$ satisfies $u_1^{\otimes_k 2}(R)=R'$. In particular, $u$ is an isomorphism of the augmented $k$-algebra $A$ to the augmented $k'$-algebra $A'$, the augmentations being the projections $A \rightarrow A_0 \cong k$ and $A' \rightarrow A'_0 \cong k'$.

As in~\cite{berger:Ncal}, for any $A$-bimodule $M$, the Manin isomorphism $u$ naturally defines an isomorphism of complexes from $(M\otimes_{k^e} W_{\bullet}, b^K)$ to $(M\otimes_{k'^e} W'_{\bullet}, b^K)$, where $M$ is an $A'$-bimodule via $u$, inducing natural isomorphisms $\HK_{\bullet}(A,M) \cong \HK_{\bullet}(A',M)$. Similarly, $u$ induces natural isomorphisms $\HK^{\bullet}(A',M) \cong \HK^{\bullet}(A,M)$. It is clear from the definitions in Subsection \ref{prod} that these isomorphisms respect the Koszul cup and cap products. To summarise  all these properties, we say that a Manin isomorphism induces isomorphic (general) Koszul calculi. Since $u_1(\me_A)= \me_{A'}$ by functoriality, it also induces isomorphic higher Koszul calculi.

Using Gaddis's theorem \cite[Theorem 5]{gaddis:isolemma}, we can now prove ungraded invariance. Let $C$ be a commutative ring. Let $A$
be an augmented associative $C$-algebra (not necessarily $C$-central) having a quadratic quiver
algebra presentation $B$, meaning that the augmented $C$-algebra $A$ is isomorphic to a quadratic
$k$-algebra $B=T_k(V)/(R)$ over a finite quiver $\mathcal{Q}$, naturally augmented over $k$ by the
projection $B \rightarrow B_0\cong k$. This implies that the ring $C$ is isomorphic to $k\cong
\mathbb{F} \mathcal{Q}_0$. Then \emph{we can define the (general) Koszul calculus of $A$ as being the (general) Koszul calculus of $B$.}
Indeed, if $B'=T_{k'}(V')/(R')$ over a finite quiver $\mathcal{Q'}$ is another quadratic quiver algebra presentation of $A$, the ungraded augmented $k$-algebra $B$ is isomorphic to the ungraded augmented $k'$-algebra $B'$. By Gaddis's theorem, there exists a Manin isomorphism from $B$ to $B'$, thus the (general) Koszul calculi of $B$ and $B'$ are isomorphic by Manin invariance. \emph{The higher Koszul calculus of $A$ is also defined as being the higher Koszul calculus of $B$}.

\subsection{Comparing Koszul (co)homology with Hochschild (co)homology in degree 2} \label{degree2}

Let $A=T_k(V)/(R)$ be a quadratic $k$-algebra over $\mathcal{Q}$. Recall that, for $p=0$ and $p=1$, we have linear isomorphisms $\HK_p(A,M) \cong \HH _p(A,M)$ and $\HH ^p(A,M) \cong \HK^p(A,M)$ (Subsection \ref{khom}). It is no longer true if $p\pgq 2$ and if $A$ is an arbitrary non-Koszul algebra. Preprojective algebras of Dynkin type will give infinitely many counterexamples when $p=2$. However, in general, we can compare the Koszul and Hochschild spaces when $p=2$, by  providing a surjection $\HK_2(A,M) \rightarrow \HH _2(A,M)$ and an injection $\HH ^2(A,M) \rightarrow \HK^2(A,M)$. To prove that, we use a minimal projective resolution of the graded $k$-algebra $A$, described as follows.

As in the one vertex case~\cite{bls:kocal}, we know that, in the category of graded $A$-bimodules, $A$ has a minimal projective resolution $P(A)$ whose component of homological degree $p$ can be written as $A\otimes_k E_p\otimes_k A$, where $E_p$ is a weight-graded $k$-bimodule. Then $P_l(A)=P(A)\otimes_Ak$ (respectively $P_r(A)=k\otimes_A P(A)$) is a minimal projective resolution of the graded left (respectively right) $A$-module $k$ (see for instance~\cite{berger:dimension}) and the differential $\delta$ of $P(A)$ is the graded sum of the differentials $\delta \otimes_A\id_k$ and $\id_k\otimes_A \delta $ naturally extended to $P(A)$. 

Define the left (respectively right) Koszul complex $K_l(A)=K(A)\otimes_A k$ (respectively $K_r(A)=k \otimes_A K(A)$). Now using \cite[Subsections 2.7 and 2.8]{bgs},  \cite[Chapter 1, Proposition 3.1]{popo:quad} adapted to the case where $k$ is a semisimple ring (rather than a field), and the construction of $\Ext_A(k,k)$ from the resolutions $P_l(A)$ and $P_r(A)$, we can show that $K_l(A)$ (respectively $K_r(A)$) is isomorphic as a left (respectively right) $A$-module to the diagonal part of the graded resolution $P_l(A)$ (respectively $P_r(A)$). We know (see for instance \cite[Section 3]{vdb:nch}) that the differential $d$ of $K(A)$ is the graded sum of the differentials $d \otimes_A\id_k$ and $\id_k\otimes_A d$ naturally extended to $K(A)$. Thus the inclusions $A\otimes_k W_p\otimes_k A \hookrightarrow A\otimes_k E_p\otimes_k A$ constitute an inclusion map $\iota : K(A) \hookrightarrow P(A)$ of weight-graded $A$-bimodule complexes. So we can view the complex $K(A)$ as the diagonal part of the weight-graded resolution $P(A)$, and $A$ is Koszul if and only if $P(A)=K(A)$. The beginning of $P(A)$ coincides with $K(A)$, that is, $E_0=k$, $E_1= V$, $E_2=R$, and the differential $\delta$ of $P(A)$ coincides with the differential $d$ of $K(A)$ in degrees 1 and 2.

For any $A$-bimodule $M$, $\iota$ induces $\tilde{\iota}=M\otimes_{A^e} \iota$ and $\iota^{\ast}=\Hom _{A^e} (\iota, M)$ decomposed in $\tilde{\iota}_p : M\otimes_{k^e} W_p \rightarrow M\otimes_{k^e} E_p$ and $\iota^{\ast}_p : \Hom _{k^e}(E_p,M) \rightarrow \Hom _{k^e}(W_p,M)$. The linear maps
\begin{equation}
  H(\tilde{\iota})_p: \HK_p(A,M) \rightarrow \HH _p(A,M) \ 
\text{and } \  H(\iota^{\ast})_p: \HH ^p(A,M) \rightarrow \HK^p(A,M)
\end{equation}
are isomorphisms for $p=0$ and $p=1$, and for any $p$ if $A$ is Koszul. Since $\iota_p$ is an identity map for $p=0,1,2$, $\tilde{\iota}_p$ and $\iota^{\ast}_p$ are also identity maps for the same $p$'s. Therefore, setting $\tilde{\delta}=M\otimes_{A^e} \delta$ and $\delta^{\ast}=\Hom _{A^e} (\delta, M)$, we have the commutative diagrams
\begin{eqnarray} \label{hom2a}
M\otimes_{k^e} W_3  \stackrel{b^K_3}{\longrightarrow} &  M\otimes_{k^e} R \stackrel{b^K_2}{\longrightarrow} & M\otimes_{k^e} V \nonumber  \\
\tilde{\iota}_3 \downarrow   \ \ \ \ \ \ \ \   &  id \downarrow  \ \ \ \ \ \ &  id \downarrow  \\
M\otimes_{k^e} E_3 \stackrel{\tilde{\delta}_3}{\longrightarrow} & M\otimes_{k^e} R \stackrel{\tilde{\delta}_2}{\longrightarrow} &  M\otimes_{k^e} V   \nonumber
\end{eqnarray}
\\

\begin{eqnarray} \label{hom2b}
\Hom _{k^e}(V, M)  \stackrel{\delta^{\ast}_2}{\longrightarrow} & \Hom _{k^e}(R, M)  \stackrel{\delta^{\ast}_3}{\longrightarrow} & \Hom _{k^e}(E_3, M) \nonumber  \\
id \downarrow   \ \ \ \ \ \ \ \   &  id \downarrow  \ \ \ \ \ \ &  \iota^{\ast}_3 \downarrow  \\
\Hom _{k^e}(V, M) \stackrel{b_K^2}{\longrightarrow} &  \Hom _{k^e}(R, M) \stackrel{b_K^3}{\longrightarrow} &  \Hom _{k^e}(W_3, M)   \nonumber
\end{eqnarray}
where $\tilde{\iota}_3$ is injective and $\iota^{\ast}_3$ is surjective (the ring $k^e\cong \ff^{|\mathcal{Q}_0|^2}$ is semisimple), so that we obtain the following.

\begin{Po} \label{1degree2}
Let $A=T_k(V)/(R)$ be a quadratic $k$-algebra over $\mathcal{Q}$. For any $M$,
\sloppy
\begin{enumerate}[\itshape(i)]
\item  $H(\tilde{\iota})_2: \HK_2(A,M) \rightarrow \HH _2(A,M)$ is surjective with kernel
  isomorphic to $\im(\tilde{\delta}_3)/\im(b^K_3)$,

\item  $H(\iota^{\ast})_2: \HH ^2(A,M) \rightarrow \HK^2(A,M)$ is injective with image isomorphic to
  $\ker(\delta^{\ast}_3)/\im(b_K^2)$.
\end{enumerate}

\end{Po}

We can be more specific when $M=A$, by using the weight grading (Subsection \ref{weight}). Unlike the Koszul differentials $b^K$ and $b_K$, the Hochschild differentials $b^H$ and $b_H$ are not homogeneous for the coefficient weight, but only for the total weight. The grading of $\HH _p(A)$ and $\HH ^p(A)$ for the total weight $t$ is denoted by $\HH _p(A)_t$ and $\HH ^p(A)_t$. Denote the weight of a homogeneous element $a$ of $A$ by $|a|$. Recall that the total weight of a homogeneous $p$-chain $z= a\otimes_{k^e} (a_1\ldots a_p)$ is equal to $t=|a|+ |a_1| + \ldots + |a_p|$, and the total weight of a homogeneous $p$-cochain $f$ mapping $a_1\ldots a_p$ to an element of $A_m$ is equal to $t=m-|a_1| - \ldots - |a_p|$. Then $H(\tilde{\iota})_2$ is homogeneous from the coefficient weight $r$ to the total weight $r+2$, while $H(\iota^{\ast})_2$ is homogeneous from the total weight $r-2$ to the coefficient weight $r$.

\begin{Cr} \label{2degree2}
Let $A=T_k(V)/(R)$ be a quadratic $k$-algebra over $\mathcal{Q}$.
\begin{enumerate}[\itshape(i)]
\item $H(\tilde{\iota})_2$ is an isomorphism from $\HK_2(A)_r$ to $\HH _2(A)_{r+2}$ if $r=0$ and
  $r=1$.

\item  Assume that $A$ is finite dimensional. Let $\max$ be the highest $m$ such that $A_m\neq
  0$.
  Then $H(\iota^{\ast})_2$ is an isomorphism from $\HH ^2(A)_{r-2}$ to $\HK^2(A)_r$ if $r=\max$ and
  $r=\max-1$.
\end{enumerate}

\end{Cr}
\begin{proof}
Denote by $E_{p,m}$ the homogeneous component of weight $m$ of $E_p$. Since $E_{3,2}=0$ and $E_{3,3}=W_3$, both maps $\tilde{\delta}_3$ and $b^K_3$ vanish on the component of total weight 2 of $A\otimes_{k^e}E_3$, while on that of total weight 3, they coincide with the inclusion map of $W_3$ into $V\otimes_{k^e} R$. Then we deduce \textit{(i)} from \textit{(i)} of the proposition.

Under  the assumptions of \textit{(ii)}, if $f: R \rightarrow A_{\max}$, then $b_K^3(f)=0$. Moreover, any other component of $\delta^{\ast}_3(f)$ mapping $E_{3,m}$ to $A_{\max+ m-2}=0$ vanishes as well. Thus $\delta^{\ast}_3(f)=0$, and we conclude by \textit{(ii)} of the proposition. The same proof works if $f: R \rightarrow A_{\max-1}$ since  $\delta^{\ast}_3(f)$ is then reduced to a map $W_3 \rightarrow A_{\max}$ coinciding with $b_K^3(f)$. 
\end{proof}

\setcounter{equation}{0}

\section{A right action on the Koszul calculus} \label{rightfreeness}

This section presents an important tool which we use  in Section 5 to adapt the known definition of
Calabi-Yau algebras due to Ginzburg to the context of quadratic quiver algebras endowed with the Koszul calculus. The idea is to put together two compatible bimodule actions on Koszul chains and cochains : the action of the quadratic quiver algebra $A$ and the action of the associated DG algebra $\tilde{A}$ defined just before Definition \ref{variouskc}.

\subsection{Compatibility} \label{compatibility}

\begin{Lm} \label{compat}
Let $A$ and $B$ be unital associative $\mathbb{F}$-algebras. Let $M$ be an $A$-bimodule (hence the induced $\mathbb{F}$-bimodule is symmetric). Assume that $M$ is a right $B$-module such that the actions of $\mathbb{F}$ induced on $M$ by $A$ and by $B$ are the same. Let $A^e=A\otimes_{\mathbb{F}} A^{op}$ be the enveloping algebra. The following are equivalent.
\begin{enumerate}[\itshape(i)]
\item  Viewing $M$ as a left $A^e$-module, $M$ is an $A^e$-$B$-bimodule.

\item  Viewing $M$ as a right $A^e$-module, the right actions of $A^e$ and $B$ on $M$ commute.

\item  $M$ is an $A$-$B$-bimodule and the right actions of $A$ and $B$ on $M$ commute.
\end{enumerate}

\end{Lm}

The proof is straightforward. Under the assumptions of the lemma and if the equivalent assertions hold, we say that the right action of $B$ on $M$ is \emph{compatible} with the $A$-bimodule $M$.

\begin{Ex} \label{example}
With $B=M=A^e$, $A^e$ is a natural $A^e$-$A^e$-bimodule for the multiplication of the $\mathbb{F}$-algebra $A^e$. Recall that the left $A^e$-module $A^e$ is isomorphic to the $A$-bimodule $A\stackrel{o}{\otimes} A$ for the outer action $(a\otimes b).(\alpha \otimes \beta)= (a\alpha)\otimes (\beta b)$, while the right $A^e$-module $A^e$ is isomorphic to the $A$-bimodule $A\stackrel{i}{\otimes} A$ for the inner action $(\alpha \otimes \beta).(a\otimes b)= (\alpha a) \otimes (b \beta)$.
\end{Ex}

\subsection{DG bimodules over the DG algebra $\tilde{A}$} \label{dgbimodules}

Let $A=T_k(V)/(R)$ be a quadratic $k$-algebra over $\mathcal{Q}$. Fix a unital associative $\mathbb{F}$-algebra $B$ and an $A$-bimodule $M$. We assume that $M$ is a right $B$-module compatible with the $A$-bimodule structure. Then the space $M\otimes_{k^e} W_q$ is a right $B$-module for the action of $b\in B$ on $z=m\otimes_{k^e}x_1 \ldots x_q \in M\otimes_{k^e} W_q$ defined by
$$z.b= (m.b)\otimes_{k^e}x_1 \ldots x_q.$$
It is well-defined since $(\lambda m \mu).b= \lambda (m.b) \mu$ for any $\lambda$ and $\mu$ in $k$. From (\ref{defb}) and \eqref{kolcap}, we check that $b^K$ and  $\me_A \underset{K}{\frown} -$ are $B$-linear. Thus $\HK_{\bullet}(A,M)$ and $\HK^{hi}_{\bullet}(A,M)$ are graded right $B$-modules.

Just before Definition \ref{variouskc}, we have associated to $A$ the $\mathbb{F}$-central DG algebra
$$\tilde{A}=\Hom _{A^e}(K(A), A)\cong \Hom _{k^e}(W_{\bullet},A)$$
whose grading is given by the cohomological degree of cochains, whose differential is $b_K$ and whose multiplication is $\underset{K}{\smile}$. We have also mentioned that $\Hom _{k^e}(W_{\bullet},M)$ and  $M\otimes_{k^e} W_{\bullet}$ are DG bimodules over $\tilde{A}$ for the actions of $\underset{K}{\smile}$ and $\underset{K}{\frown}$ respectively, so that $\HK^{\bullet}(A,M)$ and $\HK_{\bullet}(A,M)$ are graded $\HK^{\bullet}(A)$-bimodules.

For any $k$-bimodule morphism $f: W_p \rightarrow A$, we verify that
$$f \underset{K}{\frown} (z.b)=(f \underset{K}{\frown}z).b\text{ and } (z.b)\underset{K}{\frown} f =(z \underset{K}{\frown}f).b,$$
so that the right action of $B$ on $M\otimes_{k^e} W_{\bullet}$ is compatible with the
$\tilde{A}$-bimodule structure. Therefore the right action of $B$ on $\HK_{\bullet}(A,M)$ and on
$\HK^{hi}_{\bullet}(A,M)$ is  compatible with the structure of $\HK^{\bullet}(A)$-bimodule and of
$\HK_{hi}^{\bullet}(A)$-bimodule respectively. Let us sum up what we have obtained at the level of complexes.

\begin{Po} \label{dgalg1}
Let  $A=T_k(V)/(R)$ be a quadratic $k$-algebra over $\mathcal Q$, let $B$ be  a unital associative
$\mathbb{F}$-algebra and let $M$ be an $A$-bimodule.  We assume that $M$ is a right $B$-module and
that this structure is compatible with the $A$-bimodule structure. Denote by $\modb$ the category of right $B$-modules. Then
\begin{enumerate}[\itshape(i)]
\item the complex $(M\otimes_{k^e} W_{\bullet}, b^K)$ is a complex in $\modb$,
\item  the complex $(M\otimes_{k^e} W_{\bullet}, b^K)$ is a DG-bimodule over the
  $\mathbb{F}$-central DG-algebra $\tilde{A}$,
\item the right action of $B$ and the bimodule action of $\tilde{A}$  on
  $M\otimes_{k^e} W_{\bullet}$ are compatible.
\end{enumerate}
In this situation, following Yekutieli~\cite[Definition
3.8.1]{yeku:dercat}, we  say that $M\otimes_{k^e} W_{\bullet}$ is a DG
$\tilde{A}$-bimodule in the abelian category $\modb$.
\end{Po}

In order to reflect the fact that the right action of $B$ and the bimodule
actions of $\HK^{\bullet}(A)$ and $\HK_{hi}^{\bullet}(A)$ on $\HK_{\bullet}(A,M)$ and
$\HK^{hi}_{\bullet}(A,M)$ respectively are compatible, we shall also say that $\HK_{\bullet}(A,M)$ is a graded $\HK^{\bullet}(A)$-bimodule in
$\modb$, and that $\HK^{hi}_{\bullet}(A,M)$ is a graded $\HK_{hi}^{\bullet}(A)$-bimodule in $\modb$. By Lemma \ref{compat}, it is equivalent to saying that $\HK_{\bullet}(A,M)$ is a graded $\HK^{\bullet}(A)^e$-$B$-bimodule, similarly for $\HK^{hi}_{\bullet}(A,M)$.

Similarly, $\Hom _{k^e}(W_{\bullet},M)$ is a right $B$-module for the action of $b$ on $f: W_p \rightarrow M$ defined by
$$(f.b)(x_1 \ldots x_p)=f(x_1 \ldots x_p).b.$$
Then $b_K$ and $\me_A \underset{K}{\smile} -$ are $B$-linear, so that $\HK^{\bullet}(A,M)$ and  $\HK_{hi}^{\bullet}(A,M)$ are graded right $B$-modules. For $g: W_q \rightarrow A$, we have
$$g \underset{K}{\smile} (f.b)= (g \underset{K}{\smile} f).b \text{ and } (f.b)\underset{K}{\smile} g = (f \underset{K}{\smile} g).b.$$
We obtain an analogue of Proposition \ref{dgalg1}, that is, $\Hom _{k^e}(W_{\bullet},M)$ is a DG $\tilde{A}$-bimodule in $\modb$.

\begin{Po} \label{dgalg2}
  We keep the notation and assumptions of the previous proposition. Then
  \begin{enumerate}[\itshape(i)]
  \item the complex $\Hom _{k^e}(W_{\bullet},M)$ is a complex in $\modb$,

  \item  the complex $\Hom _{k^e}(W_{\bullet},M)$ is a DG $\tilde{A}$-bimodule,

  \item  the right action of $B$ and the bimodule action of $\tilde{A}$ are compatible on
    $\Hom _{k^e}(W_{\bullet},M)$.
  \end{enumerate}

\end{Po}

Therefore $\HK^{\bullet}(A,M)$ is a graded $\HK^{\bullet}(A)$-bimodule in $\modb$, and $\HK_{hi}^{\bullet}(A,M)$ is a graded $\HK_{hi}^{\bullet}(A)$-bimodule in $\modb$.

\subsection{Application to the Koszul complex $K(A)$} \label{application}

Let us specialise to $B=M=A^e$ as in Example \ref{example}. Then $M=A\stackrel{o}{\otimes} A$ is a left $A^e$-module for the outer structure, and a right $A^e$-module for the inner structure. Our aim is to identify the $A$-bimodule complex $K(A)$ with the complex $((A\stackrel{o}{\otimes} A)\otimes_{k^e} W_{\bullet}, b^K)$ endowed with the right action of $A^e$. The statement is the following.

\begin{Po}
Let $A=T_k(V)/(R)$ be a quadratic $k$-algebra over $\mathcal{Q}$.
\begin{enumerate}[\itshape(i)]
\item  For any $q\pgq 0$, the bilinear map
  $\varphi_q : (A\stackrel{o}{\otimes} A)\times W_q \rightarrow A\otimes_k W_q \otimes_k A$ defined
  by
  \[\varphi_q (\alpha \otimes \beta, x_1 \ldots x_q) = \beta \otimes_k (x_1 \ldots x_q) \otimes_k
  \alpha\]
  induces an isomorphism
  $\tilde{\varphi}_q: (A\stackrel{o}{\otimes} A)\otimes_{k^e} W_q \rightarrow A\otimes_k W_q
  \otimes_k A$.

\item The direct sum $\tilde{\varphi}$ of the maps $\tilde{\varphi}_q$ is an isomorphism from the
  complex $((A\stackrel{o}{\otimes} A)\otimes_{k^e} W_{\bullet}, b^K)$ to the Koszul complex
  $(K(A),d)$.

\item  The isomorphism $\tilde{\varphi}$ is right $A^e$-linear.
\end{enumerate}

\end{Po}

\begin{proof}
The $A$-bimodule $A\stackrel{o}{\otimes} A$ is a $k$-bimodule for the actions $\lambda (\alpha \otimes \beta) \mu =\lambda \alpha \otimes \beta \mu$, with $\alpha$ and $\beta$ in $A$, $\lambda$ and $\mu$ in $k$, thus it is a right $k^e$-module for $(\alpha \otimes \beta) (\lambda \otimes \mu) =\mu \alpha \otimes \beta \lambda$. Then it is easy to check that
$$\varphi_q (\mu \alpha \otimes \beta \lambda, x_1  \ldots x_q)= \varphi_q (\alpha \otimes \beta,
\lambda x_1  \ldots x_q \mu),$$ proving the existence of $\tilde{\varphi}_q$.  We define similarly an inverse linear map, therefore
$\tilde{\varphi}_q$ is an isomorphism, which gives \textit{(i)}.

Let us show that $\tilde{\varphi}$ is a  morphism of complexes. From
$$b^K((\alpha \otimes \beta) \otimes_{k^e} x_1 \ldots x_q) =(\alpha \otimes (\beta x_1))\otimes_{k^e} x_2 \ldots x_q +(-1)^q ((x_q\alpha)\otimes \beta) \otimes_{k^e} x_1 \ldots x_{q-1},$$
we get
$$\tilde{\varphi}\circ b^K((\alpha \otimes \beta) \otimes_{k^e} x_1 \ldots x_q)=\beta x_1 \otimes_k x_2 \ldots x_q \otimes_k \alpha +(-1)^q \beta  \otimes_k x_1 \ldots x_{q-1}\otimes_k x_q \alpha$$
whose right-hand side is equal to $d(\beta \otimes_k x_1 \ldots x_{p} \otimes_k \alpha)$, as expected.

Let us prove \textit{(iii)}. Here the $A$-bimodule $A\otimes_k W_q \otimes_k A$ is seen as a right $A^e$-module.
For $z=(\alpha \otimes \beta) \otimes_{k^e} x_1 \ldots x_q$ and $a$, $b$ in $A$, we have 
\begin{align*}
\tilde{\varphi}_q(z.(a\otimes b))&=\tilde{\varphi}_q((\alpha a \otimes b \beta) \otimes_{k^e} x_1 \ldots x_q)\\&=b \beta \otimes_k (x_1  \ldots x_q) \otimes_k \alpha a\\&= \tilde{\varphi}_q(z).(a\otimes b),
\end{align*} therefore $\tilde{\varphi}_q$ is $A^e$-linear.
\end{proof}

So $\tilde{\varphi}$ is an isomorphism from the $A$-bimodule complex $((A\stackrel{o}{\otimes} A)\otimes_{k^e} W_{\bullet}, b^K)$ whose $A$-bimodule structure is the inner one, to the $A$-bimodule complex $K(A)$. Denote by $\abimod$ the category of $A$-bimodules. According to Proposition \ref{dgalg1}, $(A\stackrel{o}{\otimes} A)\otimes_{k^e} W_{\bullet}$ is a DG $\tilde{A}$-bimodule in $\abimod$. We transport this structure via $\tilde{\varphi}$ and we obtain.

\begin{Po} \label{dgalg3}
  Let $A=T_k(V)/(R)$ be a quadratic $k$-algebra over $\mathcal{Q}$. Then the Koszul complex $K(A)$ is a DG $\tilde{A}$-bimodule in $\abimod$.
\end{Po}

This DG bimodule will play an essential role in the generalisations of Calabi-Yau algebras (Sections 4 and 5). Moreover, $H(K(A))$ is a graded $\HK^{\bullet}(A)$-bimodule in $\abimod$, so that $H(\tilde{\varphi}): \HK_{\bullet}(A,A\stackrel{o}{\otimes} A) \rightarrow H(K(A))$ is an isomorphism of graded $\HK^{\bullet}(A)$-bimodules in $\abimod$.

Let us give explicitly the underlying $\tilde{A}$-bimodule structure of the DG $\tilde{A}$-bimodule $K(A)$. Consider $z=(\alpha \otimes \beta) \otimes_{k^e} x_1 \ldots x_q$ in $(A\stackrel{o}{\otimes} A)\otimes_{k^e} W_q$ and $f$ in $\Hom _{k^e}(W_p,A)$, we easily derive from (\ref{kolcap}) that the left action of $f$ on $K(A)$ is defined by
\begin{equation} \label{lcapkocomplex}
f \underset{K}{\frown} (\beta \otimes_k x_1 \ldots  x_q \otimes_k \alpha) = (-1)^{(q-p)p} \beta \otimes_k x_1 \ldots  x_{q-p} \otimes_k f(x_{q-p+1} \ldots x_q)\alpha.
\end{equation}
Analogously, using (\ref{korcap}), we define the right action of $f$ on $K(A)$ by
\begin{equation} \label{rcapkocomplex}
(\beta \otimes_k x_1 \ldots  x_q \otimes_k \alpha) \underset{K}{\frown} f = (-1)^{pq} \beta f(x_1 \ldots x_p)\otimes_k x_{p+1} \ldots  x_q \otimes_k \alpha.
\end{equation}

The fundamental formula (\ref{fundaho}) reduces to
\begin{equation} \label{fundakocomplex}
d(z')= - \me_A\underset{K}{\frown} z' + (-1)^q z' \underset{K}{\frown} \me_A
\end{equation} on $K(A)$,
where $z'=\beta \otimes_k x_1 \ldots  x_q \otimes_k \alpha$, and
\begin{align} \label{eAcap}
\me_A\underset{K}{\frown} z' &=  (-1)^{q-1} \beta \otimes_k x_1 \ldots  x_{q-1} \otimes_k x_q\alpha ,
\\\label{capeA}
z'\underset{K}{\frown} \me_A &=  (-1)^q \beta x_1 \otimes_k x_2 \ldots  x_q \otimes_k \alpha .
\end{align}
The differential $\me_A\underset{K}{\frown}- $ induces a differential, still denoted by $\partial_{\frown}$, on $H(K(A))$. The homology of $(H(K(A)),\partial_{\frown})$ is denoted by $H^{hi}(K(A))$ and is called the higher homology of $K(A)$. Then $H^{hi}(K(A))$ is a graded $\HK_{hi}^{\bullet}(A)$-bimodule in $\abimod$ and $H(H(\tilde{\varphi})): \HK^{hi}_{\bullet}(A,A\stackrel{o}{\otimes} A) \rightarrow H^{hi}(K(A))$ is an isomorphism of graded $\HK_{hi}^{\bullet}(A)$-bimodules in $\abimod$.

\setcounter{equation}{0}

\section{Poincar\'e Van den Bergh duality of preprojective algebras} \label{pvdb}

\subsection{Preprojective algebras} \label{preproalg}

Throughout this section, $\Delta$ is a connected graph whose vertex set and edge set are finite. Following a usual presupposition in the papers devoted to Hochschild (co)homology of preprojective algebras, we assume that the graph $\Delta$ is not labelled, that is, the labels of the edges are all equal to $(1,1)$~\cite[Definition 4.1.9]{benson:repcoh}. In particular, the Dynkin graphs are limited to types ADE, and the Euclidean (or extended) Dynkin graphs are limited to types $\tilde{\textup{A}} \tilde{\textup{D}} \tilde{\textup{E}}$~\cite[Definition 4.5.1]{benson:repcoh}. 

Let $Q$ be a quiver whose underlying graph is $\Delta$. Define a quiver $Q^{\ast}$ whose vertex set is $Q_0$ and whose arrow set is $Q_1^{\ast}=\{a^{\ast};a\in Q_1\}$ where $\mo(a^{\ast})=\mt(a)$ and $\mt(a^{\ast})=\mo(a)$. Let $\ov{Q}$ be the \emph{double quiver} of $Q$, that is, the quiver whose vertex set is $\overline{Q}_0=Q_0$ and whose arrow set is the disjoint union $\overline{Q}_1=Q_1\cup Q_1^{\ast}$. We shall view $(-)^{\ast}$ as an involution of $\overline{Q}_1$.

Let $\ff$ be a field. As before, we denote the ring $\ff Q_0$ by $k$ and the $k$-bimodule $\ff
\overline{Q}_1$ by $V$ and  we identify the graded $k$-algebras $T_k(V)\cong \mathbb{F}\overline{Q}$
(see Subsection \ref{not}).

The \emph{preprojective algebra associated with the graph $\Delta$ over the field $\ff$} is the quadratic $k$-algebra $A(\Delta)$ over $\overline{Q}$ defined by $A(\Delta)=\ff \overline{Q}/(R) $, where the sub-$k$-bimodule $R$ of $\ff \overline{Q}_2$ is generated by 
\[ \sigma_i:=\sum_{\substack{a\in Q_1\\\mt(a)=i}}aa^*-\sum_{\substack{a\in Q_1\\\mo(a)=i}}a^*a=\sum_{\substack{a\in \ov Q_1\\\mt(a)=i}}\varepsilon(a)aa^*\quad\text{ for all $i\in Q_0$}, \]
where $\varepsilon(a)=1$ if $a\in Q_1$, $\varepsilon(a)=-1$ if $a\in Q_1^{\ast}$.

If $Q'$ is another quiver whose underlying graph is $\Delta$, and $R'$ is the sub-$k$-bimodule of $\ff \overline{Q'}_2$ generated by the relations $\sigma _i'=\sum_{\substack{a\in \ov Q'_1\\\mt(a)=i}}\varepsilon(a)aa^*$ for all $i\in Q_0'=Q_0$, then the preprojective algebras $\ff \overline{Q}/(R)$ and $\ff \overline{Q'}/(R')$ are isomorphic, the isomorphism being given by exchanging pairs of arrows $a$ and $a^*$ and changing the sign of one arrow in each pair (see \cite[Remark 2.2(3)]{cbDMV}, or \cite[Lemma 1.3.7]{shaw} for a complete proof in the case of generalised preprojective algebras).  Therefore, according to Subsection \ref{invariance}, the quadratic $k$-algebra $A(\Delta)$ and the (general, higher) Koszul calculus of $A(\Delta)$ depend only on the graph $\Delta$ and not on $Q$, justifying the notation $A(\Delta)$. If $\Delta$ is a tree, $A(\Delta)$ is isomorphic to the preprojective algebra defined without signs (that is, $\varepsilon (a)=1$ for all $a\in\ov Q_1$, as in \cite{es:first,es:second,es:third}).

If $\Delta=\mathrm{A}_1$, then $A(\Delta)=k$. If $\Delta=\mathrm{A}_2$, then $R=\ff \overline{Q}_2$
and $A(\Delta)=\ff Q_0 \oplus \ff \overline{Q}_1$. These quadratic $k$-algebras are Koszul, but they
are the only exceptions among the Dynkin graphs. More precisely,  the following standard result
holds, for which we just give proof references (see also \cite[Corollary 4.3]{bbk}). 
\begin{Po} \label{notkoszul}
Assume that the graph $\Delta$ is distinct from $\mathrm{A}_1$ and $\mathrm{A}_2$. The following are equivalent.
\begin{enumerate}[\itshape(i)]
\item  $\Delta$ is Dynkin of type ADE.

\item  $A(\Delta)$ is not Koszul.

\item  $A(\Delta)$ is finite dimensional.
\end{enumerate}

\end{Po}
\begin{proof}
The equivalence \textit{(i)}$\Leftrightarrow$\textit{(ii)} is treated in \cite{mv:kprepro} if $\Delta$ is a tree, in \cite{green:intro} otherwise. The equivalence \textit{(i)}$\Leftrightarrow$\textit{(iii)} for any Dynkin graph is cited in \cite{mv:intro} as a result by Gelfand and Ponomarev \cite{gp:preproj}.
\end{proof}

Sections \ref{kcquivers} and \ref{rightfreeness} can be applied to preprojective algebras. For example, according to the remark following Lemma \ref{eAcob}, the fundamental 1-cocycle $\me_{A(\Delta)}$ is not a coboundary if $\Delta$ has a loop or if $\car \ff \neq 2$ and $\Delta\neq \mathrm{A}_1$. In the remainder of this section, we often abbreviate $A(\Delta)$ to $A$ and we freely use notations and results from Sections \ref{kcquivers} and \ref{rightfreeness}.

\subsection{The Koszul complex $K(A)$ has length 2} \label{lengthtwo}

If $\Delta=\mathrm{A}_1$, then $K(A)$ has length 0. If $\Delta=\mathrm{A}_2$, then $K(A)$ has
infinite length. However, when $\Delta $ is not Dynkin ADE, the algebra $A=A(\Delta )$ has global
dimension $2$ (this is a consequence of \cite[Proposition 4.2]{bbk}, inspired by manuscript notes of
Crawley-Boevey), and since $k\cong \ff^{|Q_0|}$ is separable, it follows that the minimal projective
$A$-bimodule resolution of $A$, which is $K(A)$ because $A$ is Koszul, has length $2$ (see for
instance \cite[Proposition 3.18]{rr}). 

Actually, the fact that the length of $K(A)$ is $2$ is true  for all graphs $\Delta $ other than
$\mathrm{A}_1$ and $\mathrm{A}_2$, and we now give a unified proof of this.

\begin{Te} \label{w3iszero} Let $A=A(\Delta)$ be a preprojective algebra over $\ff$ with $\Delta\neq \mathrm{A}_1$ and $\Delta\neq \mathrm{A}_2$. Then the Koszul complex $K(A)$ of $A$ has length 2. Consequently, $\HK^p(A,M)\cong \HK_p(A,M)= 0$ for all $A$-bimodules $M$ and all $p \pgq 3$.
\end{Te}
\begin{proof}
From the defining equality (\ref{defw}) of $W_p$, we have $W_p=(W_{p-1}\otk V)\cap(V\otk W_{p-1})$ for all $p\pgq 3$. Moreover $R\neq 0$, therefore it is enough to prove that $W_3=0,$ that is, $(R\otk V)\cap (V\otk R)=0$. For that, we only assume that $\Delta\neq \mathrm{A}_1$. Our goal is to prove that $W_3 \neq 0$ implies $\Delta =\mathrm{A}_2$.

Let $u$ be a non-zero element in $W_3$, viewed as an element in $\ff \ov Q_3.$ There exist vertices $e,f$ in $Q_0$ such that $euf\neq 0$, therefore we may assume that $u$ is in $eW_3f.$ Then $u$ can be written uniquely as
\begin{align*}
  u=\sum_{\substack{a\in e\ov{Q}_1\\\alpha\in e\ov{Q}_1f}}\lambda_\alpha\varepsilon(a)aa^*\alpha
  =\sum_{\substack{b\in f\ov{Q}_1\\\beta\in f\ov{Q}_1e}}\mu_\beta\varepsilon(b)\beta bb^*.
\end{align*}
 We now use the fact that $\ov{Q}_1$ is the disjoint union of $Q_1$ and $Q_1^*$ and the definition of $\varepsilon$ to write
\begin{align*}
  u&=\sum_{\substack{\alpha \in   eQ_1f  \\  a  \in   eQ_1}}\lambda_\alpha aa^*\alpha
  -  \sum_{\substack{\alpha  \in   eQ_1f  \\ a   \in   Q_1e}}\lambda_\alpha a^*a\alpha
  +  \sum_{\substack{    \alpha \in fQ_1e  \\ a   \in   eQ_1}}\lambda_{\alpha^*} aa^*\alpha^*
  -  \sum_{\substack{ \alpha  \in   fQ_1e  \\ a   \in   Q_1e}}\lambda_{\alpha^*}a^*a\alpha^*\\
  &=\sum_{\substack{  \beta \in   eQ_1f  \\  b  \in   fQ_1}}\mu_\beta \beta bb^*
  -  \sum_{\substack{ \beta  \in   eQ_1f  \\ b   \in   Q_1f}}\mu_\beta \beta b^*b
  +  \sum_{\substack{    \beta \in fQ_1e  \\ b   \in   fQ_1}}\mu_{\beta^*} \beta^*bb^*
  -  \sum_{\substack{ \beta  \in   fQ_1e  \\ b   \in   Q_1f}}\mu_{\beta^*}\beta^*b^*b.
  \end{align*}

From these expressions, we obtain the following identities in the path algebra $\ff \ov{Q}$:
\begin{align}
\sum_{\substack{\alpha \in fQ_1e\\  a\in eQ_1}} \lambda_{\alpha^*} aa^*\alpha^*  &=0  \label{rel:W3 1}\\
\sum_{\substack{ \alpha\in eQ_1f\\  a\in Q_1e}} \lambda_\alpha a^*a\alpha &=0  \label{rel:W3 2}\\
\sum_{\substack{ \beta\in eQ_1f\\  b\in fQ_1}} \mu_\beta \beta bb^* &=0  \label{rel:W3 3}\\
\sum_{\substack{ \beta\in fQ_1e\\  b\in Q_1f}}  \mu_{\beta^*} \beta^* b^*b&=0  \label{rel:W3 4}\\
\sum_{\substack{ \alpha\in eQ_1f\\ a \in eQ_1}} \lambda_\alpha aa^*\alpha &=-\sum_{\substack{ \beta\in eQ_1f\\  b\in Q_1f}}  \mu_\beta  \beta b^*b \label{rel:W3 5}\\
\sum_{\substack{ \alpha\in fQ_1e\\  a\in Q_1e}} \lambda_{\alpha^*} a^*a\alpha^*  &=-\sum_{\substack{ \beta\in fQ_1e\\ b \in fQ_1}}  \mu_{\beta^*} \beta^*bb^*  \label{rel:W3 6}
\end{align}
Indeed, identity \eqref{rel:W3 1} follows from the fact that no other path that occurs in the expressions of $u$ ends with two arrows in $Q^*$, and the other identities are obtained from similar arguments.

In the path algebra $\ff \ov Q$, where there are no relations between paths, the identities \eqref{rel:W3 1} to \eqref{rel:W3 4} above are equivalent to
\begin{align}
\forall \alpha\in fQ_1e,\ \forall a\in eQ_1,\ \lambda_{\alpha^*}=0  \label{rel:W3 1'}\\
\forall \alpha\in eQ_1f,\ \forall a\in Q_1e,\ \lambda_{\alpha}=0  \label{rel:W3 2'}\\
\forall \beta\in eQ_1f,\ \forall a\in fQ_1,\ \lambda_{\beta}=0  \label{rel:W3 3'}\\
\forall \beta\in fQ_1e,\ \forall a\in Q_1f,\ \lambda_{\beta^*}=0.  \label{rel:W3 4'}
  \end{align}

We have assumed that $u\neq 0$, so that either there exists $\alpha\in eQ_1f$ such that $\lambda_\alpha\neq 0$ or there exists $\alpha\in fQ_1e$ such that $\lambda_{\alpha^*}\neq 0,$ using the first expression of $u.$ We separate the two cases.

Assume that there exists $\alpha\in eQ_1f$ such that $\lambda_\alpha\neq 0$.  Then it follows from identity \eqref{rel:W3 2'} that $Q_1e$ is empty. From \eqref{rel:W3 5}, for all $a\in eQ_1$, there exist $\beta\in eQ_1f$ and $b\in Q_1f$ such that $\lambda_\alpha aa^*\alpha=-\mu_\beta\beta b^*b.$ Hence $\beta=a=b=\alpha$ and therefore  $eQ_1=\set{\alpha}=eQ_1f$ and  $\mu_\beta=-\lambda_\alpha\neq 0.$ From \eqref{rel:W3 3'}, it follows that $fQ_1$ is empty. Finally, \eqref{rel:W3 5} becomes
\[ \lambda_\alpha \alpha\alpha^*\alpha=\lambda_\alpha\alpha\alpha^*\alpha+\lambda_\alpha\sum_{\substack{b\in Q_1f\\b\neq\alpha} }\alpha b^*b \] so that $\sum_{\substack{b\in Q_1f\\b\neq\alpha} }\alpha b^*b =0$ and hence $Q_1f=\set{\alpha}.$

We have proved that $Q_1e=\emptyset= fQ_1$ so that in particular $e\neq f$, and that $Q_1f=\set{\alpha}=eQ_1=eQ_1f$. Finally, $Q= e\stackrel{\alpha}\leftarrow f$ and $\Delta=A_2$.

In the case where there exists $\alpha\in fQ_1e$ such that $\lambda_{\alpha^*}\neq 0$, a similar proof using \eqref{rel:W3 1'}, \eqref{rel:W3 4'} and \eqref{rel:W3 6} shows that  $Q= f\stackrel{\alpha}\leftarrow e$ and $\Delta=A_2$.
\end{proof}

As an immediate consequence of Proposition \ref{kocommutation} and Theorem \ref{w3iszero}, we obtain that in the Koszul calculus of $A(\Delta)$, the Koszul cup product is graded commutative and the Koszul cap product is graded symmetric. The precise statement is the following. 
\begin{Cr} \label{cupcapsym}
Let $A=A(\Delta)$ be a preprojective algebra over $\ff$ with $\Delta\neq \mathrm{A}_1$ and $\Delta\neq \mathrm{A}_2$. We consider an $A$-bimodule $M$. For any $\alpha \in \HK^{\bullet}(A,M)$, $\beta \in \HK^{\bullet}(A)$ and $\gamma \in \HK_{\bullet}(A)$, we have the identities
\begin{align} \label{kocupbrazeroprepro}
[\alpha, \beta]_{\underset{K}{\smile}} &=0,
\\ \label{kocapbrazeroprepro}
[\alpha, \gamma]_{\underset{K}{\frown}} &=0.
\end{align}
\end{Cr}
 
The same conclusion holds if $\Delta= \mathrm{A}_1$ (obvious) and if $\Delta= \mathrm{A}_2$ (because $A$ is Koszul).

\subsection{Duality in Koszul (co)homology of preprojective algebras} \label{dualityprepro}

There is a remarkable duality between Koszul homology and cohomology for preprojective algebras. This duality is realised as a cap action by a Koszul 2-chain $\omega_0\in A\otimes_{k^e} R$ defined for any graph $\Delta$ by
$$\omega_0=\sum_{i\in Q_0}e_i\otimes_{k^e} \sigma_i=\sum_{i\in Q_0}e_i\ot \sigma_i.$$
From $\sigma_i = \sum_{a\in \ov Q_1,\,\mt(a)=i}\varepsilon(a)aa^*$, we get
\begin{equation} \label{defomega0}
\omega_0 = \sum_{a\in \ov{Q}_1} 1\otimes_{k^e} \varepsilon(a)aa^*= -\sum_{a\in \ov{Q}_1} 1 \otimes_{k^e}\varepsilon(a)a^*a.
\end{equation}
Then it is easy to check that $\omega_0$ is a Koszul 2-cycle. Being homogeneous of weight 0, $\omega_0$ is not a 2-boundary whenever $\Delta\neq \mathrm{A}_1$.

In the following statement, we need the DG algebra $\tilde{A}$ and the DG $\tilde{A}$-bimodules $\Hom _{k^e}(W_{\bullet},M)$ and $M\otimes_{k^e} W_{\bullet}$, introduced just before Definition \ref{variouskc}.

\begin{Te} \label{complexduality}
Let $A=A(\Delta)$ be a preprojective algebra over $\ff$ with $\Delta\neq \mathrm{A}_1$ and $\Delta\neq \mathrm{A}_2$. Consider the Koszul 2-cycle $\omega_0=\sum_{i\in Q_0}e_i\ot \sigma_i \in A\otimes_{k^e} R$. For each Koszul $p$-cochain $f$ with coefficients in an $A$-bimodule $M$, we define the Koszul $(2-p)$-chain $\theta_M (f)$ with coefficients in $M$ by
\begin{equation} \label{deftheta}
\theta_M (f) = \omega_0 \underset{K}{\frown} f.
\end{equation}
Then the equalities
\begin{equation} \label{formultheta}
\theta_{M\otimes_A N} (f\underset{K}{\smile} g) = \theta_M (f) \underset{K}{\frown} g = f \underset{K}{\frown} \theta_N (g)
\end{equation}
hold for any Koszul cochains $f$ and $g$ with coefficients in $A$-bimodules $M$ and $N$ respectively.

Moreover the linear map $\theta_M: \Hom _{k^e}(W_{\bullet},M) \rightarrow M\otimes_{k^e} W_{2-\bullet}$ is an isomorphism of DG $\tilde{A}$-bimodules.

It follows that $H(\theta_M): \HK^{\bullet}(A,M) \rightarrow \HK_{2-\bullet}(A,M)$ is an isomorphism of graded $\HK^{\bullet}(A)$-bimodules and that $H(H(\theta_M)) : \HK_{hi}^{\bullet}(A,M) \rightarrow \HK^{hi}_{2-\bullet}(A,M)$ is an isomorphism of graded $\HK_{hi}^{\bullet}(A)$-bimodules. 
\end{Te}
\begin{proof}
First we show that $f \underset{K}{\frown} \omega_0 =  \omega_0 \underset{K}{\frown} f$ for all $f \in \Hom _{k^e}(W_p,M)$. Using the definition of $\omega_0$ and the equalities (\ref{defomega0}), (\ref{kolcap}) and (\ref{korcap}), we obtain for $p=0,1,2,$
\begin{align*}
  f \underset{K}{\frown} \omega_0 & = \sum_{i\in Q_0}(f(1)e_i) \otimes_{k^e} \sigma_i= \sum_{i\in Q_0}(e_if(1)) \otimes_{k^e} \sigma_i = \omega_0 \underset{K}{\frown} f, \\
  f \underset{K}{\frown} \omega_0 & = -\sum_{a\in \ov{Q}_1} \varepsilon(a)(f(a^*)1) \otimes_{k^e} a=\sum_{a\in \ov{Q}_1} \varepsilon(a)(1f(a))\otimes_{k^e} a^* = \omega_0 \underset{K}{\frown} f, \\
 f \underset{K}{\frown} \omega_0 &= \sum_{i\in Q_0}(e_if(\sigma_i)) \otimes_{k^e} 1= \sum_{i\in Q_0}(f(\sigma_i)e_i) \otimes_{k^e}1 = \omega_0 \underset{K}{\frown} f.
\end{align*}
Next, for $f \in \Hom _{k^e}(W_p,M)$ and $g \in \Hom _{k^e}(W_q,M)$, we have
$$\theta_{M\otimes_A N} (f\underset{K}{\smile} g)= \omega_0 \underset{K}{\frown} (f\underset{K}{\smile} g)= (\omega_0 \underset{K}{\frown} f) \underset{K}{\frown} g=(f \underset{K}{\frown} \omega_0 ) \underset{K}{\frown} g=f \underset{K}{\frown} (\omega_0  \underset{K}{\frown} g),$$
providing equalities (\ref{formultheta}). Therefore $\theta_M: \Hom _{k^e}(W_{\bullet},M) \rightarrow M\otimes_{k^e} W_{2-\bullet}$ is a morphism of graded $\tilde{A}$-bimodules, where $\tilde{A}$ is just considered as a graded algebra. It remains to examine what happens for the Koszul differentials.

Assuming $M=A$ in the equalities (\ref{formultheta}), we derive 
$$\theta_N ([f,g]_{\underset{K}{\smile}}) = [\theta_A (f),g]_{\underset{K}{\frown}} = [f,\theta_N (g)]_{\underset{K}{\frown}}.$$ 
Combining $\theta_N ([\me_A,g]_{\underset{K}{\smile}}) = [\me_A,\theta_N (g)]_{\underset{K}{\frown}}$ with $b_K=- [\me_A, -]_{\underset{K}{\smile}}$ and $b^K = -[\me_A, -]_{\underset{K}{\frown}}$,
we deduce that $\theta_M$ is a morphism of complexes, thus a morphism of DG $\tilde{A}$-bimodules.

We prove that $\theta_M$ is an isomorphism by giving an inverse map $\eta : M\otimes_{k^e} W_{2-\bullet}\rightarrow  \Hom _{k^e}(W_{\bullet},M)$. We define $\eta_p:M\otimes_{k^e} W_{2-p}\rightarrow  \Hom _{k^e}(W_p,M)$ for $p=0,1,2,$ by
\begin{align*}
 & \eta_0(m\otimes_{k^e} \sigma_i)(e_j)=\delta_{ij}\, e_jme_i\\
 & \eta_1(m \otimes_{k^e} a)(b)=\delta_{ba^*}\,\varepsilon(b)\, \mt(b)\, m\, \mo(b)\quad\text{for any arrows $a$ and $b$ of $\ov{Q}$}\\
 & \eta_2(m\otimes_{k^e} e_i)(\sigma_j)=\delta_{ij}\, e_jme_i
\end{align*} where $\delta$ is the Kronecker symbol. It is routine to verify that these linear maps are well-defined and form an inverse map for $\theta_M$.

Finally the isomorphism $H(\theta_M)$ of graded $\HK^{\bullet}(A)$-bimodules satisfies
$$H(\theta_M)(\overline{\me}_A \underset{K}{\smile} \alpha) = \overline{\me}_A \underset{K}{\frown} H(\theta_M)(\alpha)$$
for all $\alpha \in \HK^{\bullet}(A,M)$. Therefore $H(\theta_M)$ is a  morphism of complexes for higher (co)homologies. Taking higher (co)homologies, we get a $\HK^{\bullet}_{hi}(A)$-bimodule isomorphism
\[H(H(\theta_M)): \HK^{\bullet}_{hi}(A,M) \rightarrow \HK^{hi}_{2-\bullet}(A,M). 
\]
\end{proof}
By analogy with the Poincar\'e duality in singular (co)homology~\cite{hat:algtop} and with the Van den Bergh duality in Hochschild (co)homology~\cite{vdb:dual,tl:bvcy}, we say that the isomorphism
$$H(\theta_M)= \overline{\omega}_0 \underset{K}{\frown} - : \HK^{\bullet}(A,M) \rightarrow \HK_{2-\bullet}(A,M)$$
is a Poincar\'e Van den Bergh duality for Koszul (co)homology, of fundamental class $\overline{\omega}_0$, where $\overline{\omega}_0\in \HK_2(A)_0$. In the next subsection, we extract from this duality a generalisation of the 2-Calabi-Yau property.

Unless $\Delta$ has no loop and $\car \ff = 2$, the class $\overline{\me}_A\in \HK^1(A)_1$ is non-zero, hence
$$H(\theta_A)(\overline{\me}_A)= \overline{\omega}_0 \underset{K}{\frown} \overline{\me}_A = \partial_{\frown} (\overline{\omega}_0)$$
is non-zero in $\HK_1(A)_1$. Consequently, the fundamental class $\overline{\omega}_0$ of the Poincar\'e Van den Bergh duality is not a cycle for the higher Koszul homology, so that the isomorphism $H(H(\theta_M))$ cannot be naturally expressed as a cap action.

The class $H(\theta_A)(\overline{\me}_A)$ is the class of the Koszul 1-cycle $\omega_0 \underset{K}{\frown} \me_A$ where
$$\omega_0 \underset{K}{\frown} \me_A = \sum_{a\in \ov{Q}_1} \varepsilon(a) a\otimes_{k^e} a^* = \sum_{a\in Q_1}  (a\otimes_{k^e} a^*- a^*\otimes_{k^e} a).$$

It is interesting to view the last element as the image by the canonical linear map $can: V \otk V \rightarrow V \otimes_{k^e} V$ of the element
$$w=\sum_{a\in Q_1}  (a\otimes_{k} a^*- a^*\otimes_{k} a)= \sum_{a\in \ov{Q}_1} \varepsilon(a)aa^* \in R \subseteq V \otk V.$$
In the identification $V \otk V\cong \ff \ov Q _2$, $V \otimes_{k^e} V$ is identified with the subspace of cycles of length 2 and the map $can$ is identified with the projection whose kernel is the space spanned by the non-cyclic paths. Since $R$ is generated by the cycles $\sigma_i$, we can make the identification $\omega_0 \underset{K}{\frown} \me_A = w$. The element $w$ was defined in~\cite[Proposition 8.1.1]{cbeg:quiver} as a representative of a bi-symplectic 2-form $\omega$. Bi-symplectic 2-forms were introduced by Crawley-Boevey, Etingof and Ginzburg as an essential ingredient of the Hamiltonian reduction in noncommutative geometry~\cite{cbeg:quiver}; they are related to the double Poisson algebras defined by Van den Bergh~\cite{vdb:double}.

\begin{Rm}
  Assume that $\Delta = \mathrm{A}_2$, so that $A$ is defined by the quiver  \[\xymatrix@C=10pt{\ov Q&&
    0\ar@/^/[rr]^{a}&&1\ar@/^/[ll]^{a^*}}\] subject to the relations $\sigma_0=-a^*a$, $\sigma_1=aa^*$. Then the statement of Theorem \ref{complexduality} is valid in a weaker form, namely the isomorphisms involved are only morphisms. Moreover, $\theta_M$ is bijective only in degree $q$, $0\leq q \leq 2$, with the same inverse $\eta_q$. More generally, for any $p\geq 1$, the Koszul ($2p$)-cycle
  $$ \omega_{p-1}= e_1 \otimes (aa^*)^p-e_0\otimes (a^*a)^p$$
  provides a morphism $\omega_{p-1}\capk -$ which is bijective only in degree $q$, $0\leq q \leq 2p$. From that, we deduce an isomorphism $\overline{\omega}_{p-1}\capk -$ from $\HK^q(A,M)$ to $\HK_{2p-q}(A,M)$ for $0<q<2p$. Varying $p$, we obtain the following duality and $2$-periodicity
  \begin{align*} \label{2periodicity}
  \HK^q(A,M)&\cong \HK_q(A,M)\cong \HK^1(A,M),\ q\ \mathrm{odd}\geq 1 \\
  \HK^q(A,M)&\cong \HK_q(A,M)\cong \HK^2(A,M),\ q\ \mathrm{even}\geq 2.
\end{align*}  
Using this for $M=A$, it is straightforward to compute explicitly the restricted Koszul calculus of $A$. We leave the details to the reader. Notice that, since $A$ is Koszul, we have the same duality and $2$-periodicity for Hochschild (co)homology, recovering a known result as a consequence of remarkable isomorphisms due to Eu and Schedler~\cite[Theorem 2.3.27, Theorem 2.3.47]{eusched:cyfrob}, here applied to~\cite[Example 2.3.10, Corollary 2.1.13]{eusched:cyfrob}.

  \end{Rm}

\subsection{Deriving an adapted 2-Calabi-Yau property} \label{subsec:2cyproperty}

Let $A=A(\Delta)$ be a preprojective algebra over $\ff$ with $\Delta\neq \mathrm{A}_1$ and $\Delta\neq \mathrm{A}_2$. Let $M$ be an $A$-bimodule. Assume that $B$ is a unital associative algebra such that $M$ is a right $B$-module compatible with the $A$-bimodule structure (see Subsection \ref{compatibility}). Denote by $\modb$ the category of right $B$-modules. Recall that $\tilde{A}$ denotes the DG algebra $(\Hom _{k^e}(W_{\bullet},A), b_K,\underset{K}{\smile})$.

According to Subsection \ref{dgbimodules}, $M\otimes_{k^e} W_{\bullet}$ and $\Hom _{k^e}(W_{\bullet},M)$ are DG $\tilde{A}$-bimodules in $\modb$. Moreover, $\HK_{\bullet}(A,M)$ and $\HK^{\bullet}(A,M)$ are graded $\HK^{\bullet}(A)$-bimodules in $\modb$. Finally, $\HK^{hi}_{\bullet}(A,M)$ and $\HK_{hi}^{\bullet}(A,M)$ are graded $\HK_{hi}^{\bullet}(A)$-bimodules in $\modb$.
\begin{Lm} \label{Blinearity}
The map $\theta_M: \Hom _{k^e}(W_{\bullet},M) \rightarrow M\otimes_{k^e} W_{2-\bullet}$ is an isomorphism of DG $\tilde{A}$-bimodules in $\modb$. Moreover, $H(\theta_M): \HK^{\bullet}(A,M) \rightarrow \HK_{2-\bullet}(A,M)$ is an isomorphism of graded $\HK^{\bullet}(A)$-bimodules in $\modb$, and $H(H(\theta_M)) : \HK_{hi}^{\bullet}(A,M) \rightarrow \HK^{hi}_{2-\bullet}(A,M)$ is an isomorphism of graded $\HK_{hi}^{\bullet}(A)$-bimodules in $\modb$. 
\end{Lm}
\begin{proof}
It is enough to prove that $\theta_M: f \mapsto \omega_0 \underset{K}{\frown} f$ is $B$-linear. For a $k$-bimodule morphism $f: W_p \rightarrow M$, $z=a\otimes_{k^e}x_1 \ldots x_q \in A\otimes_{k^e} W_q$ and $b\in B$, we verify the identities
$$(f.b) \underset{K}{\frown} z =(f \underset{K}{\frown}z).b \text{ and } z\underset{K}{\frown} (f.b) =(z \underset{K}{\frown}f).b.$$
The first one uses the fact that  the right actions of $A$ and $B$ on $M$ commute, while the second one uses the fact  that $M$ is an $A$-$B$-bimodule (see \textit{(iii)} in Lemma \ref{compat}). Applying the second one to $z= \omega_0$, we obtain that $\theta_M$ is $B$-linear.
\end{proof}

We specialise this lemma to $M=B=A^e$ and, using the isomorphism $\tilde{\varphi}$ in Subsection \ref{application}, we identify $A^e\otimes_{k^e} W_{\bullet}$ with $K(A)$ to get the next
proposition.
\begin{Po} \label{specialduality}
Let $A=A(\Delta)$ be a preprojective algebra over $\ff$ with $\Delta\neq \mathrm{A}_1$ and $\Delta\neq \mathrm{A}_2$. The map
$$\theta_{A^e}: \Hom _{k^e}(W_{\bullet},A^e) \rightarrow K(A)_{2-\bullet}$$
is an isomorphism of DG $\tilde{A}$-bimodules in $\abimod$. Moreover,
$$H(\theta_{A^e}): \HK^{\bullet}(A, A^e) \rightarrow H_{2-\bullet}(K(A))$$
is an isomorphism of graded $\HK^{\bullet}(A)$-bimodules in $\abimod$.
\end{Po}

The homology of $K(A)$ is isomorphic to $A$ in degree 0, and to $0$ in degree 1, hence we obtain a generalisation of the 2-Calabi-Yau property, formulated as follows.
\begin{Te} \label{2cyproperty}
Let $A=A(\Delta)$ be a preprojective algebra over $\ff$ with $\Delta\neq \mathrm{A}_1$ and $\Delta\neq \mathrm{A}_2$. Then the $\HK^{\bullet}(A)^e$-$A^e$-bimodules $\HK^{\bullet}(A,A^e)$ and $H_{2-\bullet}(K(A))$ are isomorphic. In particular, we have the following.
\begin{enumerate}[\itshape(i)]
\item  The $A$-bimodule $\HK^2(A, A^e)$ is isomorphic to the $A$-bimodule $A$.

\item  $\HK^1(A,A^e)= 0$.

\item  The $A$-bimodule $\HK^0(A,A^e)$ is isomorphic to the $A$-bimodule $H_2(K(A))$.
\end{enumerate}

\end{Te}

Since $H_1(K(A)) \cong \HK^1(A,A^e)= 0$, the higher Koszul differentials vanish. Therefore $H^{hi}_p(K(A))\cong H_p(K(A))$, $\HK_{hi}^p(A, A^e)\cong \HK^p(A, A^e)$ and $H(H(\theta_{A^e}))\cong H(\theta_{A^e})$.

From the generator $1\otimes_k1$ of the $A$-bimodule $H_0(K(A))$, we draw from \textit{(i)} a generator of the free $A$-bimodule $\HK^2(A, A^e)$ defined as the class of $f: R\rightarrow A \stackrel{o}{\otimes} A$ with $f(\sigma_i)=e_i\otimes e_i$ for any $i$.

In \textit{(iii)}, the $A$-bimodules are never $0$ when $\Delta$ is Dynkin of types $\mathrm{ADE}$ since $A$ is not Koszul in this case. This situation is drastically different from the 2-Calabi-Yau property defined by Ginzburg in terms of the Hochschild cohomology spaces $\HH ^p(A,A^e)$~\cite[\S3.2]{vg:cy}. In Ginzburg's definition, $\HH^p(A,A^e)=0$ for all $p<2$. 

\setcounter{equation}{0}

\section{Generalisations of Calabi-Yau algebras} \label{gency}

\subsection{Duality for Koszul complex Calabi-Yau algebras} \label{dualitygency}

From Theorem \ref{2cyproperty}, we are led to introduce a general definition in the framework  of
quiver algebras with homogeneous quadratic relations (see  Section \ref{kcquivers}). The notation introduced in
Section \ref{kcquivers} stands throughout. We are interested in quadratic $k$-algebras $A=T_k(V)/(R)$
over a finite quiver $\mathcal{Q}$ as defined in Subsection \ref{not}, and in the Koszul calculus of $A$ as presented in the remainder of Section \ref{kcquivers}.
Note that $\mathcal{Q}_1=\overline{Q}_1$ if we want to specialise to preprojective algebras.

For the definition of the bounded derived category $\mathcal{D}^b (\mathcal{C})$ of an abelian category $\mathcal{C}$, we refer to~\cite[Chapter 10]{weib:homo}. Recall that $\abimod$ denotes the category of $A$-bimodules.

\begin{Df} \label{kcy2}
Let $A=T_k(V)/(R)$ be a quadratic $k$-algebra over $\mathcal{Q}$. Let $n\pgq 0$ be an integer. We say that $A$ is Koszul complex Calabi-Yau (Kc-Calabi-Yau) of dimension $n$, or $n$-Kc-Calabi-Yau, if
\begin{enumerate}[\itshape(i)]
\item  the Koszul bimodule complex $K(A)$ of $A$ has length $n$, and

\item  $\RHom _{A^e}(K(A), A^e) \cong K(A)[-n]$ in $\mathcal{D}^b (\abimod)$.
\end{enumerate}

\end{Df}

Property \textit{(ii)} is equivalent to saying that there is an $A$-bimodule quasi-isomorphism from $\Hom _{k^e}(W_{\bullet},A^e)$ to $K(A[-n]$. According to Theorem \ref{w3iszero} and Proposition \ref{specialduality}, a preprojective algebra $A(\Delta)$ over $\ff$ with $\Delta\neq \mathrm{A}_1$ and $\Delta\neq \mathrm{A}_2$ is Kc-Calabi-Yau of dimension $2$. In fact, the isomorphism $\theta_{A^e}$ induces  an isomorphism $\RHom _{A^e}(K(A), A^e) \rightarrow K(A)[-2]$ in $\mathcal{D}^b (\abimod)$.

Let us recall Ginzburg's definition of Calabi-Yau algebras~\cite[Definition 3.2.3]{vg:cy} as reformutated by Van den Bergh~\cite[Definition 8.2]{vdb:cypot}. We shall apply this definition to quadratic quiver algebras by considering it as $\ff$-algebras.

\begin{Df} \label{cydef}
  An associative $\ff$-algebra $A$ is said to be Calabi-Yau of dimension $n$ if

  (i) $A$ is homologically smooth, that is, $A$ has a bounded resolution by finitely generated projective $A$-bimodules,

  (ii) $\RHom _{A^e}(A, A^e) \cong A[-n]$ in $\mathcal{D}^b (\abimod)$.
\end{Df}

Definition \ref{kcy2} is a true generalisation of Definition \ref{cydef} for quadratic quiver algebras. If $\Delta$ is Dynkin of type ADE, $A(\Delta)$ is not Calabi-Yau in Ginzburg's definition since $A(\Delta)$ is not homologically smooth in this case (the minimal projective resolution of $A(\Delta)$ has infinite length). However, the two definitions coincide if $A$ is Koszul.
\begin{Po} \label{defequiv}
Let $A=T_k(V)/(R)$ be a quadratic $k$-algebra over $\mathcal{Q}$. Assume that $A$ is Koszul. Then $A$ is $n$-Kc-Calabi-Yau if and only if $A$ is $n$-Calabi-Yau.
\end{Po}
\begin{proof}
Assume that $A$ is $n$-Kc-Calabi-Yau. Property \textit{(i)} and the fact that $A$ is Koszul show that $A$ is homologically smooth. Furthermore, $K(A)\cong A$ in $\mathcal{D}^b (\abimod)$. Thus $\RHom _{A^e}(A, A^e) \cong A[-n]$ in $\mathcal{D}^b (\abimod)$, and we recover Definition \ref{cydef}.

Assume that $A$ is $n$-Calabi-Yau. We know that $n$ is equal to the projective dimension of the $A$-bimodule $A$~\cite{vdb:dual} which in turn is equal to the length of a minimal projective resolution of $A$ (see for instance~\cite{berger:dimension}). Hence $K(A)$ has length $n$ and $K(A)\cong A$ in $\mathcal{D}^b (\abimod)$, which allows us to conclude that $A$ is $n$-Kc-Calabi-Yau.
\end{proof}
If the graph $\Delta$ is not Dynkin ADE, we know that $A(\Delta)$ is Koszul (Proposition \ref{notkoszul}), thus we recover the fact that $A(\Delta)$ is 2-Calabi-Yau~\cite{cbeg:quiver,boc:3cy}.

In Subsection \ref{degree2}, we have seen that $K(A)$ and the minimal projective resolution $P(A)$ coincide up to the homological degree 2. Therefore, if $n\in\set{0,1}$ and if $A$ is $n$-Calabi Yau or $n$-Kc-Calabi-Yau, then $P(A)=K(A)$ so that $A$ is Koszul, and it follows that the two definitions are equivalent when $n\in \{0,1\}$.
\begin{Po} \label{dimzero}
Let $A=T_k(V)/(R)$ be a quadratic $k$-algebra over $\mathcal{Q}$. Then $A$ is Calabi-Yau of dimension 0 if and only if $\mathcal{Q}_1=\emptyset$.
\end{Po}

We leave the proof as an exercise. If $A$ is Calabi-Yau of dimension 1, then $R=0$, that is, $A=T_k(V)\cong \ff \mathcal{Q}$ with $\mathcal{Q}_1\neq \emptyset$. It is indeed 1-Calabi-Yau if $\mathcal{Q}$ has only one vertex and one loop, but we have not yet found other examples when $\mathcal{Q}$ is connected.  

If $A$ is $n$-Calabi-Yau, the Van den Bergh duality theorem states that the vector spaces $\HH ^p(A,M)$ and $\HH _{n-p}(A,M)$ are isomorphic~\cite{vdb:dual}. From Definition \ref{kcy2}, we draw an analogous duality theorem for Koszul homology/cohomology.
\begin{Te} \label{duality2}
Let $A$ be a Koszul complex Calabi-Yau algebra of dimension $n$ over $\mathcal{Q}$. Then for any $A$-bimodule $M$, the vector spaces $\HK^p(A,M)$ and $\HK_{n-p}(A,M)$ are isomorphic.
\end{Te}
\begin{proof}
Denote by $\catvs$ the category of $\ff$-vector spaces. For any $A$-bimodule $M$, the 
left derived functor $M \stackrel{L}\otimes_{A^e}-$ and the right derived functor $\RHom _{A^e}(-,M)$ are defined from $\mathcal{D}^b (\abimod)$ to $\mathcal{D}^b (\catvs)$~\cite[Chapter 10]{weib:homo}.

Our proof is based on a natural transformation depending on an $A$-bimodule $M$. Let $F: \abimod \rightarrow \catvs$ be the functor $F: N \mapsto \Hom _{A^e}(N,M)$ where $M$ and $N$ are seen as right $A^e$-modules. Specialising to $M=A^e$ in $F$, we define a functor $G: \abimod \rightarrow \abimod$. Let $H: \abimod \rightarrow \catvs$  be the functor $H: N'  \mapsto M\otimes_{A^e} N'$ where $N'$ is viewed as a left $A^e$-module. Then we define a linear map
$$\phi_M: M\otimes_{A^e} \Hom _{A^e}(N,A^e) \rightarrow \Hom _{A^e} (N, M)$$
by $\phi_M (m \otimes_{A^e} g)(x)= m.g(x)$ for $m\in M$, $g\in \Hom _{A^e}(N,A^e)$ and $x\in N$. This map is functorial in $N$, defining a natural transformation $\phi_M: H \circ G \Rightarrow F$.

If the $A$-bimodule $P$ is projective and finitely generated,
$$\phi_M: M\otimes_{A^e} \Hom _{A^e}(P,A^e) \rightarrow \Hom _{A^e} (P, M)$$
is an isomorphism. It is standard, see e.g.~\cite[Proposition (8.3) (c)]{ksb:cohomgrp}. Then for any bounded chain complex $C$ of finitely generated projective $A$-bimodules, $\phi_M$ induces in $\mathcal{D}^b (\catvs)$ an isomorphism 
\begin{equation} \label{dercat4}
M \stackrel{L}\otimes_{A^e} \RHom _{A^e}(C, A^e) \cong \RHom _{A^e}(C,M).
\end{equation}

Applying it to $C=K(A)$ and using \textit{(ii)} in Definition \ref{kcy2}, we get an isomorphism 
$$\RHom _{A^e}(K(A),M) \cong M \stackrel{L}\otimes_{A^e} K(A)[-n]$$
in $\mathcal{D}^b (\catvs)$. Taking homology, we deduce that
$\HK^p(A,M) \cong \HK_{n-p}(A,M)$ as vector spaces.
\end{proof}

\subsection{Koszul complex Calabi-Yau algebras versus Calabi-Yau algebras} \label{sec:versus}

Recall that if $\Delta\neq \mathrm{A}_1$ and $\Delta\neq \mathrm{A}_2$, then $A(\Delta)$ is $2$-Kc-Calabi-Yau. But observe that if $A(\Delta)$ is moreover $2$-Calabi-Yau, then $A(\Delta)$ is Koszul. We are led to the following conjecture.
\begin{Cn} \label{conj}
Let $A=T_k(V)/(R)$ be a quadratic $k$-algebra over $\mathcal{Q}$. If $A$ is $n$-Calabi-Yau and $n$-Kc-Calabi-Yau, then $A$ is Koszul. In other words, if $A$ is not Koszul, the properties $n$-Calabi-Yau and $n$-Kc-Calabi-Yau are not simultaneously true.
\end{Cn}
\begin{Po} \label{repconj1}
Let $A=T_k(V)/(R)$ be a quadratic $k$-algebra over $\mathcal{Q}$. Conjecture \ref{conj} holds if $n\ppq 3$.
\end{Po}
\begin{proof}
Assume that $A$ is $n$-Calabi-Yau and $n$-Kc-Calabi-Yau. We can assume that $n\pgq 2$. Since $\HK^p(A,A^e)\cong \HH ^p(A,A^e)=0$ when $p=0$ and $p=1$, we have $H_n(K(A))\cong H_{n-1}(K(A)) = 0$, hence $A$ is Koszul if $n=2$. When $n=3$, we also have $H_1(K(A))=0$ because  the complex $K(A)$ is always exact in degree $1$, therefore $A$ is also Koszul in this case.
\end{proof}

\subsection{Strong Kc-Calabi-Yau algebras} \label{subsec:versus}

\begin{Df} \label{fundclass2}
Let $A$ be an $n$-Kc-Calabi-Yau algebra over $\mathcal{Q}$. The image $c \in \HK_n(A)$ of the unit $1$ of the algebra $A$ by the isomorphism $\HK^0(A) \cong \HK_n(A)$ in Theorem \ref{duality2} is called the fundamental class of the $n$-Kc-Calabi-Yau algebra $A$.
\end{Df}

We shall now define strong Kc-Calabi-Yau algebras. For this, we need the DG algebra $\tilde{A}=\Hom _{A^e}(K(A),A))$ and we work with DG $\tilde{A}$-bimodules in $\abimod$, as defined in Section \ref{rightfreeness} and recalled below. The point is that $K(A)$ is such a DG $\tilde{A}$-bimodule in $\abimod$ (Proposition \ref{dgalg3}). Denote by $\mathcal{C}(\tilde{A},\abimod)$ and $\mathcal{C}(\tilde{A},\catvs)$ the category of DG $\tilde{A}$-bimodules in $\abimod$ and $\catvs$ respectively~\cite{yeku:dercat}. Remark that a DG $\tilde{A}$-bimodule in $\catvs$ is just a DG $\tilde{A}$-bimodule.

\begin{Df}
A DG $\tilde{A}$-bimodule $C$ in $\abimod$ is a chain complex in $\abimod$ (as usual, $C$ can be viewed as a cochain complex) endowed with a DG $\tilde{A}$-bimodule structure such that the bimodule actions of $A$ and $\tilde{A}$ are compatible. 
\end{Df}

For any $A$-bimodule $M$, $\Hom _{A^e} (C, M)$ is a (cochain) DG $\tilde{A}$-bimodule in $\catvs$ (in $\abimod$ when $M=A^e$) for the following actions
$$(f.u)(x)=(-1)^p u(x.f), \ \ (u.f)(x)=u(f.x)$$
where $f:A\otimes_k W_p \otimes_k A \rightarrow A$, $u: C_q \rightarrow M$ and $x \in C_{p+q}$. Note that $x.f$ and $f.x$ are in $C_q$ by the graded actions of $\tilde{A}$ on $C$. If $C=K(A)$, we recover the cup actions, that is, $f.u=f\underset{K}{\smile}u$ and $u.f=u\underset{K}{\smile}f$. In particular, $\Hom _{A^e}(K(A), A^e)$ is a DG $\tilde{A}$-bimodule in $\abimod$.

Similarly, for any cochain DG $\tilde{A}$-bimodule $C'$ in $\abimod$, $M\otimes_{A^e} C'$ is a cochain DG $\tilde{A}$-bimodule in $\catvs$ for the following actions
$$f.(m \otimes_{A^e} u)=m \otimes_{A^e} (f.u), \ \ (m \otimes_{A^e} u).f=m \otimes_{A^e} (u.f)$$
where $f \in \tilde{A}$, $m\in M$ and $u \in C'$.

The bounded derived categories $\mathcal{D}^b(\tilde{A},\abimod)$ and $\mathcal{D}^b(\tilde{A},\catvs)$ are defined in~\cite[Definition 7.2.7, Definition 7.3.3]{yeku:dercat}. Unfortunately, it is not clear to us if the functors $\Hom _{A^e} (-, M): \mathcal{C}^b(\tilde{A},\abimod) \rightarrow \mathcal{C}^b(\tilde{A},\catvs)$ and $M\otimes_{A^e} - : \mathcal{C}^b(\tilde{A},\abimod) \rightarrow \mathcal{C}^b(\tilde{A},\catvs)$ can be derived. Note that the first one takes values in $\mathcal{C}^b(\tilde{A},\abimod)$ when $M=A^e$.

\begin{Df} \label{strongcy2}
Let $A$ be a Kc-Calabi-Yau algebra of dimension $n$. Then $A$ is said to be strong $n$-Kc-Calabi-Yau  if the derived functor of the endofunctor $\Hom _{A^e}(-, A^e)$ of $\mathcal{C}^b(\tilde{A},\abimod)$ exists and if $\RHom _{A^e}(K(A), A^e) \cong K(A)[-n]$ in $\mathcal{D}^b(\tilde{A},\abimod)$.
\end{Df}

The preprojective algebras of connected graphs distinct from $\mathrm{A}_1$ and $\mathrm{A}_2$ are strong $2$-Kc-Calabi-Yau algebras if they satisfy the first property in this definition. In fact, using Proposition  \ref{specialduality}, $\theta_{A^e}$ provides then an isomorphism $\RHom _{A^e}(K(A), A^e) \rightarrow K(A)[-2]$ in $\mathcal{D}^b(\tilde{A},\abimod)$.

\begin{Te} \label{strongduality2}
Let $A$ be a Kc-Calabi-Yau algebra of dimension $n$  over $\mathcal{Q}$, with fundamental class $c$. We assume that $A$ is strong Kc-Calabi-Yau and that the derived functors of the functors $\Hom _{A^e}(-,A)$ and $A\otimes_{A^e} -$ from $\mathcal{C}^b(\tilde{A},\abimod)$ to $\mathcal{C}^b(\tilde{A},\catvs)$ exist. Then
$$c \underset{K}{\frown} - : \HK^{\bullet}(A) \rightarrow \HK_{n-\bullet}(A)$$
is an isomorphism of $\HK^{\bullet}(A)$-bimodules, inducing an isomorphism of $\HK^{\bullet}_{hi}(A)$-bimodules from $\HK^{\bullet}_{hi}(A)$ to $\HK^{hi}_{n-\bullet}(A)$. For  all $\alpha \in \HK^p(A)$, we have $c \underset{K}{\frown} \alpha = (-1)^{np}\alpha \underset{K}{\frown} c$.
\end{Te}

\begin{proof}
Following the proof of Theorem \ref{duality2}, we are interested in the morphism of cochain complexes  
$$\phi_M: M\otimes_{A^e} \Hom _{A^e}(C,A^e) \rightarrow \Hom _{A^e} (C, M),$$
when the bounded chain complex $C$ of $A$-bimodules is moreover a DG $\tilde{A}$-bimodule in
$\abimod$. We prove now that $\phi_M$ is a morphism of DG $\tilde{A}$-bimodules in
$\catvs$, that is, a morphism in the category $\mathcal{C}^b(\tilde{A},\catvs)$ whose objects are viewed as cochain complexes. For this, we need only prove that $\phi_M$ is a morphism of $\tilde{A}$-bimodules. Let us prove that $\phi_M$ is left $\tilde{A}$-linear, the right linearity being similar. For $f : A \otimes_kW_p \otimes_k A \rightarrow A$, $u : C_q \rightarrow A^e$ and $x \in C_{p+q}$, we have
$$\phi_M(f.(m \otimes_{A^e} u))(x)=\phi_M(m \otimes_{A^e} (f.u))(x)=m.((f.u)(x))=(-1)^p m.(u(x.f)),$$
while $f.(\phi_M(m \otimes_{A^e} u))(x)=(-1)^p \phi_M(m \otimes_{A^e} u)(x.f)=(-1)^p m.(u(x.f))$, which is what we want. 

Continuing as in the proof of Theorem \ref{duality2}, the functors $F$, $G$ and $H$ induce functors on the complexes with enriched structures. Precisely, $F$ and $G$ are now functors from $\mathcal{C}^b(\tilde{A},\abimod)$ to $\mathcal{C}^b(\tilde{A},\catvs)$, and $H$ is now an endofunctor of $\mathcal{C}^b(\tilde{A},\abimod)$. Under these notations, $\phi_M$ defines a natural transformation $\phi_M: H \circ G \Rightarrow F$.

We specialise to $M=A$. The assumptions in the theorem show that the derived functors of $F$, $G$ and $H$ exist, so that we can derive the natural transformation $\phi_A$~\cite{yeku:dercat}. Then for any bounded chain complex DG $\tilde{A}$-bimodule $C$ in $\abimod$ formed by finitely generated projective $A$-bimodules, we obtain an isomorphism
\begin{equation} \label{dercat6}
\phi_A: A \stackrel{L}\otimes_{A^e} \RHom _{A^e}(C, A^e) \cong \RHom _{A^e}(C,A)
\end{equation}
in $\mathcal{D}^b(\tilde{A},\catvs)$. Applying this to $C=K(A)$ and using Definition \ref{strongcy2}, we get 
$$\RHom _{A^e}(K(A),A) \cong A \stackrel{L}\otimes_{A^e} K(A)[-n]$$
in $\mathcal{D}^b(\tilde{A},\catvs)$. Taking homology, we deduce an isomorphism $\HK^{\bullet}(A) \cong \HK_{n-\bullet}(A)$ of graded bimodules over the graded algebra $H(\tilde{A})=\HK^{\bullet}(A)$. Denote this isomorphism by $\psi$.

The fact that $\psi$ is a morphism of graded $\HK^{\bullet}(A)$-bimodules translates as
\begin{equation} \label{psimorphism}
\psi (\alpha \underset{K}{\smile} \beta)= \psi (\alpha) \underset{K}{\frown} \beta = (-1)^{np} \alpha \underset{K}{\frown} \psi(\beta)
\end{equation} for any $\alpha\in \HK^p(A)$ and $\beta \in \HK^{\bullet}(A)$. In accordance with Definition \ref{fundclass2}, define $c \in \HK_n(A)$ by $c=\psi(1)$ where $1\in \HK^0(A)$ is the unit of $A$. Applying identities (\ref{psimorphism}) to the trivial equalities $\alpha=1 \underset{K}{\smile} \alpha = \alpha  \underset{K}{\smile} 1$, we obtain
\begin{equation} \label{ccommut}
\psi(\alpha) = c \underset{K}{\frown} \alpha= (-1)^{np} \alpha \underset{K}{\frown} c.
\end{equation}
Finally $\psi$ is a morphism of complexes for higher (co)homologies since we have
$$\psi(\overline{\me}_A \underset{K}{\smile} \alpha) = (-1)^n \, \overline{\me}_A \underset{K}{\frown} \psi (\alpha).$$
Then $H(\psi) :\HK^{\bullet}_{hi}(A) \rightarrow \HK^{hi}_{n-\bullet}(A)$ is an isomorphism of $\HK^{\bullet}_{hi}(A)$-bimodules.
\end{proof}
Except in some particular cases, $\overline{\me}_A\in \HK^1(A)$ is non-zero, so that
$$\psi(\overline{\me}_A)= (-1)^n \, \overline{\me}_A \underset{K}{\frown} c= (-1)^n \partial_{\frown} (c)$$
is non-zero in $\HK_{n-1}(A)$. Therefore $c\in \HK_n(A)$ is not a cycle for higher Koszul homology and the isomorphism $H(\psi)$ cannot be naturally expressed as a cap action. As suggested by the preprojective algebras (Subsection \ref{dualityprepro}), the class $\psi(\overline{\me}_A)$ should be of interest for further investigations.

It is also interesting to remark that the identities (\ref{psimorphism}) involving the isomorphism $\psi$ imply that the graded algebra $\HK^{\bullet}(A)$ is commutative if and only if the graded $\HK^{\bullet}(A)$-bimodule $\HK_{\bullet}(A)$ is symmetric. As seen in Corollary \ref{cupcapsym}, we have a stronger result for the preprojective algebras.

\section{Koszul calculus of the preprojective algebras of Dynkin ADE type}
\label{sec:explicit}

We shall determine in this section the Koszul calculus and the higher Koszul calculus of any
non-Koszul preprojective algebra $A$, that is, an algebra of type A, D or E with at least $3$
vertices.

We first give some general facts and notation. 
\begin{enumerate}[(N1)]
\item \label{fact:dims} We shall use the dimensions of the Hochschild cohomology and homology spaces of $A$ which can be obtained
in all characteristics as a consequence of the work of Etingof, Eu and Schedler in \cite[Theorem 3.2.7]{eusched:cyfrob} and 
\cite{eteu:ade}. In particular, by \cite[Lemma 3.2.17]{eusched:cyfrob} the centre of $A$ is independent of the characteristic of
$\ff$. Bases of the Hochschild (co)homology spaces in characteristic zero induce free subsets of
the Hochschild (co)homology spaces in positive  characteristic, but there may be some extra basis
elements in some cases. 
\item \label{fact:cap} We know from Corollary \ref{cupcapsym} that the cup product on the Koszul cohomology of a
preprojective algebra is graded commutative and that the cap product is graded symmetric. Moreover,
it follows from Theorem \ref{complexduality} that the cap product can be obtained from the cup
product. Indeed, if $f\in\HK^p(A)$ and $x\in \HK_q(A)$, we have 
\begin{equation}\label{eq:cap from cup}
f\capk x=\theta_A(f\cupk \theta_A^{-1}(x))= \theta_A((-1)^{pq}\theta_A^{-1}(x)\cupk f)=(-1)^{pq}
  x\capk f.
\end{equation}
\item \label{fact:free up to coboundaries} Let $X$ and $Y$ be $\nn$-graded spaces and let $f:X\rightarrow Y$ be a homogeneous
  map of degree $1$. Let $y_1,\ldots,y_p$ be elements of pairwise different degrees. Then if
  $\sum_{i=1}^p y_i \in\im f$, at least one of the $y_i$ is in $\im f$. We shall use this in the
  following context. The differentials $b_K^1$ and $b_K^2$ are homogeneous of weight $1$. If we have
  a set of cocycles of pairwise different coefficient weights, that are not coboundaries, then they
  are linearly independent up to coboundaries, that is, they represent linearly independent
  cohomology classes. This also applies if some of the elements have the same weight but we already
  know that these elements are linearly independent up to coboundaries. 
\item \label{fact:anti-alg} 

We shall use the
  map $\aa\colon A\rightarrow A$ constructed as follows.
  Let $A$ be a preprojective algebra over a graph $\Delta$; let $\ov Q$ be its quiver.  Consider the map
  $\ov Q_1\rightarrow \ov Q_1$ that sends $a$ to $a^*$. It induces an anti-automorphism $\aa$ of $A$ such that $\aa(e_i)=e_i$ for all $i\in Q_0$
  (since $\aa$ sends the relation $\sigma_i=\sum_{\substack{a\in\ov
      Q_1\\\mt(a)=i}}\varepsilon(a)aa^*$ to itself).
\item \label{fact:nakayama} We shall be using the Nakayama automorphism $\nu$ of $A$ defined by Brenner, Butler and King for all preprojective algebras of Dynkin ADE type \cite[Section 4]{bbk}. 

In order to describe it, we need the Nakayama permutation $\ov{\nu }$ on the set of vertices of $\Delta $. It is known that  $\ov{\nu }=\id$ if $\Delta $ is $\mathrm{D}_n$ with $n$ even or if $\Delta $ is $\mathrm{E}_7$ or $\mathrm{E}_8$, and that otherwise $\ov{\nu }$ is induced by the unique graph automorphism of order $2$.  

If $\alpha $ is an arrow in $\ov{Q}_1$, let $\beta $ be the unique arrow from $\ov{\nu }(\mo(\alpha ))$ to $\ov{\nu }(\mt(\alpha ))$.  The Nakayama automorphism of \cite{bbk} is described as follows: 
\[ \nu (\alpha )=
\begin{cases}
\beta &\text{if $\alpha \in Q_1^*$ or $\beta \in Q_1^*$}\\-\beta &\text{if $\alpha \in Q_1$ and $\beta \in Q_1$}.
\end{cases}
 \]
Given a basis of a selfinjective quiver algebra $A=\ff Q/I$ consisting of paths and containing a
basis $\set{\pi _i\,;\,i\in Q_0}$ of the socle of $A$, an explicit construction of an associative
non-degenerate bilinear form on $A$ was given in \cite[Proposition 3.15]{Z} (see also Subsection
\ref{subsec:comparison hochschild}). We have  chosen in each case such a basis so that $\nu $ is the
Nakayama automorphism corresponding to this bilinear form (characterised on the arrows $\alpha $ in $\ov{Q}_1$
by $y\alpha  =\pi _{\mo(\alpha )}\Leftrightarrow \nu (\alpha )y=\pi _{\mt(\alpha )}$ for all the  basis elements $y$).

\item \label{fact:dfn cochain} 

  When we define a cochain $f\in\Hom_{k^e}(X,A)$ with $X\in\set{k,V,A}$, it will be implicit
that if $f(x)$ is not defined for some $x\in X$ then $f(x)=0$.

\item\label{fact:dual notation} For any cochain $f\in \Hom_{k^e}(W_p,A)$, we shall set $\check
  f=\theta_A(f) \in A\ot_{k^e}W_{2-p}$.

\item \label{fact:preproj over other fields} Finally, given a Dynkin graph $\Delta$ and a ring $L$, we shall denote by $\Lambda_L$ the
  preprojective algebra of $\Delta$ over $L$, so that $A=\Lambda_\ff$.
\end{enumerate}

\subsection{Koszul calculus for preprojective algebras of type A}

The preprojective algebra $A$ of type $A_n$ is  defined by the quiver 
\[ \xymatrix@C=10pt{\ov Q&&
0\ar@/^/[rr]^{a_0}&&1\ar@/^/[rr]^{a_1}\ar@/^/[ll]^{a_0^*}&&2\ar@/^/[rr]^{a_2}\ar@/^/[ll]^{a_1^*}&&\ar@/^/[ll]^{a_2^*}\cdots\ar@/^/[rr]^{a_{n-3}}&&n-2\ar@/^/[ll]^{a_{n-3}^*}\ar@/^/[rr]^{a_{n-2}}&&n-1\ar@/^/[ll]^{a_{n-2}^*}} \]
 subject to the relations
\[\begin{cases}
\sigma_0=-a_0^*a_0\\
\sigma_i=a_{i-1}a_{i-1}^*-a_i^*a_i\qquad 1\ppq i\ppq n-2\\
\sigma_{n-1}=a_{n-2}a_{n-2}^*
\end{cases}\]

The Nakayama automorphism  $\nk$ of $A$ described in \ref{fact:nakayama} is given by $\nk(e_i)=e_{n-1-i}$, $\nk(a_i)=a_{n-2-i}^*$ and $\nk(a_i^*)=a_{n-2-i}$. 

Erdmann and Snashall have given in \cite{es:first} a basis $B$ of $A$. We shall only need the sets $e_iBe_i$,  $e_iBe_{i+1}$ and $e_{i+1}Be_i$, which can be rewritten as follows: 
set $\ma=\pe{\frac{n-1}{2}}$; then 
\begin{align*}
&e_iBe_i=\set{(a_i^*a_i)^\ell;0\ppq \ell\ppq \min(i,n-1-i)}\text{ for $0\ppq i\ppq n-1$, with $(a_i^*a_i)^0=e_i$ for all $i$,}\\
&e_{i+1}Be_i=\set{a_i(a_i^*a_i)^\ell;0\ppq \ell\ppq \min(i,n-2-i)} \quad\text{ and }\quad
  e_iBe_{i+1}=\aa(e_{i+1}Be_i).
\end{align*}

For each $i$, $Ae_i$ contains precisely one basis element of maximal length $n-1$, which is 
\[ \pi_i=
\begin{cases}
a_{n-2}\cdots a_i(a_i^*a_i)^i&\text{if $i<\ma$}\\
a_\ma(a_\ma^*a_\ma)^\ma&\text{if $i=\ma$ and $n$ is even}\\
(a_\ma^*a_\ma)^\ma&\text{if $i=\ma$ and $n$ is odd}\\
a_{n-i-1}^*\cdots a_{i-1}^*(a_{i-1}a_{i-1}^*)^{n-1-i}&\text{if $i>\ma$.}
\end{cases} 
 \] They form a basis of the socle of $A$.

\subsubsection{The Koszul cohomology and homology spaces in type A}

The spaces $\HK^0(A)=\HH^0(A)=Z(A)$ and $\HK^1(A)=\HH^1(A)$ are known from \cite{es:first}. 
Therefore we only need to compute $\HK^2(A)$. Recall our assumption that $n\pgq 3$; then by
Theorem~\ref{w3iszero} all the elements in $\Hom_{k^e}(R,A)$ are
cocycles. Moreover, using Theorem \ref{complexduality}
and  \cite[Theorem 3.2.7]{eusched:cyfrob}, we have $\dim \HK^2(A)=\dim \HK_0(A)=\dim\HH_0(A)=n$. Since every element in $\im b_K^2$ has coefficient weight at least $1$,  the $n$ cocycles $h_i$ defined by $h_i(\sigma_j)=\delta_{ij}e_i$ for all $j$ are linearly independent modulo  $\im b_K^2$. It follows that they form a basis of $\HK^2(A)$. 

Combining with the results from \cite{es:first}, we have the following result.

\begin{Po} Let $A$ be a preprojective algebra of type A$_n$.
 
  A basis of $\HK^0(A)$ is given by the set $\set{z_\ell;0\ppq \ell \ppq \ma}$ with $z_0=1$ and
  $z_\ell=\sum_{i=1}^{n-2}(a_i^*a_i)^\ell=\sum_{i=\ell}^{n-1-\ell}(a_i^*a_i)^\ell=z_1^\ell$ for
  $1\ppq \ell \ppq \ma.$

\sloppy  A basis of $\HK^1(A)$ is given by the set
  $\set{\ov\zeta_\ell;0\ppq \ell \ppq n-2-\ma}$, where $\zeta_\ell\in\Hom_{k^e}(V,A)$ is defined by
  $\zeta_\ell(a_i)=a_i(a_i^*a_i)^\ell$  for all $i$ (or for $\ell\ppq i\ppq n-2-\ell$).
  
  A basis of $\HK^2(A)$ is given by the set
  $\set{\ov h_i;0\ppq i\ppq n-1}$ where $h_i\in\Hom_{k^e}(R,A)$ defined by $h_i(\sigma_j)=\delta_{ij}e_i$ for all $j$.
\end{Po}

As a consequence of Theorem \ref{complexduality}, we obtain
bases of the Koszul homology spaces.

\begin{Cr}\label{cor:bases hom type A}
\sloppy  A basis of $\HK_0(A)$ is given by the set
  $\set{\ov\ch_i;0\ppq i\ppq n-1}$ where $\ch_i=e_i\ot e_i$.

  A basis of $\HK_1(A)$ is given by the
  set $\set{\ov\ct_\ell;0\ppq \ell\ppq n-\ma-2}$ where $\ct_\ell=\sum_{i=0}^{n-2}a_i(a_i^*a_i)^\ell\ot a_i^*=\sum_{i=\ell}^{n-2-\ell}a_i(a_i^*a_i)^\ell\ot a_i^*$.
  A basis of $\HK_2(A)$ is given by the set
  $\set{\cz_\ell;0\ppq \ell\ppq \ma}$ where $\cz_\ell=\sum_{i=0}^{n-1}(a_i^*a_i)^\ell\ot
    \sigma_i$.
\end{Cr}

Note that $\cz_0=\omega_0$ is the fundamental class.

\subsubsection{Cup and cap products}

We know from Corollary \ref{cupcapsym} that the cup product on $\HK^{\bullet}(A)$ is graded-commutative. The following
result gives all the non zero cup products of elements in $\HK^{\bullet}(A)$.

\begin{Po}  Let $A$ be a preprojective algebra of type A$_n$.
Up to graded commutativity,   the non zero cup products in $\HK^{\bullet}(A)$ are given by
  \begin{align*}
&z_0\cupk \ov f
  =\ov f\quad\text{ for all $\ov f\in\HK^{\bullet}(A)$}\\
    &z_{\ell_1}\cupk z_{\ell_2}=
               z_{\ell_1+\ell_2}\quad\text{ if $\ell_1+\ell_2\ppq \ma$}
      \\
    &z_{\ell_1}\cupk \ov \zeta_{\ell_2}
      =
              \ov \zeta_{\ell_1+\ell_2}\quad\text{ if $\ell_1+\ell_2\ppq n-\ma-2$}
       \end{align*}
\end{Po}

\begin{proof} The first cup product is clear and the other cup products in the statement only involve
in $\HH^{0}(A)$ and $\HH^1(A)$, therefore they are known from \cite{es:first}.

 The basis elements of $\HK^{2}(A)$ have coefficient weight $0$, and $b_K^2$ is homogeneous of weight 1, therefore any element that has positive
 coefficient weight must be a coboundary. The other cup products (that all vanish) follow from this.
\end{proof}

We now deduce the cap products from \eqref{eq:cap from cup}.

\begin{Cr}\label{cor:cap type A}
Up to graded symmetry,  the  non zero cap products  are the following.
 \begin{align*}
& z_0\capk x
  =x\quad\text{ for all $x\in\HK_{\bullet}(A)$}\\
    & z_{\ell_1}\capk  \cz_{\ell_2}=
               \cz_{\ell_1+\ell_2}\quad\text{ if $\ell_1+\ell_2\ppq \ma$}
      \\
    & z_{\ell_1}\capk \ov \ct_{\ell_2}
      =
              \ov \ct_{\ell_1+\ell_2}\quad\text{ if $\ell_1+\ell_2\ppq n-\ma-2$}
       \end{align*}
\end{Cr}

\subsubsection{Higher Koszul cohomology and homology}

We start with a lemma giving the cohomology class of the fundamental $1$-cocycle.

\begin{Lm}
  The cohomology class of $\me_A$ is equal to the cohomology class of $2\zeta_0.$
\end{Lm}

\begin{proof} Let $\zeta^*_0\in\Hom_{k^e}(V,A)$ be the cocycle defined by $\zeta^*_0(a_i^*)=a_i^*$
for all $i\in Q_0$.
Since $\me_A=\zeta_0+\zeta_0^*$, we must prove that $\zeta_0^*-\zeta_0$ is a coboundary.

Consider $v=\sum_{i=0}^{n-2}\sum_{j=i}^{n-2}e_i\in\bigoplus_{i\in Q_0}e_iAe_i\cong\Hom_{k^e}(k,A)$. Then $b_K^1(v)=\zeta_0^*-\zeta_0$, as required.
\end{proof}

As a consequence, the complex defining the higher Koszul cohomology is 
\[ 0\rightarrow \HK^0(A)\xrightarrow {\partial _\smile^1}\HK^1(A)\xrightarrow {\partial _\smile^2}\HK^2(A)\rightarrow 0\cdots \]
with $\partial _\smile^1(z_\ell)=2\ov\zeta_{\ell}$ for $0\ppq \ell\ppq n-\ma-2$ and $\partial _\smile^2=0.$ We then have the following higher Koszul cohomology.

\begin{Po} \label{hkcA} Let $A$ be a preprojective algebra of type A$_n$.
  If $\car(\ff)=2$, then $\HK^{\bullet}_{hi}(A)=\HK^{\bullet}(A).$

  If $\car(\ff)\neq 2$ and $n$ is even, then 
  \begin{align*}
    \HK^2_{hi}(A)&=\HK^2(A)\\
    \HK^p_{hi}(A)&=0\text{ if $p\neq 2$.}
  \end{align*} Finally, if $\car(\ff)\neq 2$ and $n$ is odd, then 
  \begin{align*}
    \HK^0_{hi}(A)&=\HK^0(A)_{2\ma}\text{ has dimension $1$ and is spanned by $z_{\ma}$}\\
    \HK^2_{hi}(A)&=\HK^2(A)\\
    \HK^p_{hi}(A)&=0\text{ if $p\neq 0$ and $p\neq 2$.}
  \end{align*}
\end{Po}

Higher Koszul homology can then be deduced using duality (Theorem \ref{complexduality}).

\begin{Cr}\label{cor:higher type A}
  If $\car(\ff)=2$, then $\HK_{\bullet}^{hi}(A)=\HK_{\bullet}(A).$

  If $\car(\ff)\neq 2$ and $n$ is even, then
  \begin{align*}
    \HK_0^{hi}(A)&=\HK_0(A)\\
    \HK_p^{hi}(A)&=0\text{ if $p\neq 0$.}
  \end{align*} Finally, if $\car(\ff)\neq 2$ and $n$ is odd, then
  \begin{align*}
    \HK_0^{hi}(A)&=\HK_0(A)\\
    \HK_2^{hi}(A)&=\HK_2(A)_{2\ma}\text{ has dimension $1$ and is spanned by $\cz_{\ma}$}\\
    \HK_p^{hi}(A)&=0\text{ if $p\neq 0$ and $p\neq 2$.}
  \end{align*}
\end{Cr}

\subsection{Koszul calculus for preprojective algebras of type D}

The preprojective algebra $A$ of type $D_n$ is defined by the quiver 
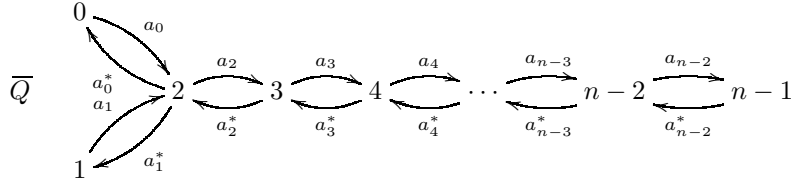
\[ \xymatrix@C=10pt@R=15pt{
&0\ar@/^/[rrd]^{a_0}&&\\
\ov Q&&&2\ar@/^/[llu]^{a_0^*}\ar@/^/[rr]^{a_2}\ar@/^/[lld]^{a_1^*}&&3\ar@/^/[rr]^{a_3}\ar@/^/[ll]^{a_2^*}&&4\ar@/^/[rr]^{a_4}\ar@/^/[ll]^{a_3^*}&&\ar@/^/[ll]^{a_4^*}\cdots\ar@/^/[rr]^{a_{n-3}}&&n-2\ar@/^/[ll]^{a_{n-3}^*}\ar@/^/[rr]^{a_{n-2}}&&n-1\ar@/^/[ll]^{a_{n-2}^*}\\
&1\ar@/^/[rru]^{a_1}
} \]
 subject to the relations 
\begin{align*}
&\sigma_0=-a_0^*a_0&&\sigma_i=a_{i-1}a_{i-1}^*-a_i^*a_i\qquad 3\ppq i\ppq n-2\\
&\sigma_1=-a_1^*a_1&&\sigma_{n-1}=a_{n-2}a_{n-2}^*\\
&\sigma_2=a_0a_0^*+a_1a_1^*-a_2^*a_2
\end{align*}

 The Nakayama automorphism  $\nk$ of $A$ described in \ref{fact:nakayama} is given by $\nk(e_i)=e_{i}$, $\nk(a_i)=-a_{i}^*$ and $\nk(a_i^*)=a_{i}$ if $n$ is even of if $n$ is odd and $i\geqslant2$, and it exchangess $e_0$ and $e_1$, $a_0$ and $-a_1$, and $a_0^*$ and $a_1^*$ if $n$ is odd.

Eu has given in \cite{eu:product} a basis $B$ of $A$. Set $\md =\pe{\frac{n-2}{2}}$ and $u=n-\md -2.$  
  We shall only need bases of the $e_jAe_i$ when $i$ and $j$ are equal or  adjacent vertices, which can be rewritten as follows: 
\begin{align*}
e_0Be_0&=\set{
         (a_0^*a_1a_1^*a_0)^\ell;0\ppq\ell\ppq \md }\\
e_1Be_1&=\set{
         (a_1^*a_0a_0^*a_1)^\ell;0\ppq\ell\ppq \md }\\
e_iBe_i&=\set{
         (a_i^*a_i)^\ell ;0\ppq \ell\ppq n-i-1}\\&\quad\cup\set{
                                                   (a_i^*a_i)^\ell a_{i-1}\cdots a_2a_1a_1^*a_2^*\cdots a_{i-1}^*;0\ppq \ell\ppq n-i-1} \quad\text{if } i\pgq 2\\
e_2Be_i&=\set{
         (a_2^*a_2)^\ell a_i;0\ppq \ell\ppq n-3}\text{ for }i\in\set{0,1}\\
e_{i+1}Be_i&=\set{
             (a_{i+1}^*a_{i+1})^\ell a_{i};0\ppq \ell\ppq
             n-i-2}\\&\quad\cup\set{
                       (a_{i+1}^*a_{i+1})^\ell a_{i}\cdots
                       a_2a_1a_1^*a_2^*\cdots a_{i-1}^*;0\ppq \ell\ppq n-i-2}\quad\text{if } i\pgq
                       2\\
e_iBe_{i+1}&=\aa(e_{i+1}Be_i), \quad e_iBe_2=\aa(e_2Be_0)\text{ for }i\in\set{0,1}.
\end{align*}

For each $i$, $Ae_i$ contains precisely one basis element of maximal length $2(n-2)$, which is 
\begin{align*}
&\pi_0=-a_0^*a_2^*a_3^*\cdots a_{n-2}^*a_{n-2}\cdots a_3a_2a_0=
\begin{cases}
-(a_0^*a_1a_1^*a_0)^{\md }&\text{ if $n$ is even}\\
-a_1^*a_0(a_0^*a_1a_1^*a_0)^{\md }&\text{ if $n$ is odd}\\
\end{cases}
\\
&\pi_1=a_1^*a_2^*a_3^*\cdots a_{n-2}^*a_{n-2}\cdots a_3a_2a_1=\begin{cases}
(a_1^*a_0a_0^*a_1)^{\md }&\text{ if $n$ is even}\\
a_0^*a_1(a_1^*a_0a_0^*a_1)^{\md }&\text{ if $n$ is odd}\\
\end{cases}
\\
&\pi_i=(a_i^*a_i)^{n-i-1}a_{i-1}\cdots a_2a_1a_1^*a_2^*\cdots a_{i-1}^*
  \quad\text{if }2\ppq i\ppq n-2\\
&\pi_{n-1}=a_{n-2}a_{n-3}\cdots a_1a_1^*a_2^*\cdots a_{n-2}^*.
\end{align*}
They form a basis of the socle of $A$.

\subsubsection{The Koszul cohomology and homology spaces in type D}

The centre of $A$ does not depend on the characteristic of $\ff$ by fact \ref{fact:dims} and  was computed in
\cite{eu:product} in characteristic $0$. A basis of $\HH^1(\Lambda _\cc)$ was determined in \cite{eu:product}.  Moreover,  $\dim\HK^1(A)=\dim\HH^1(A)$, which is equal to $\dim\HH^1(\Lambda_\cc)=n$
if $\car(\ff)\neq 2$ and to $\dim\HH^1(\Lambda_\cc)+m$ if $\car(\ff)= 2$ by \cite{eusched:cyfrob}.

It also follows from Theorem \ref{complexduality} and \cite{eusched:cyfrob} that
$\dim\HK^2(A)=\dim\HK_0(A)=\dim\HH_0(A)$, which is equal to $n-\md-2$ if  $\car(\ff)\neq 2$ and to
$n-2$  if $\car(\ff)=2$.

In order to give bases of the $\HK^p(A)$ for $p=0,1,2$, we define the following cochains:
\begin{itemize}
\item the elements $z_0=1$ and
  $z_\ell=(a_0^*a_1a_1^*a_0)^\ell+(a_1^*a_0a_0^*a_1)^\ell+\sum_{i=2}^{n-2}(a_i^*a_i)^{2\ell}=z_1^{\ell}$ for
  $\ell>0$ in $A$. Note that if  $n$ is even, then
  $z_1^u=z_1^{\md }=-\pi_0+\pi_1$, but if $n$ is odd then $z_1^u=0$;
\item the elements $\zeta_\ell\in \Hom_{k^e}(V,A)$ with $0\ppq\ell\ppq u-1$ defined by  $\zeta_\ell(a_i)=a_iz_\ell$ for all $i$;
\item   the elements $\rho_\ell\in \Hom_{k^e}(V,A)$ with $0\ppq \ell\ppq \md -1$ where
  $\rho_\ell(a_i)=(a_2^*a_2)^{2\ell+1}a_i$  for $i=0,1$ and $\rho_\ell(a_i^*)= a_i^*
  (a_2^*a_2)^{2\ell+1}$  for $i=0,1$;
 
\item  the  elements $h_j\in \Hom_{k^e}(R,A)$ for
  $0\ppq j \ppq n-1$, where $h_j(\sigma_i)=\delta_{ij}e_j$;
\item the  elements $\gamma_\ell\in \Hom_{k^e}(R,A)$ for $1\ppq \ell\ppq \md $ where $\gamma_\ell(\sigma_0)=(a_0^*a_1a_1^*a_0)^\ell$.
\end{itemize}

\begin{Po} Let $A$ be a preprojective algebra of type $D_n$. 
\begin{enumerate}[\itshape(i)]
\item \sloppy The elements in $\set{\pi_i;0\ppq i\ppq n-1\text{ and }\nk(e_i)=e_i}\cup\set{z_\ell;0\ppq \ell\ppq u-1}$ 
form a basis of $\HK^0(A).$

\item If $\car(\ff)\neq 2$, the  $\ov\zeta_\ell$, for $0\ppq\ell\ppq u-1$,
  form a basis of $\HK^1(A)$.

If  $\car(\ff)= 2$, the  $\ov\zeta_\ell$ for $0\ppq\ell\ppq u-1$ and  the
$\ov\rho_\ell$ for  $0\ppq \ell\ppq \md -1$  form a basis of $\HK^1(A)$.

\item If $\car(\ff)\neq 2$, the   $\ov h_j$ for
  $0\ppq j \ppq n-1$ form a basis of $\HK^2(A)$.

If  $\car(\ff)= 2$, the   $\ov h_j$ for
  $0\ppq j \ppq n-1$ and  the  $\ov\gamma_\ell$ for $1\ppq \ell\ppq \md $ form a basis of $\HK^2(A)$.
\end{enumerate}
  
\end{Po}

\begin{proof}
The results for $\HK^0(A)$ and, when $\car(\ff)\neq 2$, for  $\HK^1(A)$ follow from the comments before the proposition. 

Assume that $\car(\ff)=2$. In order to prove the result for $\HK^1(A)$, we must prove that the elements we have considered in $\Hom_{k^e}(V,A)$
 are cocycles that are linearly independent modulo coboundaries.
It is in fact enough to prove that the $\rho_\ell$ are cocycles that are not coboundaries by fact
\ref{fact:free up to coboundaries}.

First note that, at the level of cochains, $\rho_\ell=\rho_0\cupk z_\ell$. Therefore, to prove that
$\rho_\ell$ is a cocycle, it is enough to prove that $\rho_0$ is a cocycle, and this is easy to check.

Since $\rho_{\ell}\cupk z_{\md -1-\ell}=\rho_{\md -1}$,  in order to prove that $\rho_\ell$ is not a coboundary for all $\ell$, it is enough to prove that $\rho_{\md -1}$ is not a coboundary. The map $\rho_{\md -1}$
has coefficient weight $4\md -1$. If $\rho_{\md -1}$ is a coboundary, then it is the image of a
morphism in $\Hom_{k^e}(k,A)\cong\bigoplus_{i\in Q_0}e_iA e_i$ whose coefficients are linear
combinations of cycles in $A$ of  weight $4\md -2$, which are known. It is then straightforward to show that the image of any such morphism under $b_K^1$ is not
equal to $\rho_{\md -1}$.

For \textit{(iii)}, we first observe that every cochain in $\Hom_{k^e}(R,A)$ is a
cocycle. Moreover, the $h_j$ are $n$ cocycles that are clearly linearly independent modulo coboundaries (all coboundaries have coefficient weight at least equal to $1$). Therefore if $\car(\ff)\neq 2$, the result follows.

If $\car(\ff)=2$,  it is enough to prove that the $\gamma_\ell$ are  not coboundaries  by fact
\ref{fact:free up to coboundaries}. At the level of cochains, $\gamma_\ell\cupk z_{\md -\ell}=\gamma_\md $ for
$1\ppq \ell\ppq \md $, therefore  it is enough to prove that $\gamma_\md $ is not a coboundary. The map $\gamma_\md $ has
coefficient weight $4\md $, therefore  if $\gamma_\md $ is a coboundary, then it is the image of a
morphism in $\Hom_{k^e}(V,A)$ whose coefficients are linear
combinations of elements of weight  $4\md -2$ in $A$ between two adjacent vertices in $Q_0$, which are known. Here again, checking that  that the image of any such morphism under $b_K^2$ is not
equal to   $\gamma_\md $   is straightforward.
\end{proof}

Koszul homology follows using duality (Theorem \ref{complexduality}), as in Corollary \ref{cor:bases
hom type A}.

\subsubsection{Cup  and cap products}

We now determine the cup products of the elements in the bases of the Koszul cohomology spaces given above.

\begin{Lm}
  \label{lemma:type D coboundaries for cup} For $1\ppq i\ppq n-2$, consider the cochains $u_i,$
  $v_i$ and $w_i$ in $\Hom_{k^e}(R,A)$ defined by
     $ u_i(\sigma_j)=\delta_{ij}e_iz_1$,
    $v_i(\sigma_j)=\delta_{ij}\pi_i$ and 
    $w_i(\sigma_j)=\delta_{ij}a_i^*a_iz_1$ for all $j\in Q_0$.
 
  If $\car(\ff)\neq 2$, the $u_i$, $v_i$ and $w_i$ are all coboundaries.

  If $\car(\ff)=2$, then the $u_i$ for $i\pgq 2$ are coboundaries, and $\ov u_0=\ov u_1=\ov \gamma_1$. Moreover, if $n$ is odd, all
  the $v_i$ are coboundaries and if $n$ is even, then $\ov v_i=\ov \gamma_{\md}$ for all $i$.
  Finally, the $w_i$ are all coboundaries.
\end{Lm}

\begin{proof} Every element in $\Hom_{k^e}(R,A)$ is a cocycle. Moreover, the differential $b_K^2$ is homogeneous of degree $1$ with respect to the coefficient weight, and the coefficient weight of all the basis elements in $\HK^2(A)$ is a multiple of $4$, and is $0$ if $\car(\ff)\neq 2.$

It follows that if $\car(\ff)\neq 2$, all the $u_i,$ $v_i$ and $w_i$ must be coboundaries, and if $\car(\ff)=2$, the $w_i$ are coboundaries and so are the $v_i$ if $n$ is odd.

Assume that $\car(\ff)=2$. We must now study the $u_i$, as well as the $v_i$ when $n$ is even.

Note that $u_{n-2}=0$ and $u_0=\gamma_1$. For $0\ppq i\ppq n-3$, define $p_i\in\Hom_{k^e}(V,A)$ by
$p_i(a_0)=a_1a_1^*a_0$, $p_i(a_1)=a_0a_0^*a_1$ and $ p_i(a_i)=a_ia_i^*a_i$  if $i\pgq2$.
 Then, for $2\ppq i\ppq n-3$, we have $u_i=b_K^2\left(\sum_{j=i}^{n-3}p_j\right)$. Moreover, $b_K^{2}\left(\sum_{j=0}^{n-3}p_j\right)=u_0+u_1.$
 It follows that  the cohomology classes of $u_0$ and $u_1$ are both equal to that of $\gamma_1$.

We now turn to the $v_i.$ Note that since $n$ is even and $\car(\ff)=2$, the map $v_0$ is the map $\gamma_\md $, which is not a coboundary.

Define $q_i\in\Hom_{k^e}(V,A)$ by $q_i(a_i)= (a_{i+1}^*a_{i+1})^{n-i-2} a_{i}\cdots
                       a_2a_1a_1^*a_2^*\cdots a_{i-1}^*$ for $1\ppq i\ppq n-2$ and
                       $q_0(a_0)=(a_2^*a_2)^{n-3} $. Then, for $2\ppq i\ppq n-2$, we have
                       $v_i-v_1=b_K^2\left(\sum_{j=1}^{i-1}q_j\right)$, and
                       $v_1-v_0=b_K^2(q_0-q_1)$. Therefore  $\ov v_i=\ov v_0=\ov \gamma_\md$ for all
                       $i$.
\end{proof}

We now give all the non zero cup products.

\begin{Po} Let $A$ be a preprojective algebra of type D$_n$.
  Up to graded commutativity, the non zero cup products of elements in $\HK^{\bullet}(A)$ are:
\begin{align*}
  z_0\cupk \ov f&=\ov f\text{ for all $\ov f\in\HK^{\bullet}(A)$};&
z_{\ell_1}\cupk z_{\ell_2}&=
  \begin{cases}
    z_{\ell_1+\ell_2}\text{ if $\ell_1+\ell_2\ppq u-1$};\\
    -\pi_0+\pi_1\text{ if $n$ is even and $\ell_1+\ell_2=\md $};
  \end{cases}\\
   z_{\ell_1}\cupk \ov \zeta_{\ell_2}&=\ov \zeta_{\ell_1+\ell_2}\text{ if $\ell_1+\ell_2\ppq u-1$};
 \\
  z_{\ell_1}\cupk \ov \rho_{\ell_2}&=\ov \rho_{\ell_1+\ell_2}\text{ if $\ell_1+\ell_2\ppq u-1$};&
    z_\ell\cupk \ov h_i&=                     \begin{cases}
                       \ov \gamma_\ell\text{ if $\ell\pgq 1$, $\car(\ff)=2$ and $i\in\set{0,1}$}\\h_i\text{ if $\ell=0$;}
                     \end{cases}\\z_{\ell_1}\cupk\ov \gamma_{\ell_2}&=
  \ov \gamma_{\ell_1+\ell_2}\text{ if $\ell_1+\ell_2\ppq \md $;}&
    \pi_i\cupk \ov h_j&= \gamma_\md \text{ if $i=j$, $n$ is even and
                        $\car(\ff)=2$}.
                    \end{align*}
\end{Po}

\begin{proof} We use the notation in Lemma \ref{lemma:type D coboundaries for cup}.

For $\ell\pgq 1$, we have $z_\ell\cupk h_i=z_{\ell-1}\cupk z_1\cupk h_i=z_{\ell-1}\cupk u_i$ and the result follows from  Lemma \ref{lemma:type D coboundaries for cup}.

Next, $\pi_i\cupk h_j=\delta_{ij}v_j$ and again the result is a consequence of  Lemma \ref{lemma:type D coboundaries for cup}.

Now assume that $\car(\ff)=2$, so that the $\ov \rho_\ell$ occur in the basis of $\HK^1(A).$
\sloppy At the level of cochains, we have 
$\rho_{\ell_1}\cupk\rho_{\ell_2}=w_2\cupk z_{\ell_1+\ell_2}$, which  is a coboundary.

The map $\rho_{\ell_1}\cupk \zeta_{\ell_2}=u_2\cupk z_{\ell_1+\ell_2+1}$ is also a coboundary,  as
required.

The remaining cup products are easy to compute. Note that the cup product in $\HK_0(A)\cong Z(A)$ is
the ordinary product, and that the elements $\pi_i$ are in the socle of $A$, hence are annihilated by the radical of $A$.
\end{proof}

The cap products follow using duality, as in Corollary \ref{cor:cap type A}.

\subsubsection{Higher Koszul (co)homology}

As in the case of a preprojective algebra of type $A$, the  cohomology class of the fundamental
$1$-cocycle is equal to  $2\ov\zeta_0$ so that 
 $\partial _\smile^1(z_\ell)=2\ov\zeta_{\ell}$ for $0\ppq \ell\ppq u-1$, $\partial _\smile^1(\pi_i)=0$ and $\partial _\smile^2=0.$ We then have the following higher Koszul cohomology.

\begin{Po}  \label{hkcD} Let $A$ be a preprojective algebra of type D$_n$.
  \begin{enumerate}[\itshape(i)]
  \item 
  If $\car(\ff)=2$, then $\HK^{\bullet}_{hi}(A)=\HK^{\bullet}(A).$

\item   If $\car(\ff)\neq 2$, then
  \begin{align*}
    \HK^0_{hi}(A)&=\HK^0(A)_{>0}\text{ has basis the $\pi_i$ that are in $Z(A)$}\\
    \HK^2_{hi}(A)&=\HK^2(A)\\
    \HK^p_{hi}(A)&=0\text{ if $p\neq 0$ and $p\neq 2$.}
  \end{align*}
\end{enumerate}

\end{Po}

Higher Koszul homology follows from Theorem \ref{complexduality} as in Corollary \ref{cor:higher
  type A}.

\subsection{Koszul calculus for preprojective algebras of type E$_6$}

The preprojective algebra $A$ of type $E_6$ is defined by the quiver 
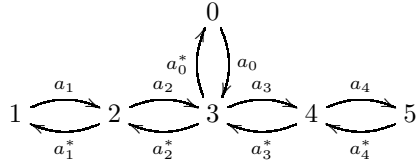
\[\xymatrix@C=10pt@R=10pt{&& && 0\ar@/^/[dd]^{a_0}  && &&\\ \\ 1\ar@/^/[rr]^{a_1} &&
  2\ar@/^/[rr]^{a_2}\ar@/^/[ll]^{a_1^*} && 3 \ar@/^/[rr]^{a_3}\ar@/^/[ll]^{a_2^*}\ar@/^/[uu]^{a_0^*}
  && 4 \ar@/^/[rr]^{a_4}\ar@/^/[ll]^{a_3^*}&& 5 \ar@/^/[ll]^{a_4^*}}\]
 subject to the relations 
\begin{align*}
&\sigma_0=-a_0^*a_0&&\sigma_3=a_0a_0^*+a_2a_2^*-a_3^*a_3\\&\sigma_1=-a_1^*a_1&&\sigma_4=a_3a_3^*-a_4^*a_4\\&\sigma_2=a_1a_1^*-a_2^*a_2&&\sigma_5=a_4a_4^*
\end{align*}

The Nakayama automorphism of \ref{fact:nakayama} is given by 
\begin{align*}
\nu (e_i)&=e_i\text{ if }i\in\set{0,3}&\nu (a_0)&=-a_0&
\nu (a_i)&=a_{5-i}^*\text{ if }i\pgq 1\\
\nu (e_i)&=e_{6-i}\text{ if }i\in\set{1,2,4,5}&\nu (a_0^*)&=a_0^*&
\nu (a_i^*)&=a_{5-i}\text{ if }i\pgq 1.
\end{align*}

To simplify notation, we shall denote by $c_0=a_0a_0^*$, $c_2=a_2a_2^*$ and $c_3=a_3^*a_3$ the three
$2$-cycles at the vertex $3$.

The socle is the part of weight $10$ of $A$, and the set $\set{\pi_i;i\in Q_0}$ where
$\pi_0=a_0^*c_3^2c_0c_3a_0$, $\pi_1=a_4a_3c_0c_3c_0a_2a_1$, $\pi_2=a_2^*(c_3c_0)^2a_3^*$,
  $\pi_3=c_3(c_0c_3)^2$, $\pi_4=\aa(\pi_2)$ and $\pi_5=\aa(\pi_1)$.

\subsubsection{The Koszul cohomology and homology spaces in type E$_6$}

We shall follow the same method as in types A and D, using the results from \cite{eusched:cyfrob}
and Theorem \ref{complexduality} to determine the dimensions of the spaces, and using results from
\cite{eu:product} for the parts that are  the same as in characteristic $0$.

We define the following elements
\begin{itemize}
\item in $A$: $z_0=1$,
  $z_6=a_1^*a_2^*a_3^*a_3a_2a_1+a_2^*c_3^2a_2-c_0c_3c_0+a_3c_2^2a_3^*+a_4a_3a_2a_2^*a_3^*a_4^*$ and
  $z_8=a_2^*c_0c_3c_0a_2+c_0c_3^2c_0+a_3c_0c_3c_0a_3^*$;
\item
in $\Hom_{k^e}(V,A)$: the maps $\zeta_\ell$ defined by $\zeta_\ell(a_i)=a_iz_\ell$ for
  $\ell\in\set{0,6,8}$, the map $\rho_3$ defined by $\rho_3(a_2)=c_0a_2$, $\rho_3(a_3)=a_3c_3$ and
  $\rho_3(a_2^*)=a_2^*c_3$, and  the map $\rho_5$  defined by $\rho_5(a_0)=c_2c_3a_0$, $\rho_5(a_1)=a_2^*c_0a_2a_1$,
  $\rho_5(a_2)=c_2^2a_2$, $\rho_5(a_0^*)=-a_0^*c_3^2$, $\rho_5(a_1^*)=-a_1^*a_2^*c_0a_2$ and
  $\rho_5(a_2^*)=a_2^*c_3^2$;
\item in $\Hom_{k^e}(R,A)$:  the maps $h_j$ defined for $0\ppq j\ppq 5$ by $h_j(\sigma_i)=\delta_{ij}e_j$ for
  all $i$,  the map $\gamma_4$ defined by $\gamma_4(\sigma_0)=a_0^*c_3a_0$ and  the map $\gamma_6$ defined by $\gamma_6(\sigma_0)=a_0^*c_3^2a_0$.
\end{itemize}

We shall use the following lemma.

\begin{Lm}\label{lemma:nsc wgt6 E6}
  Assume that $\car(\ff)=3$. Let $\gamma\in\Hom_{k^e}(R,A)$ be an element of coefficient weight $6$,
  so that 
\begin{align*}
\gamma(\sigma_0)&=\lambda_0a_0^*c_3^2a_0 &\gamma(\sigma_3)&=\lambda_3c_3^2c_0+\lambda_3'c_3c_0c_3+\lambda_3''c_0c_3^2\\
\gamma(\sigma_1)&=\lambda_1a_1^*a_2^*c_0a_2a_1 &\gamma(\sigma_4)&=\lambda_4a_3c_0c_3a_3^*+\lambda_4'a_3c_3c_0a_3^*\\
\gamma(\sigma_2)&=\lambda_2a_2^*c_0c_3a_2+\lambda_2'a_2^*c_3c_0a_2 &\gamma(\sigma_5)&=\lambda_5a_4a_3c_0a_3^*a_4^*.
\end{align*} Then $\gamma$ is a coboundary if, and only if, $\sum_{i=0}^5\lambda_i+\sum_{i=2}^4\lambda_i'+\lambda_3''=0$.
\end{Lm}

\begin{proof}
The proof is straightforward, once we know that  a cochain of weight $5$ takes its values in
$A_5=E\oplus\aa(E)$, where $E$ is the space spanned by $c_3^2a_0$, $c_0c_3a_0$, $a_2^*c_0a_2a_1,c_3^2a_2$, $c_0c_3a_2$, $a_3c_0c_3$, $a_3c_3c_0$, $a_4a_3c_0a_3^*$.
\end{proof}

\begin{Po}\label{prop:basis E6}
  Let $A$ be a preprojective algebra of type E$_6$. 
\begin{enumerate}[\itshape(i)]
\item The elements in  $\set{z_0,z_6,z_8,\pi_0,\pi_3}$ form a basis of $\HK^0(A)$.
\item If $\car(\ff)\not\in\set{2,3}$, the  elements in
  $\set{\ov\zeta_\ell;\ell=0,6,8}$  form a basis of $\HK^1(A)$.

If $\car(\ff)=2$, the  elements in $\set{\ov\zeta_\ell;\ell=0,6,8}\cup\set{\ov\rho_3}$  form a basis of $\HK^1(A)$.

If $\car(\ff)=3$, the  elements in
$\set{\ov\zeta_\ell;\ell=0,6,8}\cup\set{\ov\rho_5}$  form a basis of $\HK^1(A)$.
\item  If $\car(\ff)\not\in\set{2,3}$, the  elements in $\set{\ov h_j;j\in Q_0}$  form a basis of $\HK^2(A)$.

If $\car(\ff)=2$, the  elements in  $\set{\ov h_j;j\in Q_0}\cup\set{\ov\gamma_4}$   form a basis of $\HK^2(A)$.

If $\car(\ff)=3$, the   elements in $\set{\ov h_j;j\in Q_0}\cup\set{\ov\gamma_6}$   form a basis of $\HK^2(A)$.

\end{enumerate}

\end{Po}

\begin{proof}
The centre was given in \cite{eu:product}, so we have \textit{(i)}. 

For $\HK^1(A)$ and $\HK^2(A)$, the number of elements in the
statement is equal to the dimension of the corresponding cohomology space. Moreover, all the
elements in the statement are indeed cocycles.

If $\car(\ff)$ is not $2$ or $3$, a basis of $\HK^1(A)=\HH^1(A)$ was given in \cite{eu:product}. It
consists of the classes of the $\zeta'_\ell$ with $\ell\in\set{0,6,8}$ where
$\zeta'_\ell(a_i)=a_iz_\ell$ for $0\ppq i\ppq 2$ and $\zeta'_\ell(a_i^*)=a_i^*z_\ell$ for $3\ppq
i\ppq 4$. Since $\zeta_0-\zeta_0'$ is equal to $b_K^1(e_3+2e_5)$, and 
$\zeta_\ell-\zeta_\ell'=(\zeta_0-\zeta_0')\cupk z_\ell$ is also a coboundary,   $\zeta_\ell$ and
$\zeta_\ell'$ represent the same cohomology class for $\ell\in\set{0,6,8}$. Moreover, as in types A and D, the elements $\ov h_j$ form a basis of $\HK^2(A)$.

If $\car(\ff)\in\set{2,3}$, we need only prove that the extra elements are not coboundaries by fact
\ref{fact:free up to coboundaries}.

If $\car(\ff)=2$, we have $z_6\cupk \rho_3-\zeta_8=b_k^1(g)$ where $g$ is defined by $g(e_2)=a_2^*c_3c_0c_3a_2$, and $\zeta_8$ is not a coboundary,
therefore $\rho_3$ cannot be a coboundary. Moreover, assume that $\gamma_4=b_K^1(g')$ is a coboundary.  Then $g'$ would be of coefficient weight $3$,
and we would necessarily take values in $A_3=E\oplus\aa(E)$ where $E$ is spanned by  $c_3a_0$,
$c_0a_2,$ $c_3a_2$,
$a_3c_0,$ $a_3c_3$. This leads to a contradiction.

If $\car(\ff)=3$, assume that $\rho_5$ is a coboundary
$b_k^1(h),$ then $h$ is of weight $4$, and  necessarily $h(e_0)=\lambda_0 a_0^*c_3a_0$,
$h(e_2)=\lambda_2 a_2^*c_3a_2$ and $h(e_3)= \lambda_3 c_3c_0  + \lambda_3' c_3^2   + \lambda_3 ''
c_0c^3$, and by considering $b_K^1(h)(a_0)$, $b_K^1(h)(a_0^*)$,  $b_K^1(h)(a_2)$ and $b_K^1(h)(a_2^*)$ we get
a contradiction.
 Finally, the fact that $\gamma_6$ is not a coboundary follows from Lemma \ref{lemma:nsc
   wgt6 E6}.
\end{proof}

Koszul homology follows using duality (Theorem \ref{complexduality}), as in Corollary \ref{cor:bases
hom type A}.

\subsubsection{Cup and cap products}

We now determine the cup products of the elements in the bases of the Koszul cohomology spaces given above.

\begin{Po} Let $A$ be a preprojective algebra of type E$_6$.
  Up to graded commutativity, the non zero cup products of elements in $\HK^{\bullet}(A)$ are:
 \begin{align*}
   z_0\cupk \ov  f&=\ov  f\text{ for all $\ov f\in\HK^{\bullet}(A)$}&
 z_{\ell}\cupk \ov  \zeta_0&=\ov  \zeta_\ell\text{ for $\ell\in\set{0,6,8}$}\\
 z_6\cupk \ov  h_i&=\ov  \gamma_6\text{ if $i\in\set{1,4}$}&
z_6\cupk \ov  h_i&=-\ov  \gamma_6\text{ if $ i\in\set{2,5}$}\\
z_6\cupk \ov\rho_3&=\ov\zeta_8&\ov\zeta_0\cupk \ov\rho_3&=\ov\gamma_4.
  \end{align*}
\end{Po}

\begin{proof}  The first two cup products are clear. 

For the cup products of $z_6$ with the $\ov h_i$ and for $\ov \zeta_0\cupk \ov \rho_5$, we use Lemma \ref{lemma:nsc
   wgt6 E6}.

The last  cup product follows from the fact that we have  $\zeta_0\cupk
\rho_3-\gamma_4=b_K^1(g')$ where $g'(a_0^*)=a_0^*c_3$ and $g'(a_3)=a_3c_3$. The cup product
$z_6\cupk \ov\rho_3$ was already in the proof of Proposition \ref{prop:basis E6}.

Consideration of the coefficient weights yields the vanishing of the other cup products.
\end{proof}

The cap products follow using duality, as in Corollary \ref{cor:cap type A}.

\subsubsection{Higher Koszul cohomology and homology}

As in type A, the differential $\partial _\smile^1$ sends $z_\ell$ to $2\ov\zeta_\ell$ for
$\ell\in\set{0,6,8}$ and the differential $\partial _\smile^2$ is zero.

 We then have the following higher Koszul cohomology.

\begin{Po}  \label{hkcE6} Let $A$ be a preprojective algebra of type E$_6$.

  If $\car(\ff)=2$, then $\HK^{\bullet}_{hi}(A)=\HK^{\bullet}(A).$

  If $\car(\ff)=3$, then
  \begin{align*}
    \HK^0_{hi}(A)&=\HK^0(A)_{10}\text{ has dimension $2$ and is spanned by  $\pi_0$ and $\pi_3$}\\
    \HK^2_{hi}(A)&=\HK^2(A)\\
\HK^1_{hi}(A)&=\spn{[\ov\rho_5]}\\
    \HK^p_{hi}(A)&=0\text{ if $p>2$.}
  \end{align*}

 If $\car(\ff)\not\in\set{2,3}$, then
  \begin{align*}
    \HK^0_{hi}(A)&=\HK^0(A)_{10}\text{ has dimension $2$ and is spanned by  $\pi_0$ and $\pi_3$}\\
    \HK^2_{hi}(A)&=\HK^2(A)\\
    \HK^p_{hi}(A)&=0\text{ if $p\neq 0$ and $p\neq 2$.}
  \end{align*}
\end{Po}

Higher Koszul homology follows from Theorem \ref{complexduality}.

\subsection{Koszul calculus for preprojective algebras of type E$_7$}

The preprojective algebra $A$ of type $E_7$ is defined by the quiver 
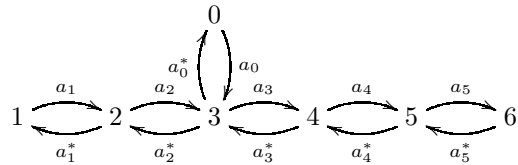
\[\xymatrix@C=10pt@R=10pt{&& && 0\ar@/^/[dd]^{a_0}  && &&\\ \\ 1\ar@/^/[rr]^{a_1} &&
  2\ar@/^/[rr]^{a_2}\ar@/^/[ll]^{a_1^*} && 3 \ar@/^/[rr]^{a_3}\ar@/^/[ll]^{a_2^*}\ar@/^/[uu]^{a_0^*}
  && 4 \ar@/^/[rr]^{a_4}\ar@/^/[ll]^{a_3^*}&& 5 \ar@/^/[ll]^{a_4^*}\ar@/^/[rr]^{a_5}&& 6 \ar@/^/[ll]^{a_5^*}}\]
 subject to the relations 
\begin{align*}
&\sigma_0=-a_0^*a_0&&\sigma_4=a_3a_3^*-a_4^*a_4\\&\sigma_1=-a_1^*a_1&&\sigma_5=a_4a_4^*-a^*_5a_5\\&\sigma_2=a_1a_1^*-a_2^*a_2&&\sigma_6=a_5a_5^*\\&\sigma_3=a_0a_0^*+a_2a_2^*-a_3^*a_3
\end{align*}

The Nakayama automorphism is given by $\nk(a_i)=-a_i$ and $\nk(a_i^*)=a_i^*$ for $i\in Q_0$.

To simplify notation, we shall denote by $c_0=a_0a_0^*$, $c_2=a_2a_2^*$ and $c_3=a_3^*a_3$ the three
$2$-cycles at the vertex $3$.

The socle of $A$ is the part of weight $16$ of $A$. A basis of the socle is given by
$\pi_0=(a_0^*c_3a_0)^4$, $\pi_1=-a_1^*a_2^*c_0c_3(c_3c_0)^2a_2a_1$, $\pi_2=-(a_2^*c_0a_2)^4$,
$\pi_3=(c_3c_0)^3c_3^2$, $\pi_4=-(a_3c_0a_3^*)^4$, $\pi_5=-a_4a_3(c_3c_0)^3a_3^*a_4^*$ and $\pi_6=-a_5a_4(a_3c_0a_3^*)^3a_4^*a_5^*$.

\subsubsection{The Koszul cohomology and homology spaces in type E$_7$}

We define the following elements
\begin{itemize}
\item in $A$: $z_0=1$,
  $z_8=a_0^*c_2c_0c_2a_0-a_2^*c_2c_0c_2a_2-c_2c_3^2c_2+a_3c_0c_2c_0a_3^*-a_4a_3c_2^2a_3^*a_4^*+a_5a_4a_3c_0a_3^*a_4^*a_5^*$ and
  $z_{12}=a_0^*(c_2c_0)^2c_2a_0^*+a_2^*(c_0c_2)^2c_0a_2-(c_3c_0c_3)^2+a_3c_3(c_0c_3)^2a_3^*$;
\item
in $\Hom_{k^e}(V,A)$: the maps $\zeta_\ell$ defined by $\zeta_\ell(a_i)=a_iz_\ell$ for
  $\ell\in\set{0,8,12}$, the map $\rho_3$ defined by $\rho_3(a_2)=c_0a_2$, $\rho_3(a_3)=a_3c_3$, $\rho_3(a_4)=a_4a_3a_3^*$ and
  $\rho_3(a_2^*)=a_2^*c_3$, the map $\rho_7$  defined by $\rho_7(a_0)=c_3^3a_0+c_3c_0c_3a_0$, $\rho_7(a_3)=a_3c_3c_0c_3$ and $\rho_7(a_3^*)=c_3c_0c_3a_3^*$, the map $\rho_{15}$ defined by $\rho_{15}(a_0)=(c_2c_0)^3c_2a_0$ and $\rho_{15}(a_0^*)=a_0^*c_2(c_0c_2)^3$ and the map $\rho_5$ defined by  $\rho_5(a_0)=-c_2c_3a_0$, $\rho_5(a_2)=c_3c_0a_2$, $\rho_5(a_3)=a_3c_3^2$, $\rho_5(a_0^*)=a_0^*c_3^2$ and $\rho_5(a_2^*)=-a_2^*c_2c_0$;
\item in $\Hom_{k^e}(R,A)$: the maps $h_j$ defined for $0\ppq j\ppq 6$ by $h_j(\sigma_i)=\delta_{ij}e_j$ for
  all $i$, the map $\gamma_4$ defined by $\gamma_4(\sigma_0)=a_0^*c_3a_0$, the map $\gamma_8$ defined by $\gamma_8(\sigma_0)=a_0^*c_3^3a_0$, the map $\gamma_{16}$ defined by $\gamma_{16}(\sigma_0)=\pi_0$ and the map  and $\gamma_6$ defined by $\gamma_6(\sigma_0)=a_0^*c_3^2a_0$.
\end{itemize}

\begin{Lm}\label{lemma:nsc wgt even E7}
  First assume that $\car(\ff)=2$.
  \begin{enumerate}[\itshape(i)]
 \item Let $u_{16}\in\Hom_{k^e}(R,A)$ be an element of weight $16$ so that $u_{16}(\sigma_i)=\lambda_i\pi_i$ for $i\in Q_0$. Then $u_{16}$ is a coboundary if, and only if, $\sum_{i=0}^6{\lambda_i}=0$.
  \item Let $u_8\in\Hom_{k^e}(R,A)$ be an element of weight $8$ so that $u_8(\sigma_0)=\lambda_0a_0^*c_3^3a_0$, $u_8(\sigma_2)=\lambda_2a_2^*c_3c_0c_3a_2+\lambda_2'a_2^*c_3^2c_0a_2$, $u_8(\sigma_3)=\lambda_3c_3^3c_0+\lambda_3'c_cc_3^3+\lambda_3''c_0c_3^2c_0+\lambda_3'''c_3^2c_0c_3$, $u_8(\sigma_4)=\lambda_4a_3c_3^2c_0a_3^*+\lambda_4'a_3c_3c_0c_3a_3^*+\lambda_4''a_3c_0c_3^2a_3^*$, $u_8(\sigma_5)=\lambda_5a_4a_3c_3c_0a_3^*a_4^*+\lambda_5'a_4a_3c_0c_3a_3^*a_4^*$ and $u_8(\sigma_6)=a_5a_4a_3a_0a_0^*a_3^*a_4^*a_5^*$. Then $u_8$ is a coboundary if, and only if, $\lambda_0+\lambda_2+\lambda_2'+\lambda_3+\lambda_3'+\lambda_3'''+\lambda_4+\lambda_4'+\lambda_4''+\lambda_5+\lambda_5'+\lambda_6=0$.
  \end{enumerate}
  Now assume that $\car(\ff)=3$.
  \begin{enumerate}[(i),resume]
  \item  Let $u_6\in\Hom_{k^e}(R,A)$ be an element of weight $6$ so that $u_6(\sigma_0)=\lambda_0a_0^*c_3^2a_0$, $u_6(\sigma_1)=\lambda_1a_1^*a_2^*c_0a_2^*a_1^*$,  $u_6(\sigma_2)=\lambda_2a_2^*c_3c_0a_2+\lambda_2'a_2^*c_3^2a_2$, $u_6(\sigma_3)=\lambda_3c_3c_0c_3+\lambda_3'c_3^2c_0+\lambda_3''c_0c_3c_0+\lambda_3'''c_3^3+$, $u_6(\sigma_4)=\lambda_4a_3c_0c_3a_3^*+\lambda_4'a_3c_3c_0a_3^*$ and $u_6(\sigma_5)=\lambda_5a_4a_3c_0a_3^*a_4^*$. Then $u_6$ is a coboundary if, and only if, $\sum_{i=0}^5{\lambda_i}-\lambda_2'+\lambda_3'-\lambda_4'=0$.
    \end{enumerate}
  \end{Lm}

  \begin{proof} For each $\ell\in\set{16,6}$, if the map $u_\ell$ were a coboundary, it would be the image of
  a map $g_\ell\in\Hom_{k^e}(V,A)$ whose coefficients would be in the space generated by the paths
  between adjacent vertices with weight $\ell-1$. The proof is then straightforward once we know
  bases of these spaces. Note that once we have a basis of  $\bigoplus_{\alpha\in
    Q_1}e_{\mt(\alpha)}A_{w}e_{\mo(\alpha)}$, applying $\aa$ gives a basis of
  $e_{\mo(\alpha)}A_we_{\mt(\alpha)}$ for a given weight $w$.

\sloppy  In weight $15$,    a basis of $\bigoplus_{\alpha\in
    Q_1}e_{\mt(\alpha)}A_{15}e_{\mo(\alpha)}$ is given by  $\ov a_0 =(a_0^*c_3a_0)^3a_0^*c_3$,  $\ov a_1 =-a_1^*a_2^*c_0c_3(c_3c_0)^2a_2$,  $\ov a_2 =-a_2^*c_0(c_2c_0)^3$,  $\ov a_3 =(c_3c_0)^3c_3a_3^*$,  $\ov a_4 =-a_3(c_3c_0)^3a_3^*a_4^*$,  $\ov a_5 =a_4(a_3c_0a_3^*)^3a_4^*a_5^*$.

  In weight $7$, a basis of $\bigoplus_{\alpha\in
    Q_1}e_{\mt(\alpha)}A_{7}e_{\mo(\alpha)}$ is given by $c_3^3a_0$, $c_3c_0c_3a_0$, $a_2^*c_3^2a_2a_1$, $c_3^2c_0a_2$, $c_0c_3^2a_2$, $c_3c_0c_3a_2$, $a_3c_3c_0c_3$, $a_3c_0c_3^2$, $a_3c_3^2c_0$. $a_4a_3c_3c_0a_3^*$, $a_4a_3c_0c_3a_3^*$, $a_5a_4a_3c_0a_3^*a_4^*$.

  In weight $5$, a basis of   $\bigoplus_{\alpha\in
    Q_1}e_{\mt(\alpha)}A_{5}e_{\mo(\alpha)}$ is given by $c_3^2a_0$, $c_0c_3a_0$, $a_2^*c_0a_2a_1$,
  $c_3c_0a_2$, $c_0c_3a_2$,  $a_3c_3c_0$, $a_3c_0c_3$, $a_3c_3^2$, $a_4a_3c_0a_3^*$.
  \end{proof}

\begin{Po}\label{prop:basis HK E7}
  Let $A$ be a preprojective algebra of type E$_7$. 
\begin{enumerate}[\itshape(i)]
\item The elements in  $\set{z_0,z_8,z_{12}}\cup\set{\pi_i;i\in Q_0}$ form a basis of $\HK^0(A)$.
\item If $\car(\ff)\not\in\set{2,3}$,  the elements in
  $\set{\ov \zeta_\ell;\ell=0,8,12}$  form a basis of $\HK^1(A)$.

If $\car(\ff)=2$, the elements in $\set{\ov \zeta_\ell;\ell=0,8,12}\cup\set{\ov \rho_3,\ov \rho_7,\ov \rho_{15}}$  form a basis of $\HK^1(A)$.

If $\car(\ff)=3$, the elements in
$\set{\ov \zeta_\ell;\ell=0,8,12}\cup\set{\ov \rho_5}$  form a basis of $\HK^1(A)$.
\item  If $\car(\ff)\not\in\set{2,3}$, the elements in $\set{\ov h_j;j\in Q_0}$  form a basis of $\HK^2(A)$.

If $\car(\ff)=2$, the elements in  $\set{\ov h_j;j\in Q_0}\cup\set{\ov \gamma_4,\ov \gamma_8,\ov \gamma_{16}}$   form a basis of $\HK^2(A)$.

If $\car(\ff)=3$, the elements in $\set{\ov h_j;j\in Q_0}\cup\set{\ov \gamma_6}$   form a basis of $\HK^2(A)$.

\end{enumerate}

\end{Po}

\begin{proof} 
The centre was given in \cite{eu:product}, so we have \textit{(i)}. 

For $\HK^1(A)$ and $\HK^2(A)$, the number of elements in the
statement is equal to the dimension of the corresponding cohomology space. Moreover, all the
elements in the statement are indeed cocycles.

If $\car(\ff)$ is not $2$ or $3$, a basis of $\HK^1(A)=\HH^1(A)$ was given in \cite{eu:product}. It
consists of the classes of the $\zeta'_\ell$ with $\ell\in\set{0,8,12}$ where
$\zeta'_\ell(a_i)=a_iz_\ell$ for $0\ppq i\ppq 2$ and $\zeta'_\ell(a_i^*)=a_i^*z_\ell$ for $3\ppq
i\ppq 5$. Since $\zeta_0-\zeta_0'$ is equal to $b_K^1(e_3+2e_5+3e_6)$, and 
$\zeta_\ell-\zeta_\ell'=(\zeta_0-\zeta_0')\cupk z_\ell$ is also a coboundary,  $\zeta_\ell$ and
$\zeta_\ell'$ represent the same cohomology class for $\ell\in\set{0,8,12}$. Moreover, as  in types
A, D and E$_{6}$, the elements $\ov h_j$ form a basis of $\HK^2(A)$. 

If $\car(\ff)\in\set{2,3}$, we need only prove that the extra elements are not coboundaries by fact
\ref{fact:free up to coboundaries}. 

If $\car(\ff)=2$, it follows from Lemma \ref{lemma:nsc wgt even E7} that $\zeta_0\cupk
\rho_{15}-\gamma_{16}$, $\zeta_8\cupk\rho_7-\gamma_{16}$, $\zeta_{12}\cupk\rho_3-\gamma_{16}$,
$z_8\cupk\gamma_8-\gamma_{16}$ and $z_{12}\cupk\gamma_4-\gamma_{16}$ are coboundaries. Therefore it
is enough to check that $\gamma_{16}$ is not a coboundary, and this also follows from Lemma
\ref{lemma:nsc wgt even E7}.

If $\car(\ff)=3$, again using  Lemma \ref{lemma:nsc wgt even E7}, $\zeta_0\cupk\rho_5-\gamma_6$
 is a coboundary and $\gamma_6$ is not a coboundary, therefore $\rho_5$ is not a coboundary either.
\end{proof}

Koszul homology follows using duality (Theorem \ref{complexduality}), as in Corollary \ref{cor:bases
hom type A}.

\subsubsection{Cup and cap products}

We now determine the cup products of the elements in the bases of the Koszul cohomology spaces given above.

\begin{Po} Let $A$ be a preprojective algebra of type E$_7$.
  Up to graded commutativity, the non zero cup products of elements in $\HK^{\bullet}(A)$ are:
 \begin{align*}
   z_0\cupk \ov  f&=\ov  f\text{ for all $\ov f\in\HK^{\bullet}(A)$}& z_8^2&=\pi_0+\pi_4-\pi_6\\
   z_{\ell}\cupk \ov  \zeta_0&=\ov  \zeta_\ell\text{ for $\ell\in\set{0,8,12}$}&z_8\cupk \ov \rho_7&=\ov \rho_{15}\\
   z_{12}\cupk\ov\rho_3&=\ov \rho_{15} & \pi_i\cupk \ov h_i&=\ov \gamma_{16}\text{ for $i\in Q_0$}\\
    z_8\cupk \ov  h_i&=\ov  \gamma_8\text{ if $i\in\set{0,4,6}$}& z_8\cupk \ov\gamma_8&=\ov \gamma_{16}\\
z_{12}\cupk \ov\gamma_4&=\ov \gamma_{16}& \ov\zeta_0\cupk \ov\rho_\ell&=-\ov\gamma_{\ell+1}\text{
                                                                        for
                                                                        $\ell\in\set{3,7,15,6}$}\\\ov
   \zeta_8\cupk \ov\rho_7&=\ov\gamma_{16}&\ov\zeta_{12}\cupk \ov\rho_3&=\ov\gamma_{16}
  \end{align*}
\end{Po}

\begin{proof} Most of the cup-products are easy to compute, follow from Lemma \ref{lemma:nsc wgt even E7} or vanish for weight reasons. The remaining ones are obtained as follows (at the level of cochains): 
\begin{align*}
&z_8\cupk \rho_7=\rho_{15}+b_K^1([e_3\mapsto (c_3c_0)^3c_3])\\&z_{12}\cupk
                                                                \rho_3=\rho_{15}+b_K^1([e_3\mapsto
                                                                (c_3c_0)^3c_2])\\
&\zeta_0\cupk \rho_3=\gamma_4+b_K^1(h)
  \end{align*} where $h\in\Hom_{k^e}(V,A_3)$ is defined by $h(a_0)=c_3a_0$ and $h(a_2)=c_2a_2$. 
\end{proof}

The cap products follow using duality, as in Corollary \ref{cor:cap type A}.

\subsubsection{Higher Koszul cohomology and homology}

As in types A, D and E$_6$, the differential $\partial _\smile^1$ sends $z_\ell$ to $2\ov\zeta_\ell$ for
$\ell\in\set{0,8,12}$ and the differential $\partial _\smile^2$ is zero except when $\car(\ff)=3$ where $\partial _\smile^2(\ov\rho_5)=\ov\gamma_6$.

 We then have the following higher Koszul cohomology.

\begin{Po}  \label{hkcE7} Let $A$ be a preprojective algebra of type E$_7$.

  If $\car(\ff)=2$, then $\HK^{\bullet}_{hi}(A)=\HK^{\bullet}(A).$

 If $\car(\ff)\neq 2$, then
  \begin{align*}
    \HK^0_{hi}(A)&=\HK^0(A)_{16}\text{  has dimension $7$ and is spanned by  the $\pi_i$ for $i\in Q_0$ }\\
    \HK^2_{hi}(A)&=\HK^2(A)_0 \text{ has dimension $7$ and is spanned by the $[\ov h_i]$ for $i\in Q_0$}\\
    \HK^p_{hi}(A)&=0\text{ if $p\neq 0$ and $p\neq 2$.}
  \end{align*}
\end{Po}

Higher Koszul homology follows from Theorem \ref{complexduality}.

\subsection{Koszul calculus for preprojective algebras of type E$_8$}

The preprojective algebra $A$ of type $E_8$ is defined by the quiver 
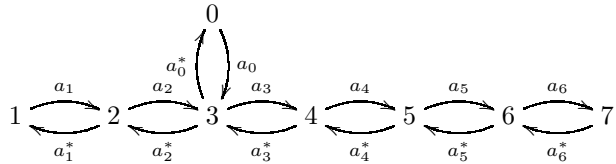
\[\xymatrix@C=10pt@R=10pt{&& && 0\ar@/^/[dd]^{a_0}  && &&\\ \\ 1\ar@/^/[rr]^{a_1} && 2\ar@/^/[rr]^{a_2}\ar@/^/[ll]^{a_1^*} && 3 \ar@/^/[rr]^{a_3}\ar@/^/[ll]^{a_2^*}\ar@/^/[uu]^{a_0^*} && 4 \ar@/^/[rr]^{a_4}\ar@/^/[ll]^{a_3^*}&& 5 \ar@/^/[ll]^{a_4^*}\ar@/^/[rr]^{a_5}&& 6 \ar@/^/[ll]^{a_5^*}\ar@/^/[rr]^{a_6}&& 7 \ar@/^/[ll]^{a_6^*}} \]
 subject to the relations 
\begin{align*}
&\sigma_0=-a_0^*a_0&&\sigma_4=a_3a_3^*-a_4^*a_4\\&\sigma_1=-a_1^*a_1&&\sigma_5=a_4a_4^*-a^*_5a_5\\&\sigma_2=a_1a_1^*-a_2^*a_2&&\sigma_6=a_5a_5^*-a_6^*a_6\\&\sigma_3=a_0a_0^*+a_2a_2^*-a_3^*a_3&&\sigma_7=a_6a_6^*
\end{align*}

The Nakayama automorphism is given by $\nk(a_i)=-a_i$ and $\nk(a_i^*)=a_i^*$ for $i\in Q_0$.

To simplify notation, we shall denote by $c_0=a_0a_0^*$, $c_2=a_2a_2^*$ and $c_3=a_3^*a_3$ the three
$2$-cycles at the vertex $3$.

The socle of $A$ is the part of weight $28$ of $A$. A basis of the socle is given by $\pi_0=(a_0^*c_2a_0)^7$, $\pi_1=a_1^*a_2^*(c_0c_2)^5c_2c_0a_2a_1$, $\pi_2=-(a_2^*c_0a_2)^7$, $\pi_3=(c_2c_0)^7$, $\pi_4=-a_3(c_2c_0)^6c_2a_3^*$, $\pi_5=a_4a_3(c_2c_0)^6a_3^*a_4^*$, $\pi_6=a_5a_4a_3(c_0c_2)^5c_0a_3^*a_4^*a_5^*$ and $\pi_7=-a_6a_5a_4a_3(c_0c_2)^3(c_2c_0)^2a_3^*a_4^*a_5^*a_6^*$.

\subsubsection{The Koszul cohomology and homology spaces in type E$_8$}

We define the following elements
\begin{enumerate}[$\bullet$,itemindent=0pt,labelindent=0pt,leftmargin=*,labelsep=3pt]
\item in $A$:\sloppy
  \begin{enumerate}[label={\tiny\ding{70}},itemindent=0pt,labelindent=0pt,leftmargin=*,labelsep=3pt]
  \item $z_0=1$,
   
  \item  $z_{12}=a_5a_4a_3c_0c_2c_0a_3^*a_4^*a_5^*+a_4a_3(c_2c_0)^2a_3^*a_4^*+a_3(c_2c_0)^2c_2a_3^*-(c_3c_2c_3)^2+a_2^*(c_0c_2)^2c_0a_2-a_1^*a_2^*c_0c_2^2c_0a_2a_1+a_0^*(c_2c_0)^2c_2a_0$,
   
  \item  $z_{20}=a_5a_4a_3(c_0c_2)^3c_0a_3^*a_4^*a_5^*+a_4a_3(c_0c_2)^2(c_2c_0)^2a_3^*a_4^*+a_3c_0(c_2c_0)^4a_3^*+(c_2c_0)^5-(c_0c_2^2)^3c_0+(c_0c_2)^5+a_2^*(c_2c_0c_2)^3a_2+a_0^*(c_2c_0)^4c_2a_0$
    and
  \item  $z_{24}=z_{12}^2$;
\end{enumerate}

\item
in $\Hom_{k^e}(V,A)$:
\begin{enumerate}[label={\tiny\ding{70}},itemindent=0pt,labelindent=0pt,leftmargin=*,labelsep=3pt]
\item the maps $\zeta_\ell$ defined by $\zeta_\ell(a_i)=a_iz_\ell$ for
  $\ell\in\set{0,12,20,24}$, 
\item  the map $\rho_3$ defined by $\rho_3(a_2)=c_0a_2$, $\rho_3(a_3)=a_3c_3$, $\rho_3(a_4)=a_4a_3a_3^*$, $\rho_3(a_5)=a_5a_4a_4^*$ and
  $\rho_3(a_2^*)=a_2^*c_3$, 
\item the map $\rho_7$  defined by $\rho_7(a_0)=c_0c_3^2a_0$, $\rho_7(a_3)=a_3c_0c_3^2+a_3c_3^2c_0+a_3c_0c_3c_0$ and $\rho_7(a_3^*)=c_3c_0c_3a_3^*$, 
\item the map $\rho_{15}$ defined by $\rho_{15}(a_3)=a_3c_2(c_2c_0)^2$, $\rho_{15}(a_4)=a_4a_3(c_0c_3)^2a_3^*+a_4a_3(c_0c_3^2)^2a_3^*+a_4a_3(c_3c_0)^3a_3^*$, $\rho_{15}(a_5)=a_5a_4a_3c_0(c_3c_0)^2a_3^*a_4^*$, $\rho_{15}(a_0^*)=a_0^*c_3(c_3c_0)^3$ and $\rho_{15}(a_5^*)=a_4a_3(c_0c_3)^2c_0a_3^*a_4^*a_5^*$,
\item the map $\rho_{27}$ defined by $\rho_{27}(a_3)=a_3c_0(c_3c_0)^6$ and $\rho_{27}(a_3^*)=(c_3c_0)^6c_3a_3^*$,
\item the map $\rho_5$ defined by  $\rho_5(a_0)=-c_0c_3a_0-c_3^2a_0$, $\rho_5(a_3)=a_3c_3c_0+a_3c_3^2$, $\rho_5(a_4)=a_4a_3c_3a_3^*$, $\rho_5(a_0^*)=a_0^*c_3^2$ and $\rho_5(a_3^*)=-c_3c_0a_3^*$,
\item the map $\rho_{17}$ defined by  defined by
  $\rho_{17}(a_0)=-c_3^2(c_0c_3)^3a_0-(c_0c_3)^4a_0$,
  $\rho_{17}(a_3)=a_3(c_3c_0)^4+a_3(c_0c_3)^4+a_3(c_3c_0)^3c_3^2$,
  $\rho_{17}(a_0^*)=a_0^*(c_2c_0)^3c_3^2$ and $\rho_{17}(a_3^*)=-(c_3c_0)^4a_3^*$, and
\item the map $\rho_9$ defined by $\rho_9(a_0)=-2c_3^2c_0c_3a_0+2c_3c_0c_3^2a_0+(c_0c_3)^2a_0$,  $\rho_9(a_2)=c_2^2c_0c_2a_2+(c_2c_0)^2a_2$, $\rho_9(a_3)=-a_3(c_0c_3)^2$, $\rho_9(a_0^*)=2a_0^*c_3^2c_0c_3-2a_0^*c_3c_0c_3^2+a_0^*(c_3c_0)^2$, $\rho_9(a_2^*)=a_2^*c_2c_0c_2^2-a_2^*(c_2c_0)^2$ and  $\rho_9(a_3^*)=(c_0c_3)^2a_3^*$;
\end{enumerate}

\item in $\Hom_{k^e}(R,A)$: the maps $h_j$ defined for $0\ppq j\ppq 7$ by $h_j(\sigma_i)=\delta_{ij}e_j$ for
  all $i$, and
  \begin{align*}
\gamma_4\colon  &  \sigma_0 \mapsto  a_0^*c_3a_0 &\gamma_8\colon &\sigma_0\mapsto  a_0^*c_3^3a_0 \\
\gamma_{16}\colon  &  \sigma_0 \mapsto  a_0^*(c_2c_0)^3c_2a_0 &\gamma_{28}\colon &\sigma_0\mapsto  \pi_0 \\
\gamma_6 \colon &  \sigma_0 \mapsto a_0^*c_3^2a_0  &\gamma_{18}\colon &\sigma_0\mapsto a_0^*c_3^2(c_0c_3)^3a_0  \\
\gamma_{10}\colon  &  \sigma_0 \mapsto   a_0^*c_3^2c_0c_3a_0
    \end{align*}

\end{enumerate}

\begin{Lm}\label{lemma:nsc wgt even E8} 
  First assume that $\car(\ff)=2$.
  \begin{enumerate}[\itshape(i)]
  \item Let $u_{28}\in\Hom_{k^e}(R,A)$ be an element of weight $28$, so that
    $u_{28}(\sigma_i)=\lambda_i\pi_i$ for all $i\in Q_0$. Then
    $u_{28}$ is a coboundary if, and only if, $\sum_{i\in Q_0}\lambda_i=0$.

\end{enumerate}
  Now assume that $\car(\ff)=3$.
  \begin{enumerate}[(i),resume]
  \item \sloppy Let $u_{18}\in\Hom_{k^e}(R,A)$ be an element of weight $18$, so that $u_{18}(\sigma_0)=
  \lambda_0  a_0^*c_3^2(c_0c_3)^3a_0+ \lambda_0'  a_0^*c_3c_0(c_3c_0c_3)^2a_0$,
  $u_{18}(\sigma_1)=\lambda_1a_1^*a_2^*c_0(c_2c_0)^3a_2a_1$,
  $u_{18}(\sigma_2)=\lambda_2a_2^*(c_2c_0c_2)^2c_0c_2a_2+\lambda_2'a_2^*(c_0c_2)^4a_2+\lambda_2''a_2^*(c_2c_0)^4a_2$,
  $u_{18}(\sigma_3)=\lambda_3
  c_0(c_3c_0)^4+\lambda_3'(c_3c_0)^4c_3+\lambda_3''(c_3c_0)^2c_3(c_3c_0)^2+\lambda_3^{(3)}(c_0c_3)^2c_3(c_0c_3)^2+\lambda_3^{(4)}(c_3c_0)^3c_3^2c_0+\lambda_3^{(5)}c_0c_3^2(c_0c_3)^3$,
  $u_{18}(\sigma_4)=\lambda_4a_3(c_3c_0)^4a_3^*+\lambda_4'a_3(c_0c_3)^4a_3^*+\lambda_4''a_3c_3^2(c_3c_0)^3a_3^*+\lambda_4^{(3)}a_3(c_0c_3)^3c_3^2a_3^*$,  $u_{18}(\sigma_5)=\lambda_5a_4a_3(c_3c_0)^3c_3a_3^*a_4^*+\lambda_5'a_4a_3c_3(c_3c_0)^3a_3^*a_4^*+\lambda_5''a_4a_3(c_0c_3)^3c_3a_3^*a_4^*$,
  $u_{18}(\sigma_6)=\lambda_6a_5a_4a_3(c_3c_0)^3a_3^*a_4^*a_5^*+\lambda_6'a_5a_4a_3(c_0c_3)^3a_3^*a_4^*a_5^*$,
  $u_{18}(\sigma_7)=\lambda_7a_6a_5a_4a_3(c_0c_3)^2c_0a_3^*a_4^*a_5^*a_6^*$. Then $u_{18}$ is a
  coboundary if, and only if, $\lambda_0+\lambda_0'+\lambda_1+\lambda_2'+\lambda_2''+\lambda_3'+\lambda_3''+\lambda_3^{(3)}+\lambda_3^{(4)}+\lambda_3^{(5)}+\lambda_4+\lambda_4'-\lambda_4''-\lambda_4^{(3)}-\lambda_5-\lambda_5'-\lambda_5''-\lambda_6-\lambda_6'-\lambda_7=0$.

\item \sloppy The map $\rho_{17}\in\Hom_{k^e}(V,A)$ is not a coboundary.
\end{enumerate}
  Now assume that $\car(\ff)=5$.
  \begin{enumerate}[(i),resume]
  \item \sloppy  Let $u_{10}\in\Hom_{k^e}(R,A)$ be an element of weight $10$, so that
    $u_{10}(\sigma_0)=\lambda_0a_0^*c_3^2c_0c_3a_0+\lambda_0'a_0^*c_3c_0c_3^2a_0$,
    $u_{10}(\sigma_1)=\lambda_1a_1^*a_2^*c_0c_2c_0a_2a_1$,
    $u_{10}(\sigma_2)=\lambda_2a_2^*(c_0c_2)^2a_2+\lambda_2'a_2^*(c_2c_0)^2a_2+\lambda_2''a_2^*c_0c_2^2c_0a_2$,
$u_{10}(\sigma_3)=\lambda_3c_0(c_3c_0)^2+\lambda_3'c_3(c_0c_3)^2+\lambda_3''c_3(c_3c_0)^2+\lambda_3^{(3)}(c_0c_3)^2c_3+\lambda_3{(4)}c_3c_0c_3^2c_0+\lambda_3^{(5)}c_0c_3^2c_0c_3$,
$u_{10}(\sigma_4)=\lambda_4a_3(c_0c_3)^2a_3^*+\lambda_4'a_3(c_3c_0)^2a_3^*+\lambda_4''a_3c_3c_0c_3^2a_3^*+\lambda_4^{(3)}a_3c_3^2c_0c_3a_3^*$,
$u_{10}(\sigma_5)=\lambda_5a_4a_3c_0c_3c_0a_3^*a_4^*+\lambda_5'a_4a_3c_3^2c_0a_3^*a_4^*+\lambda_5''a_4a_3c_0c_3^2a_3^*a_4^*$,     $u_{10}(\sigma_6)=\lambda_6a_5a_4a_3c_3c_0a_3^*a_4^*a_5^*+\lambda_6'a_5a_4a_3c_0c_3a_3^*a_4^*a_5^*$,
    $u_{10}(\sigma_7)=\lambda_7a_6a_5a_4a_3c_0a_3^*a_4^*a_5^*a_6^*$. Then $u_{10}$ is a coboundary
    if, and only if, $\lambda_0+\lambda_0'+\lambda_1+\lambda_2+\lambda_2'+\lambda_2''+\lambda_3'+\lambda_3''+\lambda_3^{(3)}+\lambda_3^{(4)}+\lambda_3^{(5)}+\lambda_4-\lambda_4'+2\lambda_4''+2\lambda_4^{(3)}+\lambda_5+2\lambda_5'+2\lambda_5''+2\lambda_6+2\lambda_6'+2\lambda_7=0$.
    \end{enumerate}
  \end{Lm}

  \begin{proof} For each $\ell\in\set{28,18,10}$, if the map $u_\ell$ were a coboundary, it would be the
  image of a map $g_\ell\in\Hom_{k^e}(V,A)$ whose coefficients would be in the space generated by
  the paths between adjacent vertices with weight $\ell-1$. The proof of \textit{(i), (ii)} and \textit{(iv)} is then straightforward once we know bases of these spaces.

  \sloppy  In weight $27$, a basis is given by  $\ov a_0 =a_0^*c_2(c_0c_2)^6$,   $\ov a_1
   =a_1^*a_2^*(c_0c_2)^5c_3^2a_2$,   $\ov a_2 =-a_2^*c_0(c_2c_0)^6$,   $\ov a_3
   =(c_2c_0)^6c_2a_3^*$,   $\ov a_4 =-a_3(c_2c_0)^6a_3^*a_4^*$,   $\ov a_5
   =-a_4a_3c_0(c_2c_0)^5a_3^*a_4^*a_5^*$,   $\ov a_6=a_5a_4a_3(c_0c_2)^3(c_2c_0)^2a_3^*a_4^*a_5^*a_6^* $,
   and the $\aa(\ov a_i)$ for all $i\in Q_0$. 

In weight $17$, the space $\bigoplus_{\alpha\in Q_1}e_{\mt(\alpha)}A_{18}e_{\mo(\alpha)}$ has basis
$(c_0c_3)^4a_0$, $c_3^2(c_0c_3)^3a_0$, $c_3c_0(c_3c_0c_3)^2a_0$, $a_2^*c_2c_0c_2(c_2c_0)^2a_2a_1$,
$a_2^*(c_0c_2)^3c_0a_2a_1$, $(c_3c_0)^4a_2$, $(c_0c_3)^4a_2$, $(c_0c_3)^3c_3c_0a_2$, $c_3c_0(c_3c_0c_3)^2a_2$, $a_3(c_3c_0)^4$, $a_3(c_0c_3)^4$, $a_3(c_0c_3)^3c_3c_0$,
$a_3c_0c_3(c_3c_0)^3$, $a_3(c_3c_0)^3c_3^2$, $a_4a_3(c_0c_3)^3c_0a_3^*$,
$a_4a_3(c_0c_3)^3c_3a_3^*$, $a_4a_3c_3(c_3c_0)^3a_3^*$, $a_5a_4a_3(c_3c_0)^3a_3^*a_4^*$,
$a_5a_4a_3(c_0c_3)^3a_3^*a_4^*$, $a_6a_5a_4a_3(c_0c_3)^2c_0a_3^*a_4^*a_5^*$ and
$\bigoplus_{\alpha\in Q_1}e_{\mo(\alpha)}A_{18}e_{\mt(\alpha)}=\aa\left(\bigoplus_{\alpha\in
    Q_1}e_{\mt(\alpha)}A_{18}e_{\mo(\alpha)}\right)$.

In weight $9$, the space $\bigoplus_{\alpha\in Q_1}e_{\mt(\alpha)}A_{9}e_{\mo(\alpha)}$ has basis
$c_3^2c_0c_3a_0$, $c_3c_0c_3^2a_0$, $(c_0c_3)^2a_0$, $a_2^*c_0c_2c_0a_2a_1$, $c_2^2c_0c_2a_2$,
$(c_2c_0)^2a_2$, $(c_0c_2)^2a_2$, $c_0c_2^2c_0a_2$, $a_3c_3c_0c_3^2$, $a_3(c_0c_3)^2$,
$a_3(c_3c_0)^2$, $a_3c_0c_3^2c_0$, $a_4a_3c_3^2c_0a_3^*$, $a_4a_3c_0c_3^2a_3^*$,
$a_4a_3c_0c_3c_0a_3^*$, $a_5a_4a_3c_3c_0a_3^*a_4^*$, $a_5a_4a_3c_0c_3a_3^*a_4^*$, $a_6a_5a_4a_3c_0a_3^*a_4^*a_5^*$ and
$\bigoplus_{\alpha\in Q_1}e_{\mo(\alpha)}A_{9}e_{\mt(\alpha)}=\aa\left(\bigoplus_{\alpha\in
    Q_1}e_{\mt(\alpha)}A_{9}e_{\mo(\alpha)}\right)$.

Finally, if $\rho_{17}$ were a coboundary, it would be equal to $b_k^1(g)$ for some
$g\in\Hom_{k^e}(k,A_{16})$. Such a map $g$ would necessarily satisfy
 $g(e_0)\in\spn{a_0^*(c_3c_0)^3c_3a_0}$, $g(e_3)\in\spn{(c_3c_0)^4,(c_0c_3)^4,c_0c_3(c_3c_0)^3,(c_0c_3)^3c_3c_0,c_3^2(c_0c_3)^3,c_3c_0(c_3c_0c_3)^2}$,  and
$g(e_4)\in\spn{a_3c_0(c_3c_0)^3a_3^*,a_3c_3(c_0c_3)^3a_3^*,a_3c_3(c_3c_0)^3a_3^*,a_3(c_0c_3)^3c_3a_3^*}$.
 Then,
by considering $b_k^1(g)(a_0)$,  $b_k^1(g)(a_0^*)$ and $b_k^1(g)(a_3)$  we get a contradiction.
\end{proof}

\begin{Po}
  Let $A$ be a preprojective algebra of type E$_8$. 
\begin{enumerate}[\itshape(i)]
\item The elements in  $\set{z_0,z_{12},z_{20},z_{24}}\cup\set{\pi_i;i\in Q_0}$ form a basis of $\HK^0(A)$.
\item If $\car(\ff)\not\in\set{2,3,5}$, the   elements in
  $\set{\ov \zeta_\ell;\ell=0,12,20,24}$  form a basis of $\HK^1(A)$.

If $\car(\ff)=2$, the   elements in $\set{\ov \zeta_\ell;\ell=0,12,20,24}\cup\set{\ov \rho_3,\ov \rho_7,\ov \rho_{15},\ov \rho_{27}}$  form a basis of $\HK^1(A)$.

If $\car(\ff)=3$, the   elements in
$\set{\ov \zeta_\ell;\ell=0,12,20,24}\cup\set{\ov \rho_5,\ov \rho_{17}}$  form a basis of $\HK^1(A)$.

If $\car(\ff)=5$, the   elements in
$\set{\ov \zeta_\ell;\ell=0,12,20,24}\cup\set{\ov \rho_9}$  form a basis of $\HK^1(A)$.
\item  If $\car(\ff)\not\in\set{2,3,5}$, the   elements in $\set{\ov h_j;j\in Q_0}$  form a basis of $\HK^2(A)$.

If $\car(\ff)=2$, the   elements in  $\set{\ov h_j;j\in Q_0}\cup\set{\ov \gamma_4,\ov \gamma_8,\ov \gamma_{16},\ov \gamma_{28}}$   form a basis of $\HK^2(A)$.

If $\car(\ff)=3$, the   elements in $\set{\ov h_j;j\in Q_0}\cup\set{\ov \gamma_6,\ov \gamma_{18}}$   form a basis of $\HK^2(A)$.

If $\car(\ff)=5$, the   elements in $\set{\ov h_j;j\in Q_0}\cup\set{\ov \gamma_{10}}$   form a basis of $\HK^2(A)$.

\end{enumerate}

\end{Po}

\begin{proof} The centre was given in \cite{eu:product}, so we have \textit{(i)}. 

For $\HK^1(A)$ and $\HK^2(A)$, the number of elements in the
statement is equal to the dimension of the corresponding cohomology space. Moreover, all the
elements in the statement are indeed cocycles.

If $\car(\ff)$ is not $2$, $3$ or $5$, a basis of $\HK^1(A)=\HH^1(A)$ was given in \cite{eu:product}. It
consists of the classes of the $\zeta'_\ell$ with $\ell\in\set{0,12,20,24}$ where
$\zeta'_\ell(a_i)=a_iz_\ell$ for $0\ppq i\ppq 2$ and $\zeta'_\ell(a_i^*)=a_i^*z_\ell$ for $3\ppq
i\ppq 6$. Since $\zeta_0-\zeta_0'$ is equal to $b_K^1(e_3+2e_5+3e_6+4e_7)$, and 
$\zeta_\ell-\zeta_\ell'=(\zeta_0-\zeta_0')\cupk z_\ell$ is also a coboundary, $\zeta_\ell$ and
$\zeta_\ell'$ represent the same cohomology class for $\ell\in\set{0,12,20,24}$. Moreover, as  in types
A, D, E$_6$ and E$_{7}$, the elements $\ov h_j$ form a basis of $\HK^2(A)$. 

The rest of the proof is the same as that of Proposition \ref{prop:basis HK E7}, based on Lemma
\ref{lemma:nsc wgt even E8} and the fact that the following cup products at the level of cochains are
all coboundaries: $\zeta_0\cupk\rho_{27}-\gamma_{28}$,   $\zeta_{12}\cupk\rho_{15}-\gamma_{28}$,
$\zeta_{24}\cupk\rho_{3}-\gamma_{28}$, $\zeta_{20}\cupk\rho_{7}-\gamma_{28}$,
$z_{24}\cupk\gamma_4-\gamma_{28}$, $z_{12}\cupk\gamma_{16}-\gamma_{28}$,
$z_{20}\cupk\gamma_8-\gamma_{28}$,  $z_{12}\cupk\rho_5-\rho_{17}$,
$z_{12}\cupk\gamma_6-\gamma_{18}$, $\zeta_0\cupk\rho_9-\gamma_{10}$,
whereas $\gamma_{28}$, $\gamma_{18}$, $\gamma_{10}$ and $\rho_{17}$ are not.
\end{proof}

Koszul homology follows using duality (Theorem \ref{complexduality}), as in Corollary \ref{cor:bases
hom type A}.

\subsubsection{Cup and cap products}

We now determine the cup products of the elements in the bases of the Koszul cohomology spaces given above.

\begin{Po} Let $A$ be a preprojective algebra of type E$_8$.
  Up to graded commutativity, the non zero cup products of elements in $\HK^{\bullet}(A)$ are:
 \begin{align*}
   z_0\cupk \ov  f&=\ov  f\text{ for all $\ov f\in\HK^{\bullet}(A)$}& z_{12}^2&=z_{24}\\
   z_{\ell}\cupk \ov  \zeta_0&=\ov  \zeta_\ell\text{ for $\ell\in\set{0,12,20,24}$}&z_{12}\cupk \ov \rho_3&=\ov \rho_{15}\\
   z_{24}\cupk\ov\rho_3&=\ov \rho_{27} & z_{12}\cupk\ov\rho_{15}&=\ov \rho_{27} \\
  z_{20}\cupk\ov\rho_7&=\ov \rho_{27} & z_{12}\cupk\ov\rho_{5}&=\ov \rho_{17} \\
 z_{12}\cupk\ov\rho_9&=-\ov\zeta_{20} &z_{12}\cupk\ov\zeta_{12}&=\ov\zeta_{24}\\\pi_i\cupk \ov h_i&=\ov \gamma_{28}\text{ for $i\in Q_0$} &
    z_{12}\cupk \ov \gamma_{4}&=\ov\gamma_{16}\\  z_{12}\cupk \ov \gamma_{16}&=\ov\gamma_{28}&
 z_{24}\cupk \ov \gamma_{4}&=\ov\gamma_{28}\\   z_{20}\cupk \ov \gamma_{8}&=\ov\gamma_{28}&
z_{12}\cupk \ov \gamma_{6}&=\ov\gamma_{18}\\ \ov\zeta_0\cupk \ov\rho_\ell&=-\ov\gamma_{\ell+1}\text{ for $\ell\in\set{3,7,15,27,9}$}&\ov\zeta_{12}\cupk\ov\rho_{3}&=\ov\gamma_{16}\\
\ov\zeta_{12}\cupk\ov\rho_{15}&=\ov\gamma_{28}&\ov\zeta_{20}\cupk\ov\rho_{7}&=\ov\gamma_{28}\\\ov\zeta_{24}\cupk\ov\rho_{3}&=\ov\gamma_{28}.  \end{align*}

\end{Po}

\begin{proof} Most of the cup-products are easy to compute, follow from Lemma \ref{lemma:nsc wgt even E8} or
vanish for weight reasons. The remaining ones are obtained as follows (at the level of cochains):
$z_{24}\cupk \rho_3=\rho_{27}+b_K^1([e_3\mapsto c_0 (c_3c_0)^6])$, we have $z_{12}\cupk \rho_3=\rho_{15}+b_K^1(h)$
  where $h$ is defined by  $h(a_2)=(c_0c_2)^3c_0a_2$,
  $h(a_3)=a_3c_0(c_3c_0)^3+a_3c_3(c_0c_3)^3$, $h(a_4)=a_4a_3(c_0c_3^2)^2a_3^*$ and
  $h(a_2^*)=a_2^*c_0(c_2c_0)^3+a_2^*c_2(c_0c_2)^3$, and finally $z_{12}\cupk
  \rho_9+\zeta_{20}=b_K(h')$ where $h'$ is defined by $h'(e_0)=a_0^*(c_3c_0)^4c_3a_0$, $h'(e_2)=2a_2^*(c_2c_0)^4c_2a_2$, $h'(e_3)=-2(c_0c_3)^5+2c_3^2(c_0c_3)^4+(c_3c_0)^4c_3^2$, $h'(e_4)=-a_3(c_0c_3)^4c_0a_3^*+a_3c_3(c_3c_0)^4a_3^*-a_3(c_0c_3)^4c_3a_3^*$ and $h'(e_5)=-a_4a_3(c_0c_3)^4a_3^*a_4^*$.
\end{proof}

The cap products follow using duality, as in Corollary \ref{cor:cap type A}.

\subsubsection{Higher Koszul cohomology and homology}

As in types A, D, E$_6$ and E$_7$, the differential $\partial _\smile^1$ sends $z_\ell$ to $2\zeta_\ell$ for
$\ell\in\set{0,8,12}$ and the differential $\partial _\smile^2$ is zero except when $\car(\ff)=5$ where $\partial _\smile^2(\ov\rho_9)=2\ov\gamma_{10}$.

 We then have the following higher Koszul cohomology.

\begin{Po}  \label{hkcE8} Let $A$ be a preprojective algebra of type E$_8$.

  If $\car(\ff)=2$, then $\HK^{\bullet}_{hi}(A)=\HK^{\bullet}(A).$

 If $\car(\ff)=3$, then
  \begin{align*}
    \HK^0_{hi}(A)&=\HK^0(A)_{28}\text{ and  the $\pi_i$ for $i\in Q_0$ form a basis}\\
    \HK^1_{hi}(A)&=\spn{[\ov\rho_5],[\ov\rho_{17}]}\\
    \HK^2_{hi}(A)&=\HK^2(A)_0 \text{ and the $[\ov h_i]$ for $i\in Q_0$ form a basis}\\
    \HK^p_{hi}(A)&=0\text{ if $p\neq 0$ and $p\neq 2$.}
  \end{align*}

 If $\car(\ff)\not\in\set{2,3}$, then
  \begin{align*}
    \HK^0_{hi}(A)&=\HK^0(A)_{28}\text{  has dimension $8$ and is spanned by  the $\pi_i$ for $i\in Q_0$ }\\
    \HK^2_{hi}(A)&=\HK^2(A)_0 \text{ has dimension $8$ and is spanned by the $[\ov h_i]$ for $i\in Q_0$}\\
    \HK^p_{hi}(A)&=0\text{ if $p\neq 0$ and $p\neq 2$.}
  \end{align*}

\end{Po}

Higher Koszul homology follows from Theorem \ref{complexduality}.

\subsection{Comparison of Koszul and Hochschild (co)homology for preprojective algebras of type ADE}\label{subsec:comparison hochschild}

Let $A$ be a preprojective algebra over a Dynkin graph of type ADE.
Schofield constructed a minimal projective resolution $(P^\bullet,\partial ^\bullet)$ of $A$
as a bimodule over itself, that is periodic (of period at most $6$), which was described in
\cite{es:third,eu:product}. Following Proposition \ref{1degree2}, the embedding $H(\iota^*)_2$ sends the Hochschild cohomology
class of an element in $\ker(\partial ^3\circ -)$ to its Koszul cohomology class, and the surjection
$H(\tilde{\iota})_2$ induces an isomorphism between $\HH_{2}(A)$  and
$\HK_2(A)/\im(\id_A\otimes \partial ^3)$.

We first transport the maps $\partial ^3\circ -$ and $\id_A\otimes \partial ^3$ via  the natural isomorphisms   $A\otimes_{A^e}(Ae_j\otimes e_iA)\rightarrow
  e_iAe_j$ that sends $\lambda\otimes (a\otimes b)$ to $b\lambda a$ and  $\Hom_{A^e}(Ae_i\otimes e_jA,A)\cong
  e_iAe_j$ that sends $f$ to $f(e_i\otimes e_j)$.

The  associative non degenerate bilinear form on the selfinjective preprojective $A$ can be defined
as follows, see \cite{es:third} and \cite[Proposition 3.15]{Z}: 
 let $B$ be a basis of $A$ consisting of homogeneous elements, that contains the idempotents $e_i$, $i\in Q_0$ and a basis $\set{\pi _i\,;\,i\in Q_0}$ of the socle of $A$, and  such that each $v\in B$ belongs to $e_jAe_i$ for some $i,j$ in $Q_0.$  
 Then if $x\in Ae_i$,  $(y,x)$ is the coefficient of $\pi_i$ in the expression of $yx$ as a linear combination of elements in $B.$
The Nakayama automorphism $\nk$ of $A$ satisfies $(y,x)=(\nk(x),y)$ for all $x,y$ in $A$, and
induces a permutation of the indices, the Nakayama permutation $\np$, that is, a permutation of  $
Q_0$ such that $\tp(Ae_i)\cong \soc(Ae_{\np(i)})$, characterised by  $\nk(e_i)=e_{\np(i)}$.

Let $\widehat{B}$ be the dual basis of $B$ with respect to the non degenerate form $(-,-)$, so that $(\widehat{w},v)=\delta_{vw}$ for all $v,w$ in $B$. In particular, if $v\in Be_i$, the coefficient of $\pi_i$ in  $\widehat{v}v$ is $1.$ Note that $v\in e_jBe_i$ if and only if $\widehat{v}\in e_{\np(i)}\widehat{B}e_j.$

Then  the maps $\partial ^3\circ -$ and
  $\id_A\otimes_{A^e}\partial ^3$ become respectively
\begin{align*}
\delta^3\colon &\bigoplus_{i\in Q_0}e_iAe_{i}\rightarrow \bigoplus_{i\in
                 Q_0}e_iAe_{\np(i)}\text{ defined by }y\mapsto \sum_{x\in\B}\widehat{x}yx\\
\text{and }\delta_3\colon &\bigoplus_{i\in Q_0}e_iAe_{\np(i)}\rightarrow \bigoplus_{i\in Q_0}e_iAe_{i}\text{ defined by }y\mapsto                                                                                      \sum_{x\in\B}xy\widehat  x.
\end{align*}

It then follows as in \cite[Proposition 3.2.25]{eusched:cyfrob} that $\im\delta_3$ is the span of
the $\delta_3(e_i)$ such that $\np(i)=i$, and that for such an $i$ we have
$\delta_3(e_i)=\sum_{\substack j\in
Q_0\\\np(j)=j}\tr(\nk_{|e_jAe_i})\pi_j$. Moreover, the matrix whose coefficients are the
$\tr(\nk_{|e_jAe_i})$ is either easy to compute or given in \cite{eu:product} for types D and
E$_7$. It is also known from  \cite{eusched:cyfrob} that the set of elements of weight $0$ in  $\ker\delta^3$ identifies with
the kernel of the Cartan matrix of $A$. Moreover, for any element of positive weight $a\in A$, we
have $\delta^3(a)=0$. Therefore $\HH^2(A)$ is obtained by taking all the elements of positive
weight in $\HK^2(A)$ and adding the kernel of the Cartan matrix.

We shall use this as well as the dimensions of the Hochschild and Koszul (co)homology spaces to
compare $\HH^2(A)$ with $\HK^2(A)$ and  $\HH_2(A)$ with $\HK_2(A)$ in each case. In particular, the injection $\HH^2(A)\rightarrow  \HK^2(A)$ is not surjective except in type E$_8$ with $\car(\ff)= 2$.

\subsubsection{Comparison of the second Koszul and Hochschild cohomology groups}

In type A, the space $\HH^2(A)$ was completely described by Erdmann and Snashall in \cite{es:first}, and they
proved that $\dim\HH^2(A)=n-\ma-1$ and gave  a basis $\set{\tilde h_i;0\ppq i\ppq n-\ma-2}$ with
$\tilde{h}_i=\ov h_i+\ov h_{n-1-i}.$  The morphism of complexes $\iota^*_2$ sends $\tilde h_i$ to $h_i+h_{n-1-i}$, and this describes the injection $\HH^2(A)\rightarrow  \HK^2(A)$.

In type D,  if $\car(\ff)\neq 2$ and $n$ is even, there is nothing to do since
$\HH^2(A)=0.$ If  $\car(\ff)\neq 2$ and $n$ is odd, then $\dim\HH^2(A)=1$, the basis given in
\cite{eu:product} for $\HH^2(\Lambda_\cc)$  also gives a basis of $\HH^2(A)$, and it is the
cohomology class of the map $\psi_0$ defined by $\psi_0(\sigma_0)=e_0$ and $\psi_0(\sigma_1)=-e_1$. The embedding $\HH^2(A)\rightarrow  \HK^2(A)$ is therefore given by $\psi_0\mapsto h_0-h_1.$

Now assume that $\car(\ff)=2$. Then $\dim\HH^2(A)=n+\md-2.$ As we explained above, a  basis of
$\HH^2(A)$ may be obtained from a basis of the set of elements of positive coefficient weight in $\HK^2(A)$ to which we add elements obtained by determining a basis of the kernel of  the Cartan matrix of $A$. It follows that a basis of $\HH^2(A)$ is given by 
\begin{align*}
&\set{\ov \gamma_\ell;1\ppq \ell\ppq m}\cup \set{\ov \varphi_i;2\ppq i\ppq n-1}\quad\text{if $n$ is even}\\
&\set{\ov \gamma_\ell;1\ppq \ell\ppq m}\cup\set{\ov \psi_0}\cup \set{\psi_i;3\ppq i\ppq n-1}\quad\text{if $n$ is odd}
\end{align*} where $\psi_{2p+1}(\sigma_{2p+1})=e_{2p+1}$, $\psi_{2p+2}(\sigma_{2p+2})=e_{2p+2}$ and
$\psi_{2p+2}(\sigma_2)=e_2$, $\varphi_{2p}(\sigma_{2p})=e_{2p}$,
$\varphi_{2p+1}(\sigma_{2p+3})=e_{2p+3}$ and $\varphi_{2p+1}(\sigma_{3})=e_{3}$, for $p\pgq 1.$
Therefore, the embedding  $\HH^2(A)\rightarrow  \HK^2(A)$ fixes the $\ov \gamma_\ell$ and sends
$\ov \varphi_{2p}$ to $\ov h_{2p}$ and $\ov \varphi_{2p+3}$ to $\ov h_3+\ov h_{2p+3}$ when $n$ is even, and $\ov \psi_0$ to
$\ov h_0+\ov h_1$, $\ov \psi_{2p+1}$ to $\ov h_{2p+1}$ and $\ov \psi_{2p+2}$ to $\ov h_2+\ov
h_{2p+2}$ when $n$ is odd. 

In type E$_6$, 
 the Cartan matrix is equivalent, through row operations, to 
\[
\left(\begin{smallmatrix}
  2&0&0&0&0&0\\0&1&0&0&0&1\\0&0&1&0&1&0\\0&0&0&2&0&0
\end{smallmatrix}\right)
\] Let $\varphi_0$, $\varphi_1$, $\varphi_2$ and $\varphi_3$ be the maps in $ \Hom_{k^e}(R,A)$
defined by $\varphi_0(\sigma_0)=e_0$,
$\varphi_1(\sigma_1)=e_1$, $\varphi_1(\sigma_5)=-e_5$, $\varphi_2(\sigma_2)=e_2$,
$\varphi_2(\sigma_4)=-e_4$ and $\varphi_3(\sigma_3)=e_3$. Then $\HH^2(A)$ is the subspace of $\HK^2(A)$ spanned by 
\begin{align*}
& \ov \varphi_1, \ov \varphi_2\text{ if }\car(\ff)\not\in\set{2,3}\\
& \ov \varphi_0,\ov \varphi_1,\ov \varphi_2,\ov \varphi_3,\ov \gamma_4\text{ if }\car(\ff)=2\\
& \ov \varphi_1, \ov \varphi_2, \ov \gamma_6\text{ if }\car(\ff)=3.
\end{align*}

In type E$_7$, if $\car(\ff)\neq 2$, we have $\dim \HH^2(A)=\dim\HK^2(A)_{>0}$, so that $\HH^2(A)$ is
precisely the subspace of $\HK^2(A)$ of elements of positive weight (which is zero unless $\car(\ff)\in\set{2,3})$. If $\car(\ff)=2$, then the
Cartan matrix of $A$ is equivalent, through row operations, to the matrix $
\left(\begin{smallmatrix}
  1&0&0&0&1&0&1
\end{smallmatrix}\right)
$ so that $\HH^2(A)$ is the subspace of $\HK^2(A)$ spanned by $\ov \gamma_4$, $\ov \gamma_8$, $\ov \gamma_{16}$,
$\ov h_1$, $\ov h_2$, $\ov h_3$, $\ov h_5$, $\ov h_0+\ov h_4$ and $\ov h_0+\ov h_6$.

In type E$_8$, we only need to look at dimensions. Indeed, if
$\car(\ff)=2$, then $\dim\HH^2(A)=\dim\HK^2(A)$ so that $\HH^2(A)\cong \HK^2(A)$, and if
$\car(\ff)\neq 2$ then $\dim\HH^2(A)=\dim\HK^2(A)_{>0}$ so that $\HH^2(A)$ is
precisely the subspace of $\HK^2(A)$ of elements of positive weight (which is zero unless $\car(\ff)\in\set{2,3,5})$.

\subsubsection{Comparison of the second Koszul and Hochschild homology groups}

In type A, 
if $n$ is even or if $n$ is odd and  $\car(\ff)$ divides $ (\ma+1)$, we have seen that $\dim\HK_2(A)=\dim\HH_2(A)$ and therefore $\HK_2(A)\cong \HH_2(A)$. Now assume that $n$ is odd and $\car(\ff)\nmid (\ma+1)$. Then $\dim\HK_2(A)=\dim\HH_2(A)+1$. We must determine the image of the map $\delta_3$ given in Subsection \ref{subsec:comparison hochschild}.

The Nakayama permutation $\np$ has precisely one fixed point, which is $\ma$. The matrix of $\nk$ restricted to $e_\ma Ae_\ma$ is the identity matrix $I_{\ma +1}$.
Consequently, $\delta_3(e_{\ma})=(\ma+1)\pi_\ma$ spans the image of $\delta_3.$
  Via the isomorphism $\bigoplus_{i\in Q_0}e_iAe_i\cong A\ot_{k^e}R$, $\pi_\ma$ corresponds to
  $\cz_\ma$ so that we have $\HH_2(A)\cong \HK_2(A)/\spn{ \cz_\ma}$.

In type D, if $n$ is odd and $\car(\ff)=2$, then we know  that $\dim\HK_2(A)=\dim\HH_2(A)$ and therefore $\HK_2(A)\cong \HH_2(A)$. In the other cases, we need to  determine the image of the map $\delta_3$. We should note that the Nakayama automorphism in \cite{eu:product} differs from $\nu $ by composition with the inner automorphism given by the invertible element $u=-e_0+\sum_{i=1}^{n-1}(-1)^ie_i$ (after changing the labelling of vertices and arrows so that they are the same as ours).

If $n$ is odd and $\car(\ff)\neq 2$, the fixed points of the Nakayama permutation are the integers
$i$ with $2\ppq i\ppq n-1.$ The matrix $(n-2)\times (n-2)$ matrix $H^\nk$ whose  $(i,j)$-coefficient is the trace of $\nk$
restricted to $e_jAe_i$ was given in \cite[paragraph 11.2.3]{eu:product}; we need only change the signs of coefficients $(i,j)$ with $i$ or $j$ (but not both) equal to $0$ or odd,  so that
$\tr\left(\nk_{|e_jAe_i}\right)=2$ if $i$ and $j$ are even and is equal to $0$ otherwise.
 Therefore for all $i$ fixed by $\np$, we have $\delta_3(e_i)=\sum_{p=1}^{\md+1}2\pi_{2p}.$ Using the isomorphism 
 $\bigoplus_{i\in Q_0}e_iAe_i\cong A\ot_{k^e}R$, we obtain  $\HH_2(A)\cong \HK_2(A)/\spn{\sum_{p=1}^{\md+1} \cp_{2p}}$.

If $n$ is even, all the integers $i$ with $0\ppq i\ppq n-1$ are fixed points of $\np.$  The matrix
$H^\nk$ was given in  \cite[paragraph 11.2.2]{eu:product} and the same adaptations as in the case $n$ odd gives the $n\times n$ symmetric matrix 
\[ H^\nu =
  \left(\begin{smallmatrix}
    m+1&-m&0&1&0&1&\cdots&0&1\\-m&m+1&0&1&0&1&\cdots&0&1\\0&0&0&0&0&0&\cdots&0&0\\1&1&0&2&0&2&\cdots&0&2\\0&0&0&0&0&0&\cdots&0&0\\1&1&0&2&0&2&\cdots&0&2\\\vdots&\vdots&\vdots&\vdots&\vdots&\vdots&\ddots&\vdots&\vdots\\0&0&0&0&0&0&\cdots&0&0\\1&1&0&2&0&2&\cdots&0&2
  \end{smallmatrix}\right)
\text{ which is column equivalent to }\left(\begin{smallmatrix} \md+1&1-n\\-\md&n-1\\0&0\\1&0\\0&0\\\vdots&\vdots\\0&0\\1&0
 \end{smallmatrix}\right).\]
It follows that, if $n$ is even, 
\[ \HH_2(A)\cong
\begin{cases}
\HK_2(A)/\spn{(\md+1) \cp_0-\md \cp_1+\sum_{p=1}^{\md} \cp_{2p+1}}&\text{if $\car(\ff)|(n-1)$}\\
\HK_2(A)/\spn{ \cp_0- \cp_1; \cp_0+\sum_{p=1}^{\md} \cp_{2p+1}}&\text{if $\car(\ff)\nmid (n-1)$}\\
\end{cases}
 \]

In types E$_6$ and E$_8$, since they have the same dimensions, the homology spaces $\HH_2(A)$ and $\HK_2(A)$ are isomorphic.

\sloppy Finally, in type E$_7$, if $\car(\ff)=3$, then $\HH_2(A)\cong \HK_2(A)$. Now assume that $\car(\ff)\neq 3$. The matrix
$H^\nk$ was given in \cite{eu:product}; here again, the Nakayama automorphism in \cite{eu:product} differs from $\nu $ by composing with the inner automorphism associated with the element $-e_0+e_1-e_2+e_3+e_4+e_5+e_6$ (after change of labelling) and therefore the non zero rows of the matrix $H^\nu $  are those corresponding to vertices
$0,$ $4$ and $6$ and are all equal to $
\left(\begin{smallmatrix}
  3&0&0&0&3&0&3
\end{smallmatrix}\right). $
It follows that $\HH_2(A)\cong\HK_2(A)/\spn{ \cp_0+ \cp_4+ \cp_6}$.

\subsection{A minimal complete list of cohomological invariants}

\begin{Te} \label{mincomp}
Let $A$ be the preprojective algebra of a Dynkin graph $\Delta$ over $\ff$. Assume that $A$ has type either $\mathrm{A}_n$ with $n\pgq 3$, or $\mathrm{D}_n$ with $n\pgq 4$, or $\mathrm{E}_n$ with $n=6,\,7,\,8$. Let $A'$ be a preprojective algebra of type $\mathrm{ADE}$, where the integer $n'$ concerning $A'$ is subjected to the same assumptions. Denote by $(d_p)$ the equality $\dim \HK^p_{hi}(A)=\dim \HK^p_{hi}(A')$. If $(d_p)$ holds for $p=0,1$ and $2$, then $n=n'$, and $A$ and $A'$ have the same type. The conclusion of this implication does not hold if $(d_2)$ is removed from the assumption.
\end{Te}
\begin{proof}
We apply the results contained in Propositions \ref{hkcA}, \ref{hkcD}, \ref{hkcE6}, \ref{hkcE7} and
\ref{hkcE8}. The implication is a consequence of the following items.
\begin{enumerate}[(1)]
\item \label{case:AD} Assume that $A$ and $A'$ have types A or D. If $\car \ff \neq 2$, then $n=n'$
  by $(d_2)$, and $A$ and $A'$ have the same type by $(d_0)$. If $\car \ff =2$ and $A$ and if $A'$ both
  have type A, then $n=n'$  by $(d_2)$. If $\car \ff =2$,  $A$ is of type $A_n$ and $A'$ is of type
  $D_{n'}$, then the sum of $(d_0)$ and $(d_1)$ shows that $n=2n'+m'_D -2$ or $-3$ which contradicts
  $(d_2)$: $n=n'+m'_D$. If $\car \ff =2$ and if $A$ and $A'$ both have type D, then $n=n'$ by $(d_1)$.
\item \label{case:E} If $A$ and $A'$ have type E, then $n=n'$ by $(d_0)$.
\item \label{case:AE} If $A$ is of type $A_n$ and $A'$ of type $E_6$, then $(d_2)$ implies either
  that $n=6$ if $\car \ff \neq 2,3$ or that $n=7$ if $\car \ff=2,3$, but each case is excluded by $(d_0)$. Similarly when $A'$ is of type $E_7$ and $E_8$.
\item \label{case DE} If $A$ is of type $D_n$ and $A'$ is of type $E_6$, then $(d_2)$ implies one of
  the three following cases: $n=6$ if $\car \ff \neq 2, 3$,  $n=7$ if $\car \ff =3$ or $n+m_D= 7$ if
  $\car \ff =2$, but each case is excluded by $(d_0)$. Similarly when $A'$ is of type $E_7$ and
  $E_8$ (we also use $(d_1)$ for $E_8$).  
\end{enumerate}

Let us show that we cannot remove assumption $(d_2)$. It is clear if $\car \ff \neq 2$ because when
$A$ is of type $A_3$ and $A'$ is of type $A_5$, both $(d_0)$ and $(d_1)$ hold. If $\car \ff =2$, we
check that when $A$ is of type $A_9$ and $A'$ is of type $E_6$ then $(d_0)$ and $(d_1)$ hold. 
\end{proof}

It is obvious that  $\HK^{\bullet}_{hi}(A)\cong \ff$ if $A$ is of type $A_1$ and it is easy to check that  $\HK^{\bullet}_{hi}(A)=
0$ when $A$ is of type $A_2$, therefore  we have obtained a minimal complete list of cohomological invariants for all the ADE preprojective algebras.

Another direct application of our computations is the following. If $A$ is as in Theorem \ref{mincomp} and if $\car \ff \neq 2$, the product of the algebra $\HK^{\bullet}_{hi}(A)$ is identically zero. If $\car \ff =2$, $\HK^{\bullet}_{hi}(A)=\HK^{\bullet}(A)$ is a unital algebra whose product is fully described in our results. 
\begin{Rm} 
In the one vertex case, the higher Koszul homology and cohomology play an essential role in a specific formulation of a Koszul Poincar\'e Lemma (a Koszul Poincar\'e duality), see Conjectures 6.5 and 7.2 in~\cite{bls:kocal}. For this  reason, we have formulated Theorem \ref{mincomp} in terms of the higher Koszul cohomology. An analogous statement with Koszul cohomology is also true and follows in the same way from our results, but in this case the minimality depends on the characteristic.
\end{Rm}

\affiliationone{%
  Roland Berger\\
   Univ Lyon, UJM-Saint-\'Etienne, CNRS UMR 5208, Institut Camille Jordan, F-42023, Saint-\'Etienne\\ France
   \email{Roland.Berger@univ-st-etienne.fr}}
\affiliationtwo{%
   Rachel Taillefer\\
   Universit\'e Clermont Auvergne, CNRS, LMBP, F-63000 Clermont-Ferrand\\
   France
   \email{Rachel.Taillefer@uca.fr}}

\end{document}